\documentclass[12pt,reqno]{amsart}
\usepackage{amsthm,amsfonts,amssymb,euscript}

\newcommand{\bea}{\begin{eqnarray}}
\newcommand{\eea}{\end{eqnarray}}
\def\beaa{\begin{eqnarray*}}
\def\eeaa{\end{eqnarray*}}
\def\ba{\begin{array}}
\def\ea{\end{array}}
\def\be#1{\begin{equation} \label{#1}}
\def \eeq{\end{equation}}
\def\lab{\label}

\newcommand{\nn}{\nonumber}
\newcommand\pa{\partial}

\def\a{{\alpha}}

\def\b{{\beta}}
\def\be{{\beta}}
\def\ga{\gamma}

\def\de{\delta}
\def\De{\Delta}
\def\ep{\epsilon}
\def\eps{\epsilon}

\def\la{\lambda}

\def\si{\sigma}
\def\Si{\Sigma}
\def\om{\omega}

\def\nab{\nabla}

\def\pr{{\partial}}

\def\les{\lesssim}
\def\c{\cdot}

\def\CC{{\mathcal C}}
\def\MM{{\mathcal M}}

\def\EE{{\mathcal E}}
\def\HH{{\mathcal H}}
\def\LL{{\mathcal L}}

\def\SS{{\mathcal S}}

\def\PP{{\mathcal P}}

\def\A{{\bf A}}

\def\D{{\bf D}}
\def\F{{\bf F}}

\def\J{{\bf J}}

\def\L{{\bf L}}
\def\O{{\bf O}}

\def\R{{\bf R}}

\def\U{{\bf U}}

\def\g{{\bf g}}
\def\m{{\bf m}}

\def\SSS{{\mathbb S}}
\def\RRR{{\mathbb R}}

\def\f12{{\frac 1 2}}

\def\dual{{\,\,^*}}
\def\div{{\mbox div\,}}
\def\curl{{\mbox curl\,}}
\def\lot{\mbox{ l.o.t.}}

\def\Lb{{\,\underline{L}}}

\def\tr{\mbox{tr}}

\def\f{\widetilde{f}}

\def\lap{\De}

\def\Da{{^{(\A)}\hskip-.15 pc \D}}

\newcommand{\nabb}{{\bf \nab} \mkern-13mu /\,}

\providecommand{\norm}[1]{\lVert#1\rVert}
\newcommand{\lsit}[1]{L^{\infty}_tL^{#1}(\Si_t)}
\newcommand{\lt}[1]{L^{#1}(\Si_t)}
\newcommand{\lsitt}[2]{L^{#1}_tL^{#2}(\Si_t)}
\newcommand{\lu}[1]{L^\infty_uL^{#1}(\mathcal{H}_u)}
\newcommand{\luom}[1]{L^\infty_{\uom}L^{#1}(\mathcal{H}_{\uom})}

\def\uom{{ \, \, ^{\om}  u}}
\def\Lom{ {\,\,  ^{\om}  L}}
\def\Nom{ {\,\,  ^{\om}  N}}

\def\bom{{ \, \, ^{\om}  b}}
\def\eom{{ \, \, ^{\om}  e}}
\def\uoms{{ \, \, ^{\om,s}  u}}
\def\prb{\boldsymbol{\pr}}

\newtheorem{theorem}{Theorem}[section]
\newtheorem{lemma}[theorem]{Lemma}
\newtheorem{proposition}[theorem]{Proposition}
\newtheorem{corollary}[theorem]{Corollary}
\newtheorem{definition}[theorem]{Definition}
\newtheorem{remark}[theorem]{Remark}

\numberwithin{equation}{section}

\begin{document}

\author{Sergiu Klainerman}
\address{Department of Mathematics, Princeton University,
 Princeton NJ 08544}
\email{ seri@math.princeton.edu}
\title[The Bounded $L^2$ Curvature Conjecture]{The Bounded $L^2$ Curvature Conjecture}

\author{Igor Rodnianski}
\address{Department of Mathematics, Princeton University, 
Princeton NJ 08544}
\email{ irod@math.princeton.edu}
\author{Jeremie Szeftel}
\address{DMA, Ecole Normale Superieure, Paris 75005}
\email{Jeremie.Szeftel@ens.fr}
\thanks{The third author is supported by the project ERC 291214 BLOWDISOL}
\vspace{-0.3in}
\maketitle
\begin{abstract}
 This is the main  paper    in a sequence  in which we 
   give a complete  proof of the   bounded $L^2$ curvature conjecture.
     More precisely we show that  the   time
    of existence   of a classical solution to    the Einstein-vacuum
     equations  depends   only on the  $L^2$-norm
      of the  curvature    and  a lower bound on the volume radius   of the corresponding  initial data set.  We note that  though  the result is not  optimal   with respect to the    scaling  of the Einstein equations, it is  nevertheless critical   with respect  to   its causal   geometry. Indeed, $L^2 $ bounds on the curvature    is     the minimum  requirement   necessary to  obtain lower bounds on   the radius of  injectivity of  causal  boundaries.   We note also  that, while   the first   nontrivial improvements  for well posedness     for quasilinear  hyperbolic systems in spacetime dimensions greater than $1+1$ (based on Strichartz estimates)  were obtained in   \cite{Ba-Ch1}, \cite{Ba-Ch2}, \cite{Ta1}, \cite{Ta2},   \cite{Kl-R1}  and optimized in      \cite{Kl-R2},  \cite{Sm-Ta},    the  result   we present  here   is the first   in which the  full    structure of the quasilinear hyperbolic   system, not just its principal part,  plays  a  crucial  role.  
      
To achieve our goals  we  recast the Einstein vacuum equations as a   quasilinear  $so(3,1)$-valued  Yang-Mills    theory and introduce a Coulomb type  gauge condition in   which  the equations    exhibit  a    specific new type of   \textit{null   structure} compatible with  the quasilinear, covariant nature of the equations.  To prove the  conjecture 
      we formulate and establish    bilinear  and trilinear estimates on  rough backgrounds which allow us to make use of that crucial structure. These require  a careful construction and control  of parametrices  including   $L^2$ error
       bounds    which   is  carried      out in \cite{param1}-\cite{param4}, as well as a proof of sharp Strichartz estimates for the wave equation on a rough background which is carried out in \cite{bil2}.  
       It is at this level     that      our  problem is critical.        Indeed,  any known notion of a parametrix  relies in an essential way on 
       the eikonal equation, and       our space-time possesses, barely,  the minimal regularity needed to make sense of   its  solutions. 
           \end{abstract}
     
\section{Introduction}
      
      This is the main   in a sequence of papers in which we 
   give a complete  proof of the   bounded $L^2$ curvature conjecture.
   According to  the conjecture  
     the   time    of existence   of a classical solution to    the Einstein-vacuum
     equations  depends   only on the  $L^2$-norm
      of the  curvature  and a   lower bound on the volume radius  
       of the corresponding  initial data set.  At a deep level   the  $L^2$ curvature conjecture  
   concerns   the  relationship
between the curvature tensor  and the causal  geometry of an Einstein vacuum  
space-time.  Thus,  though  the result is not  optimal   with respect to the    scaling  of the Einstein equations, it is  nevertheless critical  with respect    to      its      causal   properties.    More precisely,      $L^2 $   curvature       bounds    are strictly    necessary to  obtain lower bounds on   the radius of  injectivity  of  causal  boundaries.  These lower bounds  turn out  to be  crucial  for the construction of parametrices   and derivation of bilinear and trilinear spacetime estimates  for solutions  to  scalar wave equations.  We note also  that, while   the first   nontrivial improvements  for well posedness     for quasilinear  hyperbolic systems in spacetime dimensions greater than $1+1$ (based on Strichartz estimates)  were obtained in   \cite{Ba-Ch1}, \cite{Ba-Ch2}, \cite{Ta1}, \cite{Ta2},   \cite{Kl-R1}  and optimized in      \cite{Kl-R2},  \cite{Sm-Ta},    the  result   we present  here   is the first   in which the  full    structure of the quasilinear hyperbolic   system, not just its principal part,  plays  a  crucial  role.      
 
 \subsection{Initial value problem}
  
We consider the Einstein vacuum equations (EVE),
\begin{equation}\lab{EVE}
{\bf Ric}_{\alpha\beta}=0
\end{equation}
where ${\bf Ric}_{\alpha\beta}$
denotes the  Ricci curvature tensor  of  a four dimensional Lorentzian space time  $(\mathcal{M},\,  {\bf g})$. 
An  initial data  set for \eqref{EVE}  consists of a three dimensional   $3$-surface 
$\Si_0$   together with a    Riemannian  metric $g$ and a symmetric  $2$-tensor $k$  verifying the constraint equations,
\begin{equation}\lab{const}
\left\{\begin{array}{l}
\nabla^j k_{ij}-\nabla_i \textrm{tr}k=0,\\
 {R_{scal}}-|k|^2+(\textrm{tr}k)^2=0, 
\end{array}\right.
\end{equation}
where the covariant derivative $\nabla$ is defined with respect to the metric $g$, $R_{scal}$ is the scalar curvature of $g$, and $\textrm{tr}k$ is the trace of $k$ with respect to the metric $g$.  In this  work  we restrict ourselves 
 to asymptotically flat  initial data sets with one end. 
 For a given initial data set the Cauchy problem  consists in finding a metric ${\bf g}$ satisfying \eqref{EVE}   and an embedding of $\Si_0$  in $\MM$ such that the metric induced by ${\bf g}$ on $\Si_0$ coincides with $g$ and the 2-tensor $k$ is the second fundamental form of the hypersurface $\Si_0\subset \MM$. 
 The  first   local existence and uniqueness   result for (EVE)  was   established
by Y.C. Bruhat,  see \cite{Br},  with the help of 
 wave coordinates  which
 allowed her to cast
the Einstein vacuum  equations in the form of a system of nonlinear wave 
equations to which one can apply\footnote{The original proof in 
\cite{Br} relied  on  representation formulas, following an approach pioneered by Sobolev,   see \cite{Sob}.  }   the standard theory of nonlinear  hyperbolic systems.  The optimal,
classical\footnote{Based only on energy estimates and classical Sobolev inequalities.} result states the following,

\begin{theorem}[Classical local existence \cite{FM} \cite{HKM}]
\label{thm:Bruhat} Let $(\Si_0, g, k)$ be an initial data set
for the Einstein vacuum equations \eqref{EVE}. Assume that $\Si_{0}$ can
be covered by a locally finite system of coordinate charts,
 related to each other by $C^1$ diffeomorphisms, such that
$(g,\, k )\in H^s_{loc}(\Si_0)\times H^{s-1}_{loc}(\Si_0)$
with $s>\frac{5}{2}$. Then there exists a unique\footnote{The original proof in \cite{FM}, \cite{HKM} actually requires  one more derivative for the uniqueness. The fact that uniqueness holds at the same level of regularity than the existence has been obtained in \cite{PlRo}} (up to an isometry)
globally hyperbolic   development
$(\MM, \g)$, verifying   \eqref{EVE},   for which
$\Si_0$ is a Cauchy hypersurface\footnote{That is any past directed, in-extendable  causal curve in $\MM$ intersects $\Si_0$.}.
\end{theorem}

\subsection{Bounded $L^2$ curvature conjecture}
The classical   exponents $s>5/2$ are  clearly not   optimal. 
 By  straightforward  scaling considerations one might  expect to make sense of the initial
value problem for    $s\ge s_c=3/2$, with $s_c$ the natural scaling exponent   for  $L^2$ based Sobolev norms.  Note that   for  $s=s_c=3/2$ a local in time  existence result, for sufficiently small  data,  would be  equivalent to  a global result.  More precisely any  smooth initial data, small in the corresponding  critical norm,  would be  globally smooth. 
  Such a     well-posedness   (WP)  result   would  be thus  comparable with  the so called $\epsilon$- regularity results for   nonlinear elliptic and parabolic problems, which play such a fundamental     role in  the global  regularity properties
  of general solutions.   For quasilinear  hyperbolic problems     critical  WP  results have only been established in the case
  of  $1+1$ dimensional systems, or  spherically symmetric solutions  of higher dimensional problems, 
  in which case  the  $L^2$- Sobolev norms  can be replaced  by bounded variation (BV) type norms\footnote{Recall  that the entire theory of  shock waves  for  1+1 systems of conservation laws  is based on  BV norms, which are \textit{critical} with respect to the scaling of the equations.  Note also  that these   BV norms are not, typically,  conserved   and   that Glimm's famous  existence  result  \cite{Gl}
   can be interpreted as   a global well posedness result for   initial data with small   BV norms.  }.     A  particularly important     example of this type is the critical  BV    well-posedness  result  established by  Christodoulou    for spherically symmetric solutions of the  Einstein equations   coupled with a scalar field, see  \cite{Ch1}. The result  played a  crucial  role  in his celebrated work on the Weak Cosmic Censorship  for the same model, see  \cite{Ch2}.   As well known, unfortunately,   the   BV-norms   are completely   inadequate in higher dimensions;  the only norms which  can propagate   the regularity properties    of the data are necessarily    $L^2$ based.

The quest  for   optimal well-posedness   in  higher dimensions   has been  one of the major themes   in non-linear   hyperbolic   PDE's   in the last twenty years.    Major   advances  have been made in the particular    case of 
 semi-linear  wave equations.    In the case  of geometric   wave  equations such as Wave Maps and Yang-Mills,  which possess a well understood  null structure,
  well-posedness    holds true  for all  exponents larger 
 than the corresponding critical exponent.   For example, 
in the case of    Wave Maps  defined  from the Minkowski space $\RRR^{n+1}$ to a complete Riemannian manifold,    
 the critical  scaling exponents is $s_c=n/2$  and   well-posedness is known to hold all the way down  to $s_c$  for all  dimensions $n\ge 2$.  This  critical    well-posedness result, for $s=n/2$, plays a fundamental   role  in the recent,  large data,       global  results of  \cite{Tao}, \cite{St-Ta1}, \cite{St-Ta2} and \cite{Kr-S} for $2+1$ dimensional wave maps. \\
 
  The   role played  by   critical  exponents   for quasi-linear   equations  is  much less understood.
  The first well  posedness   results, on any     (higher dimensional)  quasilinear  hyperbolic system,  which go  beyond the classical Sobolev exponents,      obtained  in    \cite{Ba-Ch1}, \cite{Ba-Ch2}, and \cite{Ta1}, \cite{Ta2} and   \cite{Kl-R1},   do not take into account the 
  specific (null)  structure   of the   equations.  Yet the presence of such   structure   was  crucial    in the derivation of the optimal  results mentioned above,   for geometric  semilinear equations.  In the case of the Einstein equations it is not at all clear what such structure should be, if there is one at all.  Indeed,  the only specific structural condition, known for (EVE),    discovered    in \cite{wNC} under the name  of  the \textit{weak null condition},   is    not at all adequate   for   improved  well posedness results, see remark  \ref{rem:weekN}.  It is known however, see \cite{L}, that without such a structure  one cannot have well posedeness   for  exponents\footnote{Note that the dimension here is $n=3$.}  $s\le 2$. Yet (EVE)  are of  fundamental  importance   and as such    it is not unreasonable  to expect that such a structure must exist.

   Even assuming     such a structure,  a result of well-posedness  for the Einstein equations at, or near,  the critical regularity  $s_c=3/2$
is not only completely out of reach  but may in fact be  wrong.    This is due to the presence 
  of a different criticality  connected to  the geometry of boundaries of causal domains.  It is  because of  this  more subtle  criticality that we  need   at least $L^2$-bounds for the     curvature  to derive a lower bound on the radius of injectivity  of    null    hypersurfaces   and thus control their local regularity properties. This imposes a crucial   obstacle  to    well posedness below  $s=2$.  Indeed, as we will  show
   in the next subsection, any    such    result  would require, crucially,   bilinear and even trillinear  estimates  for solutions to wave equations of the form  $\square_\g \phi=F$. Such  estimates, however, depend   on    Fourier   integral representations, with a   phase function
    $u$  which   solves the     eikonal equation  $\g^{\a\b}\pr_\a  u \pr_\b u=0$.   Thus    the much needed
      bilinear estimates depend, ultimately, on the regularity properties  of the level hypersurfaces of the phase $u$ which are, of course, null.   The catastrophic        breakdown of    the regularity of these null  hypersurfaces,   in the absence of a lower bound for the injectivity  radius,  would     make     these  Fourier  integral representations    entirely  useless.
   
 These considerations  lead one to conclude that, the  following conjecture,   proposed in \cite{PDE},  is  most probably 
sharp in so far as the minimal     number of derivatives    in $L^2$  is concerned:

\noindent {\bf Conjecture}[Bounded $L^2$ Curvature Conjecture (BCC)]\quad 
 {\it The Einstein- vacuum equations  admit   local    Cauchy developments  for initial data sets   $(\Si_0, g, k)$
with locally  finite $L^2$ curvature and 
locally finite
$L^2$ norm of the  first covariant derivatives of $k$\footnote{As we shall see,  from the precise theorem     stated below, other  weaker conditions, such as  a lower bound
on the volume radius, are needed.}.}\\

\begin{remark}
It is important  to emphasize here that the conjecture
 should be primarily  interpreted as a continuation argument
for the Einstein equations; that is  the space-time
constructed by evolution from
smooth data can be smoothly continued,
together with a time foliation, 
as long as the curvature  of the foliation
and the  first 
 covariant derivatives of its second fundamental form
remain $L^2$- bounded on the leaves of the foliation. 
In fact the conjecture implies the  break-down criterion
 previously   obtained  in \cite{conditional} and improved 
 in \cite{Pa}, \cite{W}. According to that criterion a vacuum  space-time,
  endowed with   a constant mean curvature   (CMC) foliation  $\Si_t$,  can be extended, together 
  with the foliation, as long as the $L_t^1 L^\infty(\Si_t) $    norm
  of the deformation tensor of the  future unit normal to the 
  foliation remains  bounded.    It is straightforward to see, by standard energy estimates,  that   this  condition implies  bounds   for the $L_t^\infty L^2(\Si_t)$  norm of the  space-time   curvature  from which   one can derive bounds
  for the  induced curvature tensor $R$ and the  first derivatives
  of the second fundamental form $k$. Thus, 
  if  we can ensure that the time of  existence  of  a space-time  foliated by   $\Si_t$   depends only on the  $L^2$ norms of    $R$ and  first covariant  derivatives of $k$,  we can extend the space-time  indefinitely.
   \end{remark}
    In this paper we provide the framework and  the main ideas  of the proof of the conjecture. 
   We rely  on   the results  of \cite{param1},    \cite{param2},  \cite{param3},  \cite{param4}, \cite{bil2}, which we use  here  as a   black box. A summary of the entire  proof is given in \cite {KRS-Sum}. 
   We also need  to emphasize  that  the sharpness   of the    Bounded $L^2$  coonjecture   remains  itself    conjectural\footnote{ In other words we  believe    that        no well-posedness results   can be established  for general intial data     in $H^s$, with $s<2$.    }
   
   \subsection{Brief history}    
 The conjecture has its roots in the  remarkable   developments  
 of the last twenty years centered around  the  issue of optimal  well-posedness for semilinear wave equations. 
 The case of the Einstein equations  turns out  to   be a lot more complicated    due to the quasilinear character of the equations.
 To make  the discussion more tangible   it is  worthwhile to recall  the  form   of the Einstein vacuum equations in the wave gauge. 
 Assuming  given coordinates $x^\a$,  verifying  $\square_\g x^\a=0$,
  the  metric coefficients $g_{\a\b}=\g(\pr_\a, \pr_\b)$,  with respect to these coordinates,  verify the system of quasilinear wave equations,
 \bea
 \label{E-wave.coord}
 g^{\mu\nu}\pr_\mu\pr_\nu g_{\a\b} =F_{\a\b}(g, \pr g)
 \eea
 where $F_{\a\b}$ are   quadratic functions of   $\pr g$, i.e. the derivatives  of $g$ with respect to the coordinates $x^\a$.
  In a first approximation we  may   compare   \eqref{E-wave.coord}
 with the semilinear wave equation,
 \bea
 \label{semil}
 \square\phi=F(\phi, \pr \phi) 
 \eea
 with $F$ quadratic in $\pr\phi$. Using standard energy estimates,
  one can prove an estimate, roughly,   of the form:
  \beaa
  \|\phi(t)\|_{s}\les   \|\phi(0)\|_{s}\exp\left(C_s\int_0^t \|\pr \phi(\tau)\|_{L^{\infty}} d\tau\right) .
  \eeaa
 The classical exponent $s >3/2+1$  arises  simply 
 from the  Sobolev embedding of $H^r$, $r>3/2$ 
  into  $L^\infty$.  To go beyond  the classical exponent, see   \cite{Po-Si}, 
 one  has to replace   Sobolev    inequalities with  
 Strichartz estimates of, roughly,  the following type,
 \beaa
\left( \int_0^t   \|\pr \phi(\tau)\|_{L^{\infty}}  ^2d\tau\right)^{1/2}
\les C\left(\|\pr \phi(0)\|_{H^{1+\ep}}+\int_0^t \|\square \phi (\tau)\|_{H^{1+\ep}}\right)
 \eeaa 
 where  $\ep>0$ can be chosen arbitrarily small. This leads to
  a gain of $1/2$ derivatives, i.e. we can prove well-posedness 
  for equations  of type  \eqref{semil}  for any exponent $s>2$.
   
 The same  type of improvement in the case of quasilinear equations requires a highly non-trivial extension of such
  estimates  for wave operators with non-smooth coefficients.
   The first improved regularity results  for  quasilinear wave equations of the type,
   \bea
   g^{\mu\nu }(\phi) \pr_\mu\pr_\nu\phi= F(\phi, \pr \phi)\label{eq:wave-intr}
   \eea
    with $g^{\mu\nu}(\phi)$    a non-linear perturbation of 
    the Minkowski metric  $m^{\mu\nu}$, are due to 
  \cite{Ba-Ch1}, \cite{Ba-Ch2}, and \cite{Ta1}, \cite{Ta2} and 
  \cite{Kl-R1}.  
   The best known  results  for  equations  of type \eqref{E-wave.coord}  were obtained 
  in 
   \cite{Kl-R2} and \cite{Sm-Ta}.  According to them
    one can lower the Sobolev exponent  $s>5/2$ in 
    Theorem \ref{thm:Bruhat} to $s>2$.  It turns out,  see \cite{L},  that 
      these  results are sharp    in the  general class of  quasilinear wave equations    of type   \eqref{E-wave.coord}. To do better one needs   to take into account the special structure of the  Einstein      equations and rely     on  a class of  estimates  which go  beyond   Strichartz, namely  the so called bilinear estimates\footnote{Note that   no  such   result,   i.e.  well-posedness  for $s=2$,  is  presently  known  for   either   scalar equations of the   form   \eqref{eq:wave-intr}      or systems of the form \eqref{E-wave.coord}.}.

      In the case of     semilinear wave equations,
      such as Wave Maps,  Maxwell-Klein-Gordon and  Yang-Mills,  
      the first results  which make use     of bilinear estimates     
       go back to     \cite{Kl-Ma1}, \cite{Kl-Ma2}, \cite{Kl-Ma3}. 
       In the particular case of the Maxwell-Klein-Gordon and  Yang-Mills equation   the
       main observation  was that,  after the choice of a special gauge (Coulomb gauge),  the   most dangerous 
       nonlinear  terms exhibit a special,  null   structure  
      for   which one can apply the  bilinear  estimates 
         derived in   \cite{Kl-Ma1}.           
          With the help of  these  
         estimates  one was able to derive a well posedness 
         result, in  the flat Minkowski space  $\RRR^{1+3}$, 
         for the exponent   $s=s_c+1/2=1$, where $s_c=1/2$         is the critical Sobolev exponent   in that case\footnote{ This   corresponds precisely    to  the $s=2$  exponent in the case of          the Einstein-vacuum equations}.

                          To  carry out   a similar program  in the case of the  Einstein equations  one would need, at the very  least, the following crucial ingredients:\\
                         
                         {\em\begin{enumerate}
\item[{\bf A}.] Provide a   coordinate   condition, relative to which the Einstein vacuum equations verifies an appropriate version of the null condition.
\item[{\bf B}.]  Provide an appropriate    geometric  framework  for  deriving          bilinear estimates     for the null quadratic terms   appearing in the previous step.  
 \item[{\bf C}.] Construct an effective   progressive wave representation    $\Phi_F$   (parametrix)   for  solutions to 
the scalar linear wave equation  $\square_\g\phi=F$,    derive  appropriate   bounds for  both the parametrix and the corresponding error term   $E=F-\square_\g\Phi_F$ and use them to derive the desired bilinear estimates.
\end{enumerate}
    }
\noindent As it turns out, the proof of several bilinear estimates of Step B reduces to the proof of sharp $L^4(\MM)$ Strichartz estimates  for a localized version of the parametrix of step C. Thus we will also need the following fourth ingredient. 
                         {\em\begin{enumerate}
\item[{\bf D}.] Prove sharp $L^4(\MM)$ Strichartz estimates for a localized version of the parametrix of step C. \\
\end{enumerate}
    }

    Note that   the last three   steps are  to be implemented using only   hypothetical $L^2$ bounds for the space-time curvature tensor,  consistent with the conjectured result. 
    To start with,   it is  not at all  clear    what should be the correct     coordinate  condition, or even if there is one for that matter.
    \begin{remark}\label{rem:weekN}
  As mentioned above,  the only known   structural condition  related  to the classical  null condition, called the \textit{weak null condition} \cite{wNC}, tied to  wave coordinates,
 fails the regularity  test. Indeed, the following  simple  system  in Minkowski space  verifies  the  weak null condition  and yet, according to \cite{L},  it is ill posed for $s=2$.
 \beaa \Box \phi=0, \qquad \Box\psi= \phi\c \De\phi.
 \eeaa
Other coordinate conditions, such as    spatial harmonic\footnote{Maximal foliation  together with  spatial harmonic  coordinates   on  the leaves of the foliation would be the  coordinate condition closest in spirit to    the Coulomb gauge.  },  also  do not  seem to work. 
\end{remark}
  We rely  instead     on a      Coulomb type condition,  for   orthonormal frames,  adapted  to a maximal foliation.  Such a gauge condition 
 appears naturally if we adopt  a  Yang-Mills description    of the Einstein    field equations using
  Cartan's formalism of moving frames\footnote{ We would like to thank L. Andersson  for pointing     out to  us the possibility of using       such a  formalism as a  potential  bridge  to  \cite{Kl-Ma2} .}, see \cite{Ca}.   It is important  to note   
  that    it is not at all a priori   clear   that such a     choice  would  do the job. Indeed, the null  form nature 
  of the  Yang-Mills equations  in the Coulomb  gauge   is only revealed once we  commute the   resulting   equations   with     the projection operator $\PP$ on the divergence free vectorfields.  
  Such an operation is natural in that case,  since   $\PP$ commutes with the flat  d'Alembertian. 
     In the case  of the     Einstein   equations,  however,  the  corresponding   commutator  term  $[\square_\g,  \PP]$     generates\footnote{Note also that additional error terms  are generated   by projecting the    equations   on the components of the frame. } a whole host of new terms  and it is  quite a miracle that  they  can all be treated  by    an  extended version of bilinear estimates.      At an even more  fundamental level, the flat Yang-Mills equations  possess natural energy estimates  based on the  time symmetry of the Minkowski space.
 There are  no such timelike   Killing   vectorfield  in curved space.  We  have  to rely instead  on   the future unit normal to   the maximal foliation $\Si_t$     whose deformation tensor 
 is non-trivial.  This leads to another class 
 of    nonlinear terms which  have to be  treated    by a  novel trilinear estimate.

  We   will make more comments 
   concerning the implementations of   all   four  ingredients   later on,  in the   section \ref{sec:strategyproof}.
   
   \begin{remark}
   In addition to the  ingredients  mentioned  above, we also need  a mechanism of reducing     the proof of the  conjecture to  small  data,  in an appropriate sense. Indeed, even in the  flat case,  
   the Coulomb gauge condition cannot be   globally imposed for large data. In fact    \cite{Kl-Ma3}
    relied on a cumbersome technical  device based on   local Coulomb gauges, defined on domain of dependence of small  balls.      Here we rely instead on    a variant  of   the gluing construction of    \cite{Cor},  \cite{CorSch}, see section \ref{sec:reductionsmall}.         
   \end{remark}
\section{Statement of the main results}

\subsection{Maximal foliations}\lab{sec:maxfoliation}

In this section, we recall some well-known facts about maximal foliations (see for example the introduction in \cite{ChKl}). We assume the space-time $(\MM, \g)$ to be foliated by  the level surfaces $\Si_t$ of a time function $t$. Let $T$ denote the unit normal to $\Si_t$, and let $k$ the the second fundamental form of  $\Si_t$, i.e. $k_{ab}=-\g(\D_aT,e_b)$, where $e_a, a=1, 2, 3$ denotes an arbitrary frame on $\Si_t$ and $\D_aT=\D_{e_a}T$. We assume that the $\Si_t$ foliation is maximal, i.e. we have:
\bea
\label{maxfoliation}
\tr_g k=0
\eea
where $g$ is the induced metric  on $\Si_t$. The constraint equations on $\Si_t$ for a maximal 
foliation are given by:
\bea\label{constk}
\nabla^ak_{ab}=0,
\eea
where $\nabla$ denotes the induced covariant derivative on $\Si_t$, and
\begin{equation}\lab{constR}
R_{scal}=|k|^2.
\end{equation} 
Also, we denote by $n$ the lapse of the $t$-foliation, i.e. $n^{-2}=-\g(\D t, \D
t)$. $n$ satisfies the following elliptic equation on $\Si_t$:
\bea\lab{eqlapsen}
\Delta n=n|k|^2.
\eea
Finally, we recall the structure equations of the maximal foliation:
\bea\lab{eq:structfol1}
\nabla_0k_{ab}= \R_{a\,0\,b\,0}-n^{-1}\nabla_a\nabla_bn-k_{ac}k_b\,^c,
\eea
\bea\lab{eq:structfol2}
\nabla_ak_{bc}-\nabla_bk_{ac}=\R_{c0ab}
\eea
and:
\bea\lab{eq:structfol3}
R_{ab}-k_{ac}k^c\,_b=\R_{a0b0}.
\eea

\subsection{Main Theorem}

We recall below the definition of the volume radius on a general Riemannian manifold $M$.
\begin{definition}
Let $B_r(p)$ denote the geodesic ball of center $p$ and radius $r$. The volume radius $r_{vol}(p,r)$ at a point $p\in M$ and scales $\leq r$ is defined by
$$r_{vol}(p,r)=\inf_{r'\leq r}\frac{|B_{r'}(p)|}{r^3},$$
with $|B_r|$ the volume of $B_r$ relative to the metric on $M$. The volume radius $r_{vol}(M,r)$ of $M$ on scales $\leq r$ is the infimum of $r_{vol}(p,r)$ over all points $p\in M$.
\end{definition}

Our main result is the following:
\begin{theorem}[Main theorem]\lab{th:main}
Let $(\MM, {\bf g})$ an asymptotically flat solution to the Einstein vacuum equations \eqref{EVE} together with a maximal foliation by space-like hypersurfaces $\Si_t$ defined as level hypersurfaces of a time function $t$. Assume that the initial slice $(\Si_0,g,k)$ is such that  the Ricci curvature  $\mbox{Ric} \in L^2(\Si_0)$, $\nabla k\in L^2(\Si_0)$, and $\Si_0$ has a strictly positive volume radius on scales $\leq 1$, i.e. $r_{vol}(\Si_0,1)>0$. Then,
\begin{enumerate}
\item \textbf{$L^2$ regularity.} There exists a time
$$T=T(\norm{\mbox{Ric}}_{L^2(\Si_0)}, \norm{\nabla k}_{L^2(\Si_0)}, r_{vol}(\Si_0,1))>0$$
and a constant 
$$C=C(\norm{\mbox{Ric}}_{L^2(\Si_0)}, \norm{\nabla k}_{L^2(\Si_0)}, r_{vol}(\Si_0,1))>0$$
such that the following control holds on $0\leq t\leq T$:
$$\norm{\R}_{L^\infty_{[0,T]}L^2(\Si_t)}\leq C,\,\norm{\nabla k}_{L^\infty_{[0,T]}L^2(\Si_t)}\leq C\textrm{ and }\inf_{0\leq t\leq T}r_{vol}(\Si_t,1)\geq \frac{1}{C}.$$
\item \textbf{Higher regularity.} Within the same time interval as in
  part (1) we also have the higher derivative  estimates\footnote{Assuming that the initial has more regularity so that the right-hand side of \eqref{ffllffll} makes sense.},
\bea\lab{ffllffll}
\sum_{|\a|\le m}\|\D^{(\a)}\R\|_{L^\infty_{[0,T]}L^2(\Sigma_t)}\le  C_m  \sum_{|i|\le m} \bigg [\|\nab^{(i)}\textrm{Ric}\|_{L^2(\Sigma_0)} + \|\nab^{(i)} \nab k\|_{L^2(\Sigma_0)}\bigg],
\eea
where $C_m$  depends only on  the   previous $C$ and $m$. 
\end{enumerate}
\end{theorem}
\begin{remark}
Since  the core of the  main theorem is local in nature  we do not need to be very precise here with  our asymptotic flatness assumption.    We may thus  assume the existence of a coordinate system
 at infinity,   relative to which the metric     has two derivatives bounded in $L^2$, with
  appropriate asymptotic decay. Note that such bounds could be deduced from 
   weighted $L^2$ bounds  assumptions for $\mbox{Ric}$ and $\nab k$.
\end{remark}
\begin{remark} Note  that the dependence  on
 $ \norm{\mbox{Ric}}_{L^2(\Si_0)}    , \norm{\nabla k}_{L^2(\Si_0)}$    in the main theorem  can be replaced by   dependence on $  \norm{\R}_{L^2(\Si_0)}$ where $\R$ denotes the space-time curvature tensor\footnote{Here and in what follows the notations $R, {\bf R}$ will stand for the Riemann curvature tensors of $\Sigma_t$ and ${\mathcal M}$, while 
 $Ric$, ${\bf {Ric}}$ and $R_{scal}, {\bf {R}}_{scal}$ will denote the corresponding Ricci and scalar curvatures.}.  Indeed this follows from the following well known   $L^2$   estimate (see section 8 in \cite{conditional}).
 
\bea\label{eq:L2-estim-k}
\int_{\Si_0}|\nab k|^2+\frac{1}{4} |k|^4\le\int_{\Si_0} |\R|^2.
\eea
and the Gauss equation relating $\mbox{Ric}$ to $\R$.
\end{remark}

\subsection{Reduction to small initial data}\lab{sec:reductionsmall}

We first need an appropriate covering of $\Si_0$ by harmonic coordinates. This is obtained using the following general result based on Cheeger-Gromov convergence of Riemannian manifolds. 

\begin{theorem}[\cite{An} or Theorem 5.4 in \cite{Pe}]\label{th:coordharm}
Given $c_1>0, c_2>0, c_3>0$, there exists $r_0>0$ such that any 3-dimensional, complete, Riemannian manifold $(M,g)$ with $\norm{\mbox{Ric}}_{L^2(M)}\leq c_1$ and volume radius at scales $\leq 1$ bounded from below by $c_2$, i.e. $r_{vol}(M,1)\geq c_2$, verifies the following property:

Every geodesic ball $B_r(p)$ with $p\in M$ and $r\leq r_0$ admits a system of harmonic coordinates $x=(x_1,x_2,x_3)$ relative to which we have
\begin{equation}\label{coorharmth1}
(1+c_3)^{-1}\de_{ij}\leq g_{ij}\leq (1+c_3)\de_{ij},
\end{equation}
and
\begin{equation}\label{coorharmth2}
r\int_{B_r(p)}|\partial^2g_{ij}|^2\sqrt{|g|}dx\leq c_3.
\end{equation}
\end{theorem}

We consider $\ep>0$ which will be chosen as a small universal constant. We apply theorem \ref{th:coordharm} to the Riemannian manifold $\Si_0$. Then, there exists a constant:
$$r_0=r_0(\norm{\mbox{Ric}}_{L^2(\Si_0)}, \norm{\nabla k}_{L^2(\Si_0)}, r_{vol}(\Si_0,1),\ep)>0$$
such that every geodesic ball $B_r(p)$ with $p\in \Si_0$ and $r\leq r_0$ admits a system of harmonic coordinates $x=(x_1,x_2,x_3)$ relative to which we have:
$$(1+\ep)^{-1}\de_{ij}\leq g_{ij}\leq (1+\ep)\de_{ij},$$
and
$$r\int_{B_r(p)}|\partial^2g_{ij}|^2\sqrt{|g|}dx\leq \ep.$$

Now, by the asymptotic flatness of $\Si_0$,  the complement of its end   can be covered by the union of a finite number of geodesic balls of radius $r_0$, where the number $N_0$ of geodesic balls required only depends on $r_0$. In particular, it is therefore enough to obtain the control of $\R$, $k$ and $r_{vol}(\Si_t,1)$ of Theorem \ref{th:main} when one restricts to the domain of dependence of one such ball. Let us denote this ball by $B_{r_0}$. Next, we rescale the metric of this geodesic ball by:
$$g_\la(t,x)=g(\la t,\la x),\, \la=\min\left(\frac{\ep^2}{\norm{R}^2_{L^2(B_{r_0})}},\,\frac{\ep^2}{\norm{\nabla k}^2_{L^2(B_{r_0})}},\, r_0\ep\right)>0.$$
Let\footnote{Since in  what follows there is no danger to confuse  the Ricci curvature  $\mbox{Ric}$ with the scalar  curvature $R$
we use the short hand  $R$    to denote the full curvature  tensor   $\mbox{Ric}$.  }    $R_\la, k_\la$ and $B_{r_0}^\la$ be the rescaled versions of $R$, $k$ and $B_{r_0}$. Then, in view of our choice for $\la$, we have:
$$\norm{R_\la}_{L^2(B_{r_0}^\la)}=\sqrt{\la} \norm{R}_{L^2(B_{r_0})}\leq \ep,$$
$$\norm{\nabla k_\la}_{L^2(B_{r_0}^\la)}=\sqrt{\la} \norm{\nabla k}_{L^2(B_{r_0})}\leq \ep,$$
and 
$$\norm{\partial^2g_\la}_{L^2(B_{r_0}^\la)}=\sqrt{\la} \norm{\partial^2g}_{L^2(B_{r_0})}\leq \sqrt{\frac{\la\ep}{r_0}}\leq \ep.$$
Note that $B_{r_0}^\la$ is the rescaled version of $B_{r_0}$. Thus, it is a geodesic ball for $g_\la$ of radius $\frac{r_0}{\la}\geq \frac{1}{\ep}\geq 1$. Now, considering $g_\la$ on $0\leq t\leq 1$ is equivalent to considering $g$ on $0\leq t\leq \la$. Thus, since $r_0$, $N_0$ and $\la$ depend only on $\norm{R}_{L^2(\Si_0)}$, $\norm{\nabla k}_{L^2(\Si_0)}$, $r_{vol}(\Si_0,1)$ and $\ep$, Theorem \ref{th:main} is equivalent to the following theorem:

\begin{theorem}[Main theorem, version 2]\lab{th:mainbis}
Let $(\MM, {\bf g})$ an asymptotically flat solution to the Einstein vacuum equations \eqref{EVE} together with a maximal foliation by space-like hypersurfaces $\Si_t$ defined as level hypersurfaces of a time function $t$. Let $B$ a geodesic ball of radius one in $\Si_0$, and let $D$ its domain of dependence. Assume that the initial slice $(\Si_0,g,k)$ is such that:
$$\norm{R}_{L^2(B)}\leq \ep,\,\norm{\nabla k}_{L^2(B)}\leq \ep\textrm{ and }r_{vol}(B,1)\geq \frac{1}{2}.$$
Let $B_t=D\cap \Si_t$ the slice of $D$ at time $t$. Then:
\begin{enumerate}
\item\textbf{$L^2$ regularity.} There exists a small universal constant $\ep_0>0$ such that if $0<\ep<\ep_0$, then the following control holds on $0\leq t\leq 1$:
$$\norm{\R}_{L^\infty_{[0,1]}L^2(B_t)}\lesssim \ep,\,\norm{\nabla k}_{L^\infty_{[0,1]}L^2(B_t)}\lesssim \ep\textrm{ and }\inf_{0\leq t\leq 1}r_{vol}(B_t,1)\geq \frac{1}{4}.$$
\item \textbf{Higher regularity.} The following control holds on $0\leq t\leq 1$:
\bea
\sum_{|\a|\le m}\|\D^{(\a)} \R\|_{L^\infty_{[0,1]}L^2(B_t)}  \les  
  \|\nab^{(i)} \textrm{Ric}\|_{L^2(B)} + \|\nab^{(i)} \nab k\|_{L^2(B)}.
\eea
\end{enumerate}
\end{theorem}

\noindent{\bf Notation:}\,\, In the statement of Theorem \ref{th:mainbis}, and in the rest of the paper, the notation $f_1\les f_2$ for two real positive scalars $f_1, f_2$ means that there exists a universal constant $C>0$ such that:
$$f_1\leq C f_2.$$\\ 

Theorem \ref{th:mainbis} is not  yet   in  suitable form for our proof since some of our constructions will be global in space and may not be carried out on a subregion $B$ of $\Si_0$. Thus, we glue a smooth asymptotically flat solution of the constraint equations \eqref{const} outside of $B$, where the gluing takes place in an annulus just outside  $B$. This can be achieved using the construction  in \cite{Cor},  \cite{CorSch}. We finally get an asymptotically flat solution to the constraint equations,  defined everywhere on $\Si_0$,  which  agrees with our original data set $(\Si_0,g,k)$ inside $B$. We still denote this data set by $(\Si_0,g,k)$. It satisfies the bounds: 
$$\norm{R}_{L^2(\Si_0)}\leq 2\ep,\,\norm{\nabla k}_{L^2(\Si_0)}\leq 2\ep\textrm{ and }r_{vol}(\Si_0,1)\geq \frac{1}{4}.$$

\begin{remark}
Notice that the gluing process in \cite{Cor}--\cite{CorSch} requires the kernel of a certain linearized operator to be trivial. This is achieved by conveniently choosing the asymptotically flat solution to \eqref{const} that is glued outside of $B$ to our original data set. This choice is always possible since the metrics for which the kernel is nontrivial are non generic (see \cite{BCS}).
\end{remark}

\begin{remark}
The    formal  proofs  in the above mentioned  results  require higher regularity than $R, \nabla k \in L^2$.    However the  problem solved  there is elliptic and  subcritical  at   our desired   level of regularity.   Based on recent progress  it is highly   expected  that the results   in  \cite{Cor}--\cite{CorSch}   hold   true   under our    regularity assumptions, and even below, see   \cite{Sch}.\end{remark}

\begin{remark}
Since  $\|k\|_{L^4(\Si_0)} ^2\le \|\mbox{Ric}\|_{L^2} $
 we  deduce   that 
$\norm{k}_{L^2(B)}\lesssim \ep^{1/2}$
 on the geodesic ball $B$ of radius one.
Furthermore, asymptotic flatness is compatible with a decay of $|x|^{-2}$ at infinity, and in particular with $k$ in $L^2(\Si_0)$. So we may assume that the gluing process is such that the resulting $k$ satisfies:
$$\norm{k}_{L^2(\Si_0)}\lesssim \ep.$$
\end{remark}

Finally, we have reduced Theorem \ref{th:main} to the case of a small initial data set:
\begin{theorem}[Main theorem, version 3]\lab{th:mainter}
Let $(\MM, {\bf g})$ an asymptotically flat solution to the Einstein vacuum equations \eqref{EVE} together with a maximal foliation by space-like hypersurfaces $\Si_t$ defined as level hypersurfaces of a time function $t$. Assume that the initial slice $(\Si_0,g,k)$ is such that:
$$\norm{R}_{L^2(\Si_0)}\leq \ep,\,\norm{k}_{L^2(\Si_0)}+\norm{\nabla k}_{L^2(\Si_0)}\leq \ep\textrm{ and }r_{vol}(\Si_0,1)\geq \frac{1}{2}.$$
Then:
\begin{enumerate}
\item\textbf{$L^2$ regularity.} There exists a small universal constant $\ep_0>0$ such that if $0<\ep<\ep_0$,  the following control holds on $0\leq t\leq 1$:
$$\norm{\R}_{L^\infty_{[0,1]}L^2(\Si_t)}\lesssim \ep,\,\norm{k}_{L^\infty_{[0,1]}L^2(\Si_t)}+\norm{\nabla k}_{L^\infty_{[0,1]}L^2(\Si_t)}\lesssim \ep\textrm{ and }\inf_{0\leq t\leq 1}r_{vol}(\Si_t,1)\geq \frac{1}{4}.$$
\item \textbf{Higher regularity.} The following estimates hold on $0\leq t\leq 1$ and any $m>0$:
\bea
\sum_{|\a|\le m}\|\D^{(\a)} \R\|_{L^\infty_{[0,1]}L^2(\Sigma_t)}  \les  
\sum_{|i|\le m}  \|\nab^{(i)} \textrm{Ric}\|_{L^2(\Sigma_0)} + \|\nab^{(i)} \nab k\|_{L^2(\Sigma_0)}.
\eea
\end{enumerate}
\end{theorem}

The rest of this paper is devoted to the proof of Theorem \ref{th:mainter}. Note that  we will concentrate mainly on  part (1) of the  theorem.  The proof of part (2) - which concerns the propagation of higher regularity - follows exactly the same   steps   as the proof of  part (1) and is sketched in section \ref{sect:propag}.

\subsection{Strategy of the proof}\lab{sec:strategyproof}

The proof of Theorem \ref{th:mainter} consists in four steps. \\

\noindent{\bf Step A (Yang-Mills formalism)} We  first cast the Einstein-vacuum equations  within a Yang-Mills formalism. This relies on  the Cartan formalism of moving frames. The idea 
 is to give up on    a choice of coordinates and express  instead 
 the Einstein vacuum equations  in terms of the connection  $1$-forms    associated to  moving orthonormal  frames,  i.e.  vectorfields  $e_\a$,   
 which verify, 
 \beaa
\g(e_\a, e_\b)=\m_{\a\b}=\mbox{diag}(-1,1,1,1).
\eeaa
The connection $1$-forms (they are to be interpreted 
as $1$-forms with respect to the  external index $\mu$  with values  in the Lie algebra of $so(3,1)$),  defined  by the formulas,
\bea
(\A_\mu)_{\a\b}=\g(\D_{\mu}e_\b,e_\a)\label{intr.A}
\eea
verify the equations,
\bea
\label{intr-YM}
\D^\mu \F_{\mu\nu} + [ \A^\mu, \F_{\mu\nu}]=0
\eea
where,  denoting $ (\F_{\mu\nu})_{\a\b}:=\R_{\a\b\mu\nu}$,
\bea
\label{intr-YM1}
( \F_{\mu\nu})_{\a\b}=\big(\D_\mu \A_\nu-\D_\nu \A_\mu
-[\A_\mu,
\A_\nu]\big)_{\a\b}.
\eea
In other words  we can interpret   the curvature   tensor    as the curvature of the $so(3,1)$-valued  connection  1-form  $\A$. Note also that the covariant derivatives    are taken only with respect to the \textit{external indices} $\mu, \nu$ and  do not affect the \textit{internal indices} $\a,\b$.
 We can rewrite \eqref{intr-YM} in the form,
 \bea
 \label{intr-YM2}
 \square_\g \A_\nu-\D_\nu(\D^\mu \A_\mu) =\J_\nu(\A, \D \A)
 \eea
 where,
 \beaa
 \J_\nu=\D^\mu([\A_\mu, \A_\nu])-[\A_\mu, \F_{\mu\nu}].
 \eeaa
 Observe that the equations  \eqref{intr-YM}-\eqref{intr-YM1}
 look just like the Yang-Mills equations    on a fixed Lorentzian
 manifold $(\MM, \g)$ except, of course,   that  in our case
 $\A$ and $\g$ are  not independent but connected rather   by  \eqref{intr.A},
  reflecting the quasilinear  structure of the Einstein equations.
  Just as in the case   of   \cite{Kl-Ma1},  which  establishes the well-posedness of the  Yang-Mills equation  in Minkowski space 
  in   the    energy norm (i.e. $s=1$),    we    rely   in an essential manner  on   a    Coulomb type  gauge condition.   More precisely,  we  take 
   $e_0$  to be  the  future  unit normal to the 
   $\Si_t$ foliation and choose  $e_1, e_2, e_3$ an 
   orthonormal basis  to  $\Si_t$,   in such a way
   that we have, essentially (see precise discussion in 
   section \ref{sect:compatible}),  $\div A=\nab^i A_i=0$, where 
   $A$ is the spatial component of $\A$.   It turns out that $A_0$ satisfies
   an elliptic equation  while each  component  $A_i=\g(\A, e_i)$,  $ i=1,2,3$
   verifies an equation of the form,
   \bea
     \square_\g A_i  &=&-\pr_i (\pr_0 A_0)+ A^j \pr_j A_i +A^j\pr_i A_j+\lot\label{intr:null}
\eea
with $\lot$ denoting nonlinear  terms which  can be treated by more elementary techniques (including non sharp Strichartz estimates). \\     

\noindent{\bf Step B (Bilinear and trilinear estimates)} To eliminate $\pr_i (\pr_0 A_0)$  in \eqref{intr:null},  we need   to project \eqref{intr:null} onto  divergence free vectorfields   with the help of  a non-local operator which we denote  by $\PP$.      In  the  case of the   flat  Yang-Mills 
equations, treated in   \cite{Kl-Ma1},   this leads to an equation 
of the form,
\beaa
\square  A_i&=&\PP(A^j \pr_j A_i ) +\PP(A^j\pr_i A_j)+\lot
\eeaa
where  both terms on the right can be handled  by bilinear 
 estimates.    In our case  we encounter  however three 
 fundamental differences with the  flat  situation of    \cite{Kl-Ma1}.
 \begin{itemize}
 \item To start with the operator $\PP$ does not commute 
 with $\square_\g$.   It turns out,  fortunately,   that  the terms
 generated by  commutation     can still be estimated  by 
 an extended  class of bilinear estimates which includes  contractions 
 with the curvature tensor, see section \ref{sec:proofconj}. 
 \item All  energy estimates  used in   \cite{Kl-Ma1}
 are based on the standard    timelike   Killing vectorfield $\pr_t$. In our case the corresponding 
 vectorfield  $e_0=T$ ( the future unit   normal to $\Si_t$) is not Killing.
 This leads  to another class of trilinear  error terms  which 
  we  discuss in sections    \ref{sec:bobo6} and \ref{sec:proofconj}.
  \item The main difference 
   with  \cite{Kl-Ma1}  is that we now need bilinear and trilinear  estimates
    for solutions of wave equations on    background metrics 
   which   possess only limited  regularity.  
\end{itemize} 
    This last  item is a major   problem,  both conceptually and technically.
   On the conceptual side we need to rely on  a more geometric    proof of  bilinear estimates  based on   a plane    wave representation    formula\footnote{We follow   the proof   of the  bilinear estimates  outlined   in  \cite{Kl-R3}
    which differs substantially  from that of  \cite{Kl-Ma1} and
    is reminiscent of the  null frame space strategy used by   Tataru  in his      fundamental paper \cite{Ta3}. } for solutions of  scalar wave equations,
   \beaa
   \square_\g \phi=0.
   \eeaa 
   The proof of the bilinear estimates rests on the representation formula\footnote{\eqref{parametrix.intr} actually corresponds to the representation formula for a half-wave. The full representation formula corresponds to the sum of two half-waves (see section \ref{sec:parametrix})}
   \bea
   \label{parametrix.intr}
   \phi_f(t,x)=\int_{\SSS^2} \int_0^\infty e^{i\la \uom(t,x)} \,  f(\la\om) \la^2 d\la  d\om
   \eea
    where $f$ represents schematically the initial data\footnote{Here $f$ is in fact at the level of  the Fourier transform of the initial data  and the norm  $\|\la f\|_{L^2(\RRR^3)  }  $  corresponds, roughly,  to the $H^1$ norm of the data .},    and where  $\uom$ is a solution of the eikonal equation\footnote{In the flat Minkowski space   $\uom(t,x)=t\pm x\c\om$.},
   \bea
   \g^{\a\b}\pr_\a\uom\,  \pr_\b \uom=0,\label{eikonal-intr}
   \eea
   with appropriate initial conditions on $\Si_0$ and $d\om$
    the area element of the standard sphere in $\RRR^3$.\\
    
    \begin{remark}\label{rem:scalarize}
    Note that \eqref{parametrix.intr} is a parametrix for a scalar wave equation. The lack of a good parametrix for a covariant wave equation forces us to develop a strategy based on writing the main equation in components relative to a frame, i.e. instead of dealing with the tensorial wave equation \eqref{intr-YM2} directly, we consider the system of scalar wave equations \eqref{intr:null}. Unlike in the flat case, this scalarization procedure produces several terms which are potentially dangerous, and it is fortunate that they can still be controlled by the use of an extended\footnote{such as  contractions  between the Riemann curvature tensor and   derivatives of  solutions of   scalar  wave equations.} class of  bilinear estimates.
    \end{remark}
    
 \vspace{0.3cm}   
 
    \noindent{\bf Step C (Control of the parametrix)} To prove the bilinear and trilinear estimates of Step B, we need in particular to control the parametrix at initial time (i.e. restricted to the initial slice $\Sigma_0$)
   \bea
   \label{parametrixinit.intr}
   \phi_f(0,x)=\int_{\SSS^2} \int_0^\infty e^{i\la \uom(0,x)} \,  f(\la\om) \la^2 d\la  d\om
   \eea
 and the error term corresponding to \eqref{parametrix.intr}    
   \bea
   \label{error-param}
  Ef(t,x)=\square_\g  \phi_f(t,x)=i \int_{\SSS^2} \int_0^\infty e^{i\la \uom(t,x)}   
  \,(\square_\g \uom )
   f(\la\om)  \la^3d\la d\om
   \eea
   i.e.  $\phi_f$   is an exact solution of  $\square_\g\phi=0$  only in flat space
   in which case $ \square_\g \uom=0$. This requires the following four sub steps
{\em\begin{enumerate}
\item[{\bf C1}] Make an appropriate choice for the equation satisfied by $\uom(0,x)$ on $\Sigma_0$, and control the geometry of the foliation of $\Sigma_0$ by the level surfaces of $\uom(0,x)$.

\item[{\bf C2}] Prove that the parametrix at $t=0$ given by \eqref{parametrixinit.intr} is bounded in $\mathcal{L}(L^2(\mathbb{R}^3), L^2(\Sigma_0))$ using the estimates for $\uom(0,x)$ obtained in {\bf C1}.

\item[{\bf C3}] Control the geometry of the foliation of $\mathcal{M}$ given by the level hypersurfaces of $\uom$.

\item[{\bf C4}] Prove that the error term \eqref{error-param} satisfies the estimate $\norm{Ef}_{L^2(\mathcal{M})}\leq C\norm{\lambda f}_{L^2(\mathbb{R}^3)}$ using the estimates for $\uom$ and $\square_{\g}\uom$ proved in {\bf C3}.
\end{enumerate}
}
   
To achieve Step C3 and Step C4, we need, at the very  least,  to control  $\square_\g \uom$
in $L^\infty$. This issue   was first addressed in 
 the sequence of papers \cite{Kl-R4}--\cite{Kl-R6} where an $L^\infty$ bound for  $\square_\g \uom$ was established, 
 depending only on the $L^2$ norm of the curvature flux along null hypersurfaces.      The proof required   an interplay between 
 both  geometric and analytic techniques and  had all the appearances 
 of being sharp, i.e.    we don't expect an $L^\infty$ bound for $\square_\g \uom$  which requires bounds on  less
 than two derivatives  in $L^2$ for the metric\footnote{classically, this requires, at the very  least, the control of $\R$ in $L^\infty$}.
 
 \begin{remark}
 The control of $\square_\g \uom$ in $L^\infty$ is not only crucial to achieve Step C4, but it also plays a fundamental role in Step C3. Indeed, it allows to control the radius of injectivity of the level hypersurfaces of $\uom$. Let us note that in \cite{Kl-R7}, such a control is obtained under the additional assumption of a uniform bound for the second fundamental form $k$. As we do not have such a uniform bound for $k$ in our case,  we are forced to proceed differently. We refer the reader to section 4.1 in \cite{param3} for the details.
 \end{remark}
 
 To obtain the $L^2$ bound for the Fourier integral operator $E$ defined in \eqref{error-param}, we 
 need, of course, to  go  beyond uniform   estimates 
 for $\square_\g \uom$. The classical $L^2$ bounds 
 for Fourier integral operators of the form \eqref{error-param} 
 are not at all economical in terms of the number of integration by parts which are needed.   In our case the total  number of     
 such  integration by parts  is limited by  the   regularity properties of the function $\square_\g\uom$.  To
   get an $L^2$ bound for the parametrix at initial time \eqref{parametrixinit.intr}  and the error term \eqref{error-param} within such restrictive regularity properties 
   we  need, in particular:
   \begin{itemize}
   \item In Step C1 and Step C3, a precise  control of     derivatives of $\uom$ and $\square_\g \uom$ 
   with respect to both $\om$ as well as with respect to various 
   directional derivatives\footnote{Taking  into account 
   the   different    behavior in tangential and transversal directions
    with respect to the level surfaces of $\uom$.}.   
   To get optimal control  we need,
   in particular,  a  very careful   construction of   the initial condition for    $\uom $ on $\Si_0$  and then sharp space-time estimates  of Ricci   coefficients, and their  derivatives,     associated     to the foliation induced by  $\uom$.  
   
\item  In Step C2 and Step C4, a careful decompositions of    the Fourier integral operators \eqref{parametrixinit.intr} and  \eqref{error-param} in both  $\la$  and $\om$, 
 similar to the first and second dyadic decomposition in harmonic analysis, see 
 \cite{St}, as well as a third    decomposition, which in the case of  \eqref{error-param} is done with respect 
 to the space-time variables    relying on  the geometric Littlewood-Paley  theory developed in  \cite{Kl-R6}.
\end{itemize}

Below, we make further  comments on Steps C1-C4:
\begin{enumerate}
\item \textit{The choice of $u(0,x,\om)$ on $\Sigma_0$ in Step C1.} Let us note that the typical choice $u(0,x,\om)=x\c\om$ in a given coordinate system would not work for us, since we don't have enough control on the regularity of a given coordinate system within our framework. Instead, we need to find a geometric definition of $u(0,x,\om)$. A natural choice would be
$$\square_\g u=0\textrm{ on }\Sigma_0$$
which by a simple computation turns out to be the following simple variant of the minimal surface equation\footnote{In the time symmetric case $k=0$, this is exactly the minimal surface equation}
$$\div\left(\frac{\nabla u}{|\nabla u|}\right)=k\left(\frac{\nabla u}{|\nabla u|}, \frac{\nabla u}{|\nabla u|}\right)\textrm{ on }\Sigma_0.$$
Unfortunately, this choice does not allow us to have enough control of the derivatives of $u$ in the normal direction to the level surfaces of $u$. This forces us to look for an alternate equation for $u$:
$$\div\left(\frac{\nabla u}{|\nabla u|}\right)=1-\frac{1}{|\nabla u|}+k\left(\frac{\nabla u}{|\nabla u|}, \frac{\nabla u}{|\nabla u|}\right)\textrm{ on }\Sigma_0.$$
This equation turns out to be parabolic in the normal direction to the level surfaces of $u$, and allows us to obtain the desired regularity in Step C1. On closer inspection  it is related with the  well known mean curvature flow on $\Si_0$.

\item \textit{How to achieve Step C3.} The regularity obtained in Step C1, together with null transport equations tied to the eikonal equation, elliptic systems of Hodge type, the geometric Littlewood-Paley theory of \cite{Kl-R6}, sharp trace theorems, and an extensive use of the structure of the Einstein equations, allows us to propagate the regularity on $\Sigma_0$ to the space-time, thus achieving Step C3.

\item \textit{The regularity with respect to $\om$ in Steps C1 and C3.} The regularity with respect to $x$ for $u$ is clearly limited as a consequence of the fact that we only assume $L^2$ bounds on $\R$. On the other hand, $\R$ is independent of the parameter $\om$, and one might infer that $u$ is smooth with respect to $\om$. Surprisingly, this is not at all the case. Indeed,  the regularity in $x$ obtained for $u$ in Steps C1 and C3 is better in directions tangent to the level hypersurfaces of $u$. Now, the $\om$ derivatives of the tangential directions have non zero normal components. Thus, when differentiating the structure equations with respect to $\om$, tangential derivatives to the level surfaces of $u$ are transformed in non tangential derivatives which in turn severely limits the regularity in $\om$ obtained in Steps C1 and C3.

\item \textit{How to achieve Steps C2 and C4.} Let us note that the classical arguments for proving $L^2$ bounds for Fourier operators are based either on a $T T^*$ argument, or a $T^* T$ argument, which requires several integration by parts either with respect to $x$ for $T^*T$, of with respect to $(\la, \om)$ for $TT^*$. Both methods would fail by far within the regularity for $u$ obtained in Step C1 and Step C3. This forces us to design a method which allows to take advantage both of the regularity in $x$ and $\om$. This is achieved using in particular the following ingredients:
\begin{itemize}
\item geometric integrations by parts taking full advantage of the better regularity properties in directions tangent to the level hypersurfaces of $u$,

\item the standard first and second dyadic decomposition in frequency  space, with  respect to  both size and angle (see \cite{St}), an additional decomposition in physical space relying on the geometric Littlewood-Paley projections of \cite{Kl-R6} for Step C4, as well as another decomposition involving frequency and angle for Step C2.
\end{itemize}
Even with these precautions, at several places in the proof, one encounters log-divergences which have to be tackled by ad-hoc techniques,  taking full advantage of the structure of the Einstein equations. 
\end{enumerate}

\vspace{0.5cm}

\noindent{\bf Step D (Sharp $L^4(\MM)$ Strichartz estimates)} Recall that the parametrix constructed   in Step C needs also to be used to prove  sharp $L^4(\MM)$ Strichartz estimates. Indeed the proof of several bilinear estimates of Step B reduces to the proof of sharp $L^4(\MM)$ Strichartz estimates for the parametrix \eqref{parametrix.intr} with $\la$ localized in a dyadic shell. 

More precisely, let $j\geq 0$, and let $\psi$ a smooth function on $\RRR^3$ supported in 
$$\frac 1 2 \leq |\xi|\leq 2.$$
Let $\phi_{f,j}$ the parametrix \eqref{parametrix.intr} with a additional frequency localization $\la\sim 2^j$
\begin{equation}\lab{paraml}
\phi_{f,j}(t,x)=\int_{\SSS^2} \int_0^\infty  e^{i \la \uom(t,x)}\psi(2^{-j}\la)f(\la\om)\la^2d\la d\om.
\end{equation}
We will need the sharp\footnote{Note in particular that the corresponding estimate in the flat case is sharp.} $L^4(\MM)$ Strichartz estimate
\begin{equation}\lab{strichgeneralintro}
\norm{\phi_{f,j}}_{L^4(\MM)}\les 2^{\frac{j}{2}}\norm{\psi(2^{-j}\la)f}_{L^2(\RRR^3)}.
\end{equation}
The standard procedure for proving\footnote{Note that the procedure we  describe would prove not only   \eqref{strichgeneralintro}   but the full   range  of mixed Strichartz estimates.}          \eqref{strichgeneralintro} is  based  on a  $TT^*$ argument  which  reduces  it  to an  $L^\infty$ estimate for an oscillatory integral with a phase involving $\uom$. This is then achieved  by  the method of  stationary phase which  requires  quite a  few integrations by parts.    In fact   the standard  argument   would require, at the least\footnote{The regularity \eqref{staphase} is necessary to make sense of the change of variables involved in the stationary phase method.}, that the phase function
$u=\uom$ verifies,
\begin{equation}\lab{staphase}
\partial_{t,x} u  \in L^\infty,\,\partial_{t,x}\partial_\omega^2 u\in L^\infty.
\end{equation}
This level of regularity is, unfortunately,  incompatible with the regularity  properties    of solutions
  to our  eikonal equation  \eqref{eikonal-intr}. 
 In fact,  based  on the estimates for $\uom$  derived in step C3, 
  we  are only allowed      to assume
\begin{equation}\lab{reguintro}
\partial_{t,x}  u  \in L^\infty,\, \partial_{t,x}\partial_\omega u \in L^\infty.
\end{equation}
 We are  thus are forced to follow an alternative approach\footnote{ We  refer  to  the approach  based on   the  overlap estimates for wave packets derived in \cite{Sm} and \cite{Sm-Ta} in the context of Strichartz estimates respectively for $C^{1,1}$ and $H^{2+\ep}$ metrics.  Note however that our approach    does not require a wave packet  decomposition. } to the stationary phase  method inspired by    \cite{Sm} and    \cite{Sm-Ta} .

\subsection{Structure of the paper} 

The rest of this paper is devoted to the proof of Theorem \ref{th:mainter}. Here  are the main steps.
\begin{itemize}
\item
In section \ref{sec:YMgaugetheory}, we start by  describing  the  Cartan formalism and  introduce compatible  frames, i.e. frames   $e_0, e_1, e_2, e_3$   with $e_0$  the future unit normal to the foliation $\Si_t$ and $(e_1, e_2, e_3)$ an orthonormal
  basis on $\Si_t$.  We choose  $e_1, e_2, e_3$  such that the   spatial components 
$A=(A_1, A_2, A_3)$  verify the Coulomb   condition $\nab^i A_i=0$. We then 
decompose the equations  \eqref{eq:YM7}-\eqref{eq:YM8}
  relative   to the frame. This leads   to scalar  equations for  
  $A_0=\g(\A, e_0)$ and $A_i=\g(\A, e_i)$ of the form (see Proposition \ref{prop:decomp}),
  \beaa
  \De A_0&=&  \lot\\
  \square A_i  &=&-\pr_i (\pr_0 A_0)+ A^j \pr_j A_i +A^j\pr_j A_i+\lot
  \eeaa
  where $\lot$ denote nonlinear  terms   for  which  the specific structure is irrelevant, i.e.  no bilinear estimates are needed.  The  entire proof of the  bounded $L^2$  conjecture   is  designed  to  treat the difficult  terms $A^j \pr_j A_i$  and   $A^j\pr_j A_i$. 
\item To eliminate $\pr_i (\pr_0 A_0)$ and exhibit the null structure of the term  $A^j\pr_j A_i$  we need to project the second equation onto  divergence free vectorfields.    Unlike the   flat  case of  the Yang-Mills equation (see \cite{Kl-Ma3}),  the projection  does not commute with $\square$ and we have to be  very careful with  the   commutator terms which it generates.
We effectively achieve    the desired effects of  the projection by introducing the quantity $B=(-\De)^{-1} \curl A$.  The main commutation  formulas  are discussed  in section 6 and proved in  the appendix. 
 \item In section \ref{sec:preliminaries}, we  start by deriving various preliminary  estimates 
 on the initial slice $\Si_0$, discuss an appropriate version of  Uhlenbeck's lemma  and show how to control  $A_0$, $A$ as well  as $B=(-\De)^{-1} \curl A$,  from our initial assumptions  on $\Si_0$.
 
 \item In section \ref{sec:strategyoftheproof}, we introduce our  bootstrap assumptions  and describe  the principal  steps in the proof of   version 3  our main theorem, i.e. Theorem  \ref{th:mainter}.  Note that  we are by no means economical in our choice of bootstrap assumptions. We have decided to  give  a longer  list, than strictly necessary,    in the hope that it will make  the   proof more transparent.  We make,  in particular,   a list of  bilinear, and even trilinear and Strichartz   bootstrap assumptions which   take advantage  of the special structure of the Einstein equations.  The trilinear    bootstrap assumption is   needed  in order to derive the crucial  $L^2$ estimates for the curvature  tensor. The  entire proof of   Theorem \ref{th:mainter}  is summarized  in Propositions \ref{prop:improve1} and \ref{prop:improve2} in which  all the bootstrap assumptions are   improved by estimates  which depend only on the  initial data, as well as Proposition \ref{prop:propagreg} in which we prove the propagation of higher regularity.  
  
 \item In section \ref{sec:simpleconsequencesboot}, we discuss  various elliptic estimates on the slices  $\Si_t$,  derive estimates for $B$ from the bootstrap assumptions on $A$, and we show how to  derive estimates for   $A$ from those of  $B$.
 
 \item In  section \ref{sec:bobo5}, we use the bootstrap assumptions  to derive $L^2$-spacetime  estimates
  for $\square B$ and $\pr\square B$, estimates which are crucial    in order to provide a parametrix
  representation for $B$  and prove the bilinear estimates stated  in proposition   \ref{prop:improve2}.
 It is crucial here that all the commutator  terms generated   in the process continue to have 
 the crucial bilinear structure   discussed above and thus can be all estimated by  our bilinear  bootstrap 
 assumptions.
  
  \item In section \ref{sec:bobo6}, we  derive energy estimates for the wave equations  $\square_\g  \phi=F$,
  relying again on the  bootstrap assumptions, in particular the  trilinear ones. 
  
  \item In section \ref{sec:improve1}, we improve on our basic bootstrap assumption, i.e. all bootstrap assumptions except the bilinear, trilinear and Strichartz bootstrap assumptions. This corresponds to proving Proposition   \ref{prop:improve1}.   
  
  \item In section \ref{sec:parametrix}, we show how to construct  parametric representation formulas for solutions   to the scalar  wave equation $\square_\g\phi=F$.  The main result of the section, Theorem \ref{lemma:parametrixconstruction}, depends  heavily  on Theorem \ref{prop:estparam} whose proof 
  requires, essentially,    all the constructions and proofs of    the   papers  \cite{param1}-\cite{param4}. Theorem   \ref{prop:estparam} is in fact the main black box of this paper. 
  
  \item In sections \ref{sec:improve2} and \ref{sec:improve2B}, we improve on our bilinear, trilinear and Strichartz bootstrap assumptions. This corresponds to proving Proposition \ref{prop:improve2}. Note that we rely on sharp $L^4(\MM)$ Strichartz estimates for a parametrix localized in frequency (see Proposition \ref{prop:L4strichartz}) which are proved in \cite{bil2}.

\item Finally, we prove the propagation of higher regularity in section \ref{sect:propag}. This corresponds to proving Proposition \ref{prop:propagreg}. 
\end{itemize}

The rest of this paper is devoted to the proof of Theorem \ref{th:mainter}. Note that  we will concentrate mainly on  part (1) of the  theorem.  The proof of part (2), which   follows  exactly  the same   steps   as the proof  (with some obvious simplifications)   of  part (1),    is          sketched in section \ref{sect:propag}.

\begin{remark}
To re-emphasize  that   the special structure of the Einstein equations is of fundamental importance in    deriving our result 
we would like to stress that the  bilinear estimates are needed not only to   treat the terms of the form
$A^j \pr_j A_i$  and   $A^j\pr_j A_i$ mentioned above (which are  also present   in  flat space) but also to derive  energy estimates  for solutions to $\square_\g \phi=F$. 
Moreover, a trilinear estimate is required to     get $L^2$ bounds for $\R$.  In addition  to   these, a result such as Theorem  \ref{prop:estparam} cannot possibly hold  true, for metrics $\g$   with our  limited degree of regularity,   unless   the Einstein equations are satisfied,  i.e. $\mbox{Ric}(\g)=0$.
Indeed a   crucial element of  a construction of  a parametrix
 representation for solutions to $\square_\g \phi=F$, guaranteed in  Theorem  \ref{prop:estparam}, is the control and regularity of a family of phase functions 
 with level hypersurfaces which are null with respect to $\g$.  As mention   a few times in this introduction, such controls are intimately tied to the null geometry of a space-time, e.g. lower  lower  bounds 
 for  the radius of injectivity of null hypersurfaces, and would fail,  by a lot,  for a general Lorentzian metric $\g$. 
  \end{remark}
  
  \vspace{0.2cm}
  
{\bf Conclusion.} Though this result falls short of the  crucial  goal of finding a scale invariant  well-posedness  criterion in GR, it is clearly optimal in terms of all currently available  ideas and  techniques. Indeed, within our current understanding,   a better  result would require enhanced   bilinear estimates, which in turn would rely heavily on parametrices. On the other hand, parametrices are based on solutions to the eikonal equation whose control 
  requires, at least, 
$L^2$ bounds for the curvature tensor, as can be seen in many instances in our work.  Thus,  if we are   
 to ultimately   find    a    scale invariant  well-posedness  criterion, it is clear that an entirely new circle        of ideas is  needed. Such a goal is clearly of fundamental importance not just to GR, but also to any physically relevant quasilinear hyperbolic system. \\ 

{\bf Acknowledgements.} This work would be inconceivable without the extraordinary  advancements made 
on nonlinear wave equations  in the last twenty years in which so many have
 participated.
We would like to  single out the contributions of those who have affected this  work  in a more direct fashion,   either through their papers or through  relevant  discussions,  in  various stages   of its long gestation.  D.  Christodoulou's     seminal work \cite{Ch2}  on the    weak cosmic censorship conjecture   had a direct motivating   role on our program, starting with   a series of papers  between  the first author and M. Machedon, in which spacetime  bilinear estimates     were first introduced and used    to take advantage of   the null structure of geometric semilinear  equations such as Wave Maps and Yang-Mills.  The works of  Bahouri- Chemin  \cite {Ba-Ch1}-\cite{Ba-Ch2}   and     D.Tataru \cite{Ta2}   were       the first  to go below  the classical  Sobolev exponent  $s=5/2$,   for any quasilinear system  in  higher dimensions. This   was, at the time,  a major psychological  and technical   breakthrough which opened the way for future developments.  Another major breakthrough  of the period,   with direct   influence on   our  approach to bilinear estimates in curved spacetimes,  is   D. Tataru's work \cite{Ta3}  on  critical well posedness      for Wave Maps, in which null frame spaces were first introduced.    His  joint work  with H. Smith \cite{Sm-Ta}  which, together with \cite{Kl-R2},   is the first to reach   optimal     well-posedness without  bilinear estimates, has also   influenced our approach on parametrices and Strichartz estimates.  
 The authors  would also  like  to acknowledge fruitful conversations with  L. Anderson,  and   J. Sterbenz.

\section{Einstein vacuum equations as Yang-Mills gauge theory}\lab{sec:YMgaugetheory}

\subsection{Cartan formalism} \label{sect:Cartan} 

Consider an Einstein vacuum spacetime $(\MM, \g)$. We denote the  covariant differentiation by $\D$. 
   Let $e_{\a}$ be an orthonormal frame
on $\MM$, i.e. 
\beaa
\g(e_\a, e_\b)=\m_{\a\b}=\mbox{diag}(-1,1,\ldots,1).
\eeaa
Consistent with the Cartan formalism we  define the connection 1 form, 
\bea
\label{eq:YM1abstract}
(\A)_{\a\b}(X)=\g(\D_{X}e_\b,e_\a)
\eea
where $X$ is an arbitrary vectorfield in  $T(\MM)$.
Observe that,
\beaa
(\A)_{\a\b}(X)=-(\A)_{\b\a}(X)
\eeaa
i.e.  the $1$ -form $\A_\mu dx^\mu$  takes values  in 
 the Lie  algebra of $so(3,1)$. We separate the internal indices
 $\a,\b$ from the  external indices $\mu$  according to the following 
 notation.
\bea
\label{eq:YM1}
(\A_\mu)_{\a\b}:=  (\A)_{\a\b}(\pr_\mu)      =       \g(\D_{\mu}e_\b ,e_\a)
\eea

Recall that the Riemann curvature tensor is defined by
\beaa
\R(X, Y,  U, V)=\g\big(X, \big[ \D_U \D_V-\D_V \D_U -\D_{[U,V]} Y\big]\big)
\eeaa
with $X, Y, U, V$ arbitrary vectorfields in $T(\MM)$. 
Thus, taking $U=\pr_\mu, V=\pr_\nu$,   coordinate vector-fields,
\beaa
\R(e_\a, e_\b , \pr_\mu, \pr_\nu)=\g\big(e_\a,  \D_\mu \D_\nu e_\b  -\D_\nu \D_\mu e_\b \big).
\eeaa
We write,
\beaa
\D_\nu e_\b=(D_\nu e_\b, e_\la) e_\la= (\A_\nu)^{\la}\, _{\b} e_\la
\eeaa
and,
\beaa
\D_\mu\D_\nu e_\b&=&\D_\mu\big((\A_\nu)^{\la}\,_{\b} e_\la)=\pr_\mu
 (\A_\nu)^\la\, _\b e_\la+(\A_\nu)^\la\, _\b \D_\mu e_\la\\
 &=&\pr_\mu (\A_\nu)^\la\, _\b  e_\la+(\A_\nu)^\la\, _\b   (\A_\mu)^\si\, _\la e_\si.
\eeaa
Hence,
\bea\lab{linkRA}
\R(e_\a, e_\b , \pr_\mu, \pr_\nu)&=&\pr_\mu (\A_\nu)_{\a\b}-
\pr_\nu (\A_\mu)_{\a\b}+(\A_\nu)_\a\,^\la  (\A_\mu)_{\la\b}-
(\A_\mu)_\a\,^\la  (\A_\nu)_{\la\b}.
\eea
Thus defining the Lie bracket,
\bea
\label{eq:YM2}
([\A_\mu, \A_\nu])_{\a\b}=(\A_\mu)_{\a}{\,^\ga} \,( \A_\nu)_{\ga\b}-
(\A_\nu)_{\a}{\,^\ga} \,( \A_\mu)_{\ga\b}
\eea
we obtain:
$$
\R_{\a\b\mu\nu}=\pr_\mu (\A_\nu)_{\a\b}-\pr_\nu (\A_\mu)_{\a\b}-([\A_\mu, \A_\nu])_{\a\b},
$$
or, since $\pr_\mu (\A_\nu)-\pr_\nu
(\A_\mu)=\D_\mu \A_\nu-\D_\nu \A_\mu$
\bea
\label{eq:YM3}
( \F_{\mu\nu})_{\a\b}=\R_{\a\b\mu\nu}=\big(\D_\mu \A_\nu-\D_\nu \A_\mu
-[\A_\mu,
\A_\nu]\big)_{\a\b}.
\eea
Therefore we can interpret   $\F$  as the curvature of the connection  $\A$.

Consider  now the covariant derivative of the Riemann curvature tensor,
\beaa
\D_\si \R_{\a\b\mu\nu}&=&(\D_\si \F_{\mu\nu})_{\a\b}-
 \R_{\D_\si\a\b\mu\nu}-\R_{\a \D_\si\b\mu\nu}\\
&=&(\D_\si \F_{\mu\nu})_{\a\b}-\R^\de_{\,\,\b\mu\nu}\,\g(\D_\si e_\a, e_\de)-
{\R_{\a}^{\,\, \de}}_{\,\mu\nu}\,\g(\D_\si e_\b ,e_\de)\\
&=&(\D_\si \F_{\mu\nu})_{\a\b}-(\A_\si)_{\a}\, ^{\de}(\F_{\mu\nu})_{\de\b}-
(\A_\si)_{\b}\, ^{\de}(\F_{\mu\nu})_{\a\de}\\
&=&(\D_\si \F_{\mu\nu})_{\a\b}+(\A_\si)_{\a}\,^{\de}(\F_{\mu\nu})_{\de\b}-
(\F_{\mu\nu})_{\a}\,^{\de}
(\A_\si)_{\de\b}\\
&=&\big( \D_\si \F_{\mu\nu}+[\A_\si, \F_{\mu\nu}]\big)_{\a\b}.
\eeaa
Hence,
\bea
\label{eq:YM4}
\D_\si \R_{\a\b\mu\nu}=\Da_\si F_{\mu\nu}:= \D_\si \F_{\mu\nu}+[\A_\si, \F_{\mu\nu}]
\eea
where we denote by $\Da$ the  covariant derivative
on the corresponding vector bundle. More precisely
if $\U=\U_{\mu_1\mu_2\ldots \mu_k}$ is any $k$-tensor
 on $\MM$ with values on the Lie algebra of $so(3,1)$,
\bea
\label{eq:YM4'}
\Da_\si  \U=
 \D_\si  \U+ [\A_\si, \U]. 
\eea

\begin{remark}
Recall that in $(\A_\mu)_{\a\b}$, $\a, \b$ are called the internal indices, while $\mu$ are called the  external indices. Now, the internal indices will be irrelevant for most of the paper. Thus, from now on, we will drop these internal indices, except for rare instances where we will need to distinguish between internal indices of the type $ij$ and internal indices of the type $0i$.
\end{remark}

The Bianchi identities for $\R_{\a\b\mu\nu}$
 take the form 
\bea
\label{eq: YM5}
\Da_\si \F_{\mu\nu}+\Da_\mu F_{\nu\si}+\Da_\nu \F_{\si\mu}=0.
\eea
As it is well known the Einstein vacuum equations $\R_{\a\b}=0$
imply 
$ \D^\mu \R_{\a\b\mu\nu}=0.$
Thus, in view of   equation    \eqref{eq:YM4},
\bea
\label{eq:YM6}
0=\Da^\mu\, \F_{\mu\nu}= \D^\mu \F_{\mu\nu} + [ \A^\mu, \F_{\mu\nu}]
\eea
or, in view of \eqref{eq:YM3} and  the vanishing of the Ricci curvature of $\g$,.
\bea
\label{eq:YM7}
\square \A_\nu-\D_\nu(\D^\mu \A_\mu) =\J_\nu
\eea
where
\bea
\label{eq:YM8}
\J_\nu=\D^\mu([\A_\mu, \A_\nu])-[\A_\mu, \F_{\mu\nu}].
\eea
Using again the vanishing of the Ricci curvature it is easy to check,
\bea
\label{eq:YM9}
\D^\nu \J_\nu=0.
\eea

Finally we recall the general formula of transition between two different orthonormal frames  $e_\a$ and $ \widetilde{e}_\a$ on $\MM$, related by, 
\beaa
 \widetilde{e}_\a=\O_{\a}^\ga e_\ga
\eeaa
where
$
\m_{\a\b}=\O_\a^\ga \O_\b^\de \, \m_{\ga\de}
$,
i.e. $\O$ is a smooth map from $\MM$ to the Lorentz group $O(3,1)$.
In other words,  raising and lowering indices with respect to $\m$,
\bea
\O_{\a\la}\O^{\b\la}=\de_{\a}^\b
\eea
Now, $(\widetilde{\A}_\mu)_{\a\b}=\g(\D_\mu  \widetilde{e}_\b, \widetilde{e}_\a)$. Therefore,
\bea
\label{gauge.tr}
(\widetilde{\A}_\mu)_{\a\b}=\O_\a^\ga \O_\b^\de (\A_\mu)_{\ga\de}+\pr_\mu (\O_\a^\ga) \, \O_{\b}^\de\,  \m_{\ga\de}
\eea

\subsection{Compatible frames}
\label{sect:compatible}

Recall that our spacetime is assumed to be foliated by  the level surfaces $\Si_t$ of a 
time function $t$,  which are maximal, i.e.  denoting by $k$
the second fundamental form of  $\Si_t$ we have,
\bea\lab{eq:YM9bis}
\tr_g k=0
\eea
where $g$ is the induced metric  on $\Si_t$. Let us choose $e_{(0)}=T$, the future unit normal to the $\Si_t$  foliation, and $e_{(i)}$, $i=1,2,3$  an orthonormal frame tangent  to $\Si_t$.  We call this   a frame compatible with our $\Si_t$ foliation.  We consider the connection coefficients \eqref{eq:YM1}
with respect to this frame. Thus, in particular, denoting   by  $A_0$,   respectively $A_i$, the temporal and spatial components
of $\A_\mu$
\bea
(A_i)_{0j}&=&(A_j)_{0i}=-k_{ij} ,\qquad i,j=1,2,3\label{eq:YM10}  \\
(A_0)_{0i}   &=&-n^{-1}\nabla_in  \qquad i=1,2,3\ \label{eq:YM11} 
\eea
where $n$ denotes the lapse of the $t$-foliation, i.e. $n^{-2}=-\g(\D t, \D
t)$. With this notation  we note  that,
\beaa
\nabla_lk_{ij}=\nabla_l(k_i)_j+k_{in}(A_l)_j\,^n=\nabla^l(A_i)_{0j}+k_{in}(A_l)_j\,^n
\eeaa
 where $\nabla$ is the induced covariant derivative on $\Sigma_t$ and, as  before, the notation  $\nabla_l(k_i)_j$ or    $ \nabla^l(A_i)_{0j}   $,  is meant to suggest that
 the covariant differentiation affects only the external index $i$. Recalling from \eqref{constk} that $k$ verifies the constraint equations,
$$\nabla^ik_{ij}=0,$$
 we derive,
 \bea\lab{coulomblike1}
 \nab^i( A_i)_{0j}=k_i\,^m (A_i)_{mj}.
 \eea
 
  Besides the choice of $e_0$ we are still free to make  a choice
  for the  spatial elements of the frame $e_1, e_2, e_3$.  In other words we consider  frame transformations which keep $e_0$ fixed, i.e transformations
  of the type, 
  \beaa
  \tilde{e}_i=O_i^j e_j
  \eeaa
  with  $O$ in the orthogonal group $O(3)$. 
  We now have,  according to \eqref{gauge.tr},
  \beaa
  (\tilde{A}_m)_{ij}=O_i^k O_j^l  (A_m)_{kl}+\pr_m(O_i^k) O_j^l \de_{kl}
  \eeaa
  or, schematically,
  \bea
  \tilde{A}_m &=& O A_m O^{-1}+(\pr_m O) O^{-1}
  \eea
  formula in which we understand that   only the   spatial internal indices 
  are involved. We shall use this freedom later to exhibit a frame $e_1, e_2, e_3$ 
  such that the corresponding connection $A$ satisfies the Coulomb gauge condition 
  $\nabla^l(A_l)_{ij}=0$ (see Lemma \ref{lemma:uhlenbeck}).
 
 \subsection{Notations}
 
 We introduce notations used throughout the paper. From now on, we use greek indices to denote general indices on $\MM$ which do not refer to the particular frame $(e_0, e_1, e_2, e_3)$. The letters $a, b, c, d$ will be used to denote general indices on $\Si_t$ which do not refer to the particular frame $(e_1, e_2, e_3)$. Finally, the letters $i, j, l, m, n$ will only denote indices relative to the frame $(e_1, e_2, e_3)$. Also, recall that $\D$ denotes the covariant derivative on $\MM$, while $\nabla$ denotes the induced covariant derivative on $\Si_t$. Furthermore, $\prb$ will always refer to the derivative of a scalar quantity relative to one component of the frame $(e_0, e_1, e_2, e_3)$, while $\pr$ will always refer to the derivative of a scalar quantity relative to one component of the the frame $(e_1, e_2, e_3)$, so that $\prb=(\pr_0, \pr)$. For example, 
$\pr A$ may be any term of the form $\pr_i(A_j)$, $\pr_0A$ may be any term of the form $\pr_0(A_j)$, 
$\pr A_0$ may be any term of the form $\pr_j(A_0)$, and $\prb\A=(\prb A, \prb A_0)=(\pr_0A_0, \pr A_0, \pr_0A, \pr A)$.

We introduce the curl operator $\curl$ defined for any $so(3,1)$-valued triplet $(\om_1, \om_2, \om_3)$ of functions on $\Si_t$  as  follows:
\bea\lab{def:curl}
(\curl \om)_i=\in_i\,^{jl}\partial_j(\om_l),
\eea
where $\in_{ijl}$ is fully antisymmetric and such that $\in_{123}=1$. We also introduce the divergence operator $\div$ defined for any $so(3,1)$-valued tensor $A$ on $\Si_t$  as  follows:
\bea\lab{def:div}
\div A=\nabla^l (A_l)=\pr^l (A_l)+A^2.
\eea
\begin{remark}
Since $\pr_0$ and $\pr_j$ are not coordinate derivatives, note that the commutators $[\pr_j, \pr_0]$ and $[\pr_j, \pr_l]$ do not vanish. Indeed, we have for any scalar function $\phi$ on $\MM$:
\beaa 
[\pr_i, \pr_j]\phi&=&[e_i, e_j]\phi
=(\D_ie_j-\D_je_i)\phi\\
&=&-((\D_ie_j,e_0)-(\D_je_i,e_0))e_0(\phi)+((\D_ie_j,e_l)-(\D_je_i,e_l))e_l(\phi)\\
&=&-((A_i)_{0j}-(A_j)_{0i})\pr_0\phi+((A_i)_{lj}-(A_j)_{li})\pr_l\phi,
\eeaa
and:
\beaa 
[\pr_i, \pr_0]\phi&=&[e_i, e_0]\phi=(\D_ie_0-\D_0e_i)\phi\\
&=&-((\D_ie_0,e_0)-(\D_0e_i,e_0))e_0(\phi)+((\D_ie_0,e_l)-(\D_0e_i,e_l))e_l(\phi)\\
&=&(A_0)_{0i}\pr_0\phi+((A_i)_{l0}-(A_0)_{li})\pr_l\phi.
\eeaa
Taking \eqref{eq:YM10} and \eqref{eq:YM11} into account, this can be written schematically as:
\bea\lab{notcoordinate}
[\pr_i, \pr_j]\phi=A\pr\phi\textrm{ and }[\pr_j, \pr_0]\phi=n^{-1}\nabla n\,\pr_0\phi+\A\pr\phi,
\eea
for any scalar function $\phi$ on $\MM$.
\end{remark}

\begin{remark}
The term $A^2$ in \eqref{def:div} corresponds to a quadratic expression in components of $A$, 
where the particular indices do not matter. In the rest of the paper, we will adopt this schematic notation for lower order terms (e.g. terms of the type $A^2$ and $A^3$) where the particular indices do not matter. 
\end{remark}

Finally, $\square A_0$ and $\square A_i$ will always be understood as $\square$ applied to the scalar functions$A_0$ and $A_i$, while $(\square A)_\a$ 
will refer to $\square$ acting on the differential form $A_\alpha$.
Also, $\Delta A_0$ will always refer to $\Delta(A_0)$.
 
 \subsection{ Main equations  for $(A_0, A)$} 
 
In what follows we  rewrite  equations  \eqref{eq:YM7}--\eqref{eq:YM8} with respect to the components $A_0$ and $A=(A_1, A_2, A_3)$. To do this we need the following simple lemma.
 \begin{lemma}\label{lemma:scalarize}
For any vectorfield $X$, we have:
\bea
\label{eq:YM12}
X^\a(\square \A)_\a=\square(X\cdot \A)-2\D^\lambda X\c\D_\la \A-(\square X)\cdot \A.
\eea
\end{lemma} 

Taking $X=e_0$ in the lemma   and noting that,
\beaa
\square \, e_0= \D^\la\D_\la e_0=-\D^\la( \A_\la)_0\,^\ga e_\ga -( \A_\la)_0\,^\ga ( \A^\la)_{\ga}\, ^\mu e_\mu
\eeaa  
as well as\footnote{Recall that $\tr k=0$.}   
\bea
\label{eq:YM15}
\D^\mu (\A_\mu)&=&-\D_0 (\A_0)+\D^i (\A_i)=-\big[\pr_0 A_0+ (A_0)_0\,^i  (A_i)\big]
+\big[\nab^i (A_i)\big].
\eea
 we derive, keeping track of the term in $\pr_0 A_0$,
\beaa
 (\square \A)_0&=&\square A_0+2 (\A^\la)_0\,^\ga\, \D_\la (\A_\ga)
 -(\square e_0)\c \A\\
  &=&\square A_0+2 (\A^\la)_0\,^\ga\, \D_\la (\A_\ga)+
 \D^\la (\A_\la)_0\,^\ga \A_\ga +( \A_\la)_0\,^\ga ( \A^\la)_{\ga}\, ^\mu \A_\mu\\
 &=&
 \square A_0+2 (\A^\la)_0\,^i\, \D_\la (\A_i)+
 \D^\la (\A_\la)_0\,^i  \A_i +( \A_\la)_0\,^i  ( \A^\la)_{i}\, ^\mu \A_\mu\\
 &=&\square A_0+2 (\A^0 )_0\,^i\, \D_0 (\A_i)+2 (\A^j  )_0\,^i\, \D_j (\A_i)
-  \D_0 (\A_0)_0\,^i  \A_i         +\D_i  (\A_i)_0\,^j  \A_j         \\
&+&( \A_\la)_0\,^i  ( \A^\la)_{i}\, ^\mu \A_\mu
 \eeaa
 Thus, symbolically, recalling that $(A_0)_{0i}=-n^{-1}\nab_i n$,
 \beaa
  (\square \A)_0&=& \square(A_0)-\pr_0 (A_0)_0\,^i (A)_i+(A, n^{-1}\nab n) (\pr  \A+\pr_0 A)+ (A, n^{-1}\nab n) \c \A^2.
 \eeaa

  On the other hand,
 \beaa
\pr_0(\D^\mu (\A_\mu))&=&
-\pr_0^2 A_0-\pr_0 (A_0)_0\,^i (A)_i+\pr_0(\nab^i(A_i))   - (A_0)_0\,^i     \pr_0(A)_i
 \eeaa
Hence,
\beaa
(\square \A)_0- \pr_0(\D^\mu (\A_\mu))&=&
\square A_0-\pr_0 (A_0)_0\,^i (A_i)+\pr_0^2 A_0+\pr_0 (A_0)_0\,^i (A)_i- \pr_0(\nab^i(A_i))\\
&+&(A, n^{-1}\nab n) (\pr  \A+\pr_0 A)+ (A, n^{-1}\nab n) \c \A^2\\
&=&\square A_0 +\pr_0^2 A_0- \pr_0(\nab^i(A_i))\\
&+&(A, n^{-1}\nab n) (\pr  \A+\pr_0 A)+ (A, n^{-1}\nab n) \c \A^2.
\eeaa
On the other hand we have, by a straightforward computation,
 for any scalar $\phi$,
\bea
\label{eq:YM16}
\square\phi=-\pr_0(\pr_0\phi)+\Delta\phi+n^{-1}\nabla n\cdot\nabla\phi,
\eea
with $\De$ denoting the standard Laplace-Beltrami operator on $\Si_t$.
Therefore,
\bea\lab{eq:YMintermediateequation}
(\square \A)_0- \pr_0(\D^\mu (\A_\mu))&=&\lap A_0- \pr_0(\nab^i(A_i))+(A, n^{-1}\nab n) (\pr  \A+\pr_0 A)+ (A, n^{-1}\nab n) \c \A^2.
\eea
Finally, recalling \eqref{eq:YM8}, we have, 
\beaa
\J_0&=&\D^\mu[\A_\mu, \A_0]-[\A_\mu, \F_{\mu 0}]=(A, n^{-1}\nab n) (\pr  \A+\pr_0 A)+ (A, n^{-1}\nab n) \c \A^2.
\eeaa
Hence the $e_0$ component of \eqref{eq:YM7} takes the form,
\bea
\label{eq:YM17}
\De A_0-\pr_0(\nab^i A_i) = (A, n^{-1}\nab n) (\pr  \A+\pr_0 A)+ (A, n^{-1}\nab n) \c \A^2.
\eea
According to  \eqref{coulomblike1} we have,
\beaa
 \nab^i( A_i)_{0j}=-k_i\,^m (A_i)_{mj}.
 \eeaa
We are thus  free  to impose the  \textit{Coulomb like } gauge condition,
\bea
\label{eq:YM18}
\nab^i (A_i) _{jk}&=&0.
\eea
In fact we  write both \eqref{coulomblike1} and \eqref{eq:YM18}
in the form,
\bea
\label{eq:YM19}
\nab^i (A_i) &=&A^2.
\eea
With this choice of gauge 
 equation \eqref{eq:YM17}  takes  the form,
\bea
\label{eq:YM20}
\De A_0 = (A, n^{-1}\nab n) (\pr  \A+\pr_0 A)+ (A, n^{-1}\nab n) \c \A^2.
\eea
\begin{remark}
\label{Rem:important}
Note   that  the symbolic  structure  of the nonlinear  terms  in \eqref{eq:YM20}  is meant to emphasize  that   there are no  quadratic 
 terms of the form   $A_0  \prb A  $ and   $A_0\pr A_0$ and no cubic terms   in $(A_0)_{ij}$.  This fact is important;  as we shall see below we have  good  control  for   $L^2(\Si_t)$ norms    for  $(A_k)_{ij}$, $(A_i)_{0j} = -k_{ij} $   and   $(A_0)_{0i}=-    n^{-1}\nab_i n$ but not  for $(A_0)_{ij}$. 
 \end{remark}

It remains to derive equations for the  scalar components $A_i$, $i=1,2,3$.
First we observe, in view of \eqref{eq:YM15} and \eqref{eq:YM19},
\bea
\label{eq:YM20'}
\D^\la \A_\la=-\D_0A_0+\D^iA_i=-\pr_0A_0+\nabla^iA_i+A^2=-\pr_0 A_0+ A^2.
\eea
Using lemma \ref{lemma:scalarize} with  $X=e_{(i)}$, $i=1,2,3$  
 we  derive, 
 \beaa
 \square A_i&=& (\square \A)_i-2 (\A^\la)_i\,^\ga\, \D_\la (\A_\ga)-
\D^\la( \A_\la)_i\,^\ga \A_\ga-
( \A_\la)_i\,^\ga ( \A_\la)_\ga\,^\mu \A_\mu
\eeaa
or,
schematically, ignoring  signs or numerical constants in front of the quadratic and cubic terms:
\beaa
 \square A_i&=& (\square \A)_i+A{^j} \pr_j A_i+A_0\prb\A+A\prb A_0 +\A^3.
 \eeaa
 Recalling  \eqref{eq:YM9} we have,
\beaa
(\square \A)_i- \pr_i(\D^\mu (\A_\mu)) =J_i.
\eeaa
where $J_i$ is the $e_{(i)}$ component of $\J$.
Therefore,
\beaa
\square A_i       +\pr_i (\pr_0 A_0)   &=&A^j\c\pr_j A_i+ J_i+A_0\prb\A+A\prb A_0 +\A^3.
\eeaa
On the other hand, recalling the definition of $\J$ in \eqref{eq:YM8},
we easily find,
\beaa
J_i&=&A^i\c \pr_i A+ [A^j, F_{ji}]+A_0\prb\A+A\prb A_0+\A^3.
\eeaa 
Therefore, schematically,
\beaa
\square A_i       +\pr_i (\pr_0 A_0)   &=&A^j\c \pr_j A_i+A^j\c  \pr_i A_j+A_0\prb\A+A\prb A_0 +\A^3.
\eeaa
We summarize the results of this subsection in the following proposition.
\begin{proposition}
\label{prop:decomp}
Consider  an orthonormal frame $e_\a$  compatible with  a    maximal  $\Si_t$ foliation of the space-time $\MM$ 
with connection coefficients    $\A_\mu$   defined by \eqref{eq:YM1}, their decomposition  $\A=(A_0, A)$ relative to the  same frame  $e_\a$,
and Coulomb- like condition on the frame,
\beaa
\div A=A^2.
\eeaa
 In such a frame the  Einstein-vacuum equations take the form,
\bea
\De A_0&=&(A, n^{-1}\nab n) (\pr  \A+\pr_0 A)+ (A, n^{-1}\nab n) \c \A^2. \label{eq:YM21}
\eea
\bea
\label{eq:YM22}
\square A_i+\pr_i (\pr_0 A_0)   &=& A^j\pr_j A_i+A^j\pr_i A_j+A_0\prb\A+A\prb A_0 +\A^3.
\eea
\end{proposition}
 
 \begin{remark}
 It is extremely important to our strategy that we have reduced the covariant wave equation \eqref{eq:YM7} to the system of scalar equations \eqref{eq:YM21} \eqref{eq:YM22} (see remark \ref{rem:scalarize}).    
 \end{remark} 
 
\begin{remark}\lab{rem:improvedeqpr0A0} 
We will also need a suitable equation for $\Delta(\pr_0A_0)$. This is done in Proposition \ref{prop:equationpr0A0} below where we carefully check that the  terms of the type $(\prb A)^2$ in $\Delta(\pr_0A_0)$ can be avoided. Indeed, in view of the expected estimate $\prb A\in L^\infty_tL^2(\Sigma_t)$, this would lead to $\Delta(\pr_0A_0) \in L^\infty_tL^1(\Sigma_t)$ which is borderline for solving the Laplace equation. To check that such terms can be avoided, we treat separately the components $(A_0)_{0j}$  and the components $(A_0)_{ij}$. We exploit the fact that
\begin{itemize}
\item The components $(A_0)_{0j}$ are given by $n^{-1}\nabla n$ in view of \eqref{eq:YM11}. This allows us to rely on the elliptic equation \eqref{eqlapsen} for $n$ which is easier to handle than the equation \eqref{eq:YM21} for $A_0$ as it is at the level of one less derivative.

\item The equation for $\Delta(A_0)$ involves in particular a term $\pr_0(A^2)$ which comes from $\pr_0(\nab^iA_i)$ together with the Coulomb type gauge condition \eqref{eq:YM19}. This term would prevent us from deriving a good equation for $\Delta(\pr_0A_0)$ as it would lead in particular to a term of the type $(\pr_0 A)^2$ which is dangerous as explained above. In the case of the components $(A_0)_{ij}$, this term is fortunately absent since we have $\nab^l (A_l) _{ij} =0$ in view of \eqref{eq:YM18}.
\end{itemize} 
\end{remark}

\begin{proposition}\lab{prop:equationpr0A0}
We have 
\beaa
\Delta(\pr\pr_0n) &=& \pr(A\R)+(A, n^{-1}\nabla n) \,\pr(\pr_0A_0)+(\pr(\pr_0n), \A^2)(\R, \prb A,\pr A_0, \A^2).
\eeaa
and
\beaa
&&\De (\pr_0((A_0)_{jl}))\\
&=&  \pr\Big((A, n^{-1}\nab n)(\R, \prb \A)+A_0(\pr_0 A, \R)+\A^3\Big)+\pr A_0 (\R, \pr_0 A)\\
&&+(A, n^{-1}\nabla n) (\pr(\pr_0A_0), A\pr_0A_0)+(\pr(\pr_0n), \A^2)(\R, \prb A,\pr A_0, \A^2).
\eeaa
\end{proposition} 
 
\begin{proof}
We first derive an equation for $\nabla \pr_0 n$. Recall the elliptic equation \eqref{eqlapsen} for $n$
$$\Delta n=n|k|^2.$$
Differentiating with respect to $\pr_0$ and using the commutator formula \eqref{commutdeltapr0}, we infer
\beaa
\Delta(\pr_0n) &=& \pr_0(\Delta n)+[\Delta, \pr_0]n\\
&=& \pr_0(n)|k|^2+2nk\c\nab_0 k+2k^{ab}\nabla_a\nabla_bn-2n^{-1}\nabla_bn\nabla_b(\pr_0n)-|k|^2\pr_0n\\
&&+2n^{-1}\nabla_an k^{ab}\nabla_bn\\
&=& 2nk\c\nab_0 k+2k^{ab}\nabla_a\nabla_bn-2n^{-1}\nabla_bn\nabla_b(\pr_0n)+2n^{-1}\nabla_an k^{ab}\nabla_bn.
\eeaa
Together with the identity \eqref{eq:structfol1} for $\nab_0k$, this leads to
\beaa
\Delta(\pr_0n) &=& 2nk^{ab} (\R_{a\,0\,b\,0}-n^{-1}\nabla_a\nabla_bn-k_{ac}k_b\,^c)+2k^{ab}\nabla_a\nabla_bn-2n^{-1}\nabla_bn\nabla_b(\pr_0n)\\
&&+2n^{-1}\nabla_an k^{ab}\nabla_bn\\
&=&  2nk^{ab}\R_{a\,0\,b\,0}- 2nk^{ab}k_{ac}k_b\,^c-2n^{-1}\nabla_bn\nabla_b(\pr_0n)+2n^{-1}\nabla_an k^{ab}\nabla_bn.
\eeaa
Differentiating with respect to $\nabla$ and using the commutator formula \eqref{commutdeltanabla}, we obtain symbolically
\beaa
\Delta\nabla\pr_0n &=& \nabla\Delta(\pr_0n)+[\Delta,\nabla]\pr_0n\\
&=&  \nabla\Big(nk\R+nk^3+n^{-1}\nabla n\nabla (\pr_0n)+n^{-1}(\nabla n)^2 k\Big)+(\R+k^2)\nabla\pr_0n\\
&=& \nabla(k\R)+nk^2\nabla k+k^3\nabla n+n^{-1}\nabla n \nabla^2(\pr_0n)+n^{-1}\nabla^2n \nabla(\pr_0n)\\
&& +(n^{-1}\nabla n)^2 \nabla(\pr_0n)+n^{-1}(\nabla n)^2\nabla k+n^{-1}\nabla n \nabla^2nk+n^{-2}(\nabla n)^3 k+\\
&& +(\R+k^2)\nabla\pr_0n\\
&=& \pr(A\R)+n^{-1}\nabla n \,\pr(\pr_0A_0)+(\pr(\pr_0n), \A^2)(\R, \prb A,\pr A_0, \A^2).
\eeaa
Scalarizing with $e_1$, $e_2$ and $e_3$, we infer
\beaa
\Delta(\pr\pr_0n) &=& \pr(A\R)+(A, n^{-1}\nabla n) \,\pr(\pr_0A_0)+(\pr(\pr_0n), \A^2)(\R, \prb A,\pr A_0, \A^2).
\eeaa

Next, we derive an elliptic equation for $(A_0)_{ij}$. First, note that the full structure was not taken into account in the derivation of the equation for $\Delta(A_0)$. Indeed, the terms $A\pr A$ in \eqref{eq:YMintermediateequation} have a notable structure, and a more careful derivation of  actually leads to the following improved version
\beaa
&& (\square \A)_0- \pr_0(\D^\mu (\A_\mu))\\
&=&\lap A_0- \pr_0(\nab^i(A_i))+A^i\pr_i A+A \pr^i A_i+n^{-1}\nab n\, \pr_0 A+ (A, n^{-1}\nab n) \c \A^2\\
&=&\lap A_0- \pr_0(\nab^i(A_i))+\pr(A^2)+n^{-1}\nab n (\pr  A_0, \R_{0i\c\c})+ (A, n^{-1}\nab n) \c \A^2
\eeaa
where we used in the last equality the Coulomb type gauge condition \eqref{eq:YM19} to simplify the terms $A^i\pr_iA$ and $A\pr^iA_i$, and the identity \eqref{eq:YM3} to replace $\pr_0A$ by $\pr A_0$, $\R_{0i\c\c}$ and lower order terms. Hence the $e_0$ component of \eqref{eq:YM7} takes the form,
\bea\label{eq:YM17bbis}
\De A_0-\pr_0(\nab^i A_i) = \J_0+\pr(A^2)+n^{-1}\nab n (\pr  A_0, \R_{0i\c\c})+ (A, n^{-1}\nab n) \c \A^2.
\eea
Also, we have
\beaa
\De ((A_0)_{jl})-\pr_0((\nab^i A_i)_{jl}) &=& (\De A_0-\pr_0(\nab^i A_i))_{jl}+A\pr A_0+A_0 \pr A
\eeaa
which together with the Coulomb gauge condition \eqref{eq:YM18} for $(A_0)_{jl}$ implies
\beaa
\De ((A_0)_{jl}) &=& (\De A_0-\pr_0(\nab^i A_i))_{jl}+A\pr A_0+A_0 \pr A.
\eeaa
In view of \eqref{eq:YM17bbis}, we deduce
\begin{equation}\label{eq:YM18bbis}
\De ((A_0)_{jl})  = \J_0+\pr(A^2)+(A, n^{-1}\nab n) \pr  A_0+n^{-1}\nab n \R_{0i\c\c} +A_0 \pr A+ (A, n^{-1}\nab n) \c \A^2.
\end{equation}

\begin{remark}
Our next task is to differentiate \eqref{eq:YM18bbis} with respect to $\pr_0$. Recall from Remark \ref{rem:improvedeqpr0A0} that the main  main advantages of \eqref{eq:YM18bbis} for $(A_0)_{jl}$ compared to equation \eqref{eq:YM21} which is valid for all components of $A_0$ is that the dangerous term $\pr_0(\nab^iA_i)$ is absent since we have $\nab^l (A_l) _{ij} =0$ in view of \eqref{eq:YM18}. Also, by being more careful in the derivation of \eqref{eq:YM18bbis} compared to \eqref{eq:YM21}, we have achieved the following
\begin{itemize} 
\item In \eqref{eq:YM18bbis}, we have kept $\J_0$ as it is. When taking a $\pr_0$ derivative of \eqref{eq:YM18bbis}, this will allow us to  recast $\pr_0\J_0$ as a space divergence up to lower order terms by exploiting the fact that $\J_\mu$ is divergence free in view of \eqref{eq:YM9}.

\item The term $A\pr A$ in \eqref{eq:YM21} can actually be integrated by parts as $\pr(A^2)$ up to lower order terms. This will avoid dangerous terms of the type $\pr A \pr_0 A$ after differentiation with respect to $\pr_0$.  
\end{itemize}
\end{remark}

Next, we differentiate \eqref{eq:YM18bbis} with respect to $\pr_0$. We have
\beaa
\pr_0(\De ((A_0)_{jl})) &=&  \pr_0(\J_0)+\pr(A\pr_0 A)+[\pr_0,\pr](A^2)+(A, n^{-1}\nab n) \pr (\pr_0A_0)\\
&&+(A, n^{-1}\nab n) [\pr, \pr_0]A_0 +(\pr_0 A, \pr_0(n^{-1}\nab n)) \pr  A_0 +n^{-1}\nab n \pr_0(\R_{0i\c\c})\\
&&+A_0(\pr(\pr_0A), [\pr_0, \pr]A)+\pr_0A_0(\R, \pr A)+(A, n^{-1}\nab n) \c \A\pr_0\A\\
&&+(\pr_0A, \pr_0(n^{-1}\nab n)) \c \A^2\\
&=&  \pr_0(\J_0)+n^{-1}\nab n\pr_0(\R_{0i\c\c})+\pr\Big(A(\R, \pr_0 A)+A_0\pr_0 A\Big)+\pr A_0 \pr_0 A\\
&&+(A, n^{-1}\nabla n) (\pr(\pr_0A_0), A\pr_0A_0)+(\pr(\pr_0n), \A^2)(\R, \prb A,\pr A_0, \A^2)\\
&=&  \D_0\J_0+n^{-1}\nab n\,\J+n^{-1}\nab n\D_0\R_{0i\c\c}+\pr\Big(A(\R, \pr_0 A)+A_0\pr_0 A\Big)+\pr A_0 \pr_0 A\\
&&+(A, n^{-1}\nabla n) (\pr(\pr_0A_0), A\pr_0A_0)+(\pr(\pr_0n), \A^2)(\R, \prb A,\pr A_0, \A^2)\\
&=&  \D_0\J_0+n^{-1}\nab n \D_0\R_{0i\c\c}+\pr\Big(A(\R, \pr_0 A)+A_0\pr_0 A\Big)+\pr A_0 \pr_0 A\\
&&+(A, n^{-1}\nabla n) (\pr(\pr_0A_0), A\pr_0A_0)+(\pr(\pr_0n), \A^2)(\R, \prb A,\pr A_0, \A^2).
\eeaa
Exploiting the fact that $\J_\mu$ is divergence free in view of \eqref{eq:YM9}, we have
\beaa
\D_0\J_0 &=& \D_i\J_i\\
&=& \pr(\J)+A\J\\
&=& \pr\Big((A, n^{-1}\nab n)(\R, \prb \A)+A_0(\pr_0 A, \R)+\A^3\Big)+A^2\pr_0A_0+\A^2(\R, \prb A, \pr A_0, \A^2).
\eeaa
Also, the Einstein vacuum equations $\R_{\a\b}=0$ imply $\D^\mu \R_{\a\b\mu\nu}=0$ and hence
\beaa
n^{-1}\nab n\D_0\R_{0i\c\c} &=& -n^{-1}\nab n\D_j\R_{ji\c\c}\\
&=& \pr(n^{-1}\nab n\R)+\A^2\R+\pr A_0 \R.
\eeaa
We deduce
\beaa
\pr_0(\De ((A_0)_{jl})) &=&  \pr\Big((A, n^{-1}\nab n)(\R, \prb \A)+A_0(\pr_0 A, \R)+\A^3\Big)+\pr A_0 (\R, \pr_0 A)\\
&&+(A, n^{-1}\nabla n) (\pr(\pr_0A_0), A\pr_0A_0)+(\pr(\pr_0n), \A^2)(\R, \prb A,\pr A_0, \A^2).
\eeaa
On the other hand, we have in view of the commutation formula \eqref{commutdeltapr0}
\beaa
[\De, \pr_0](A_0) &=& 2k^{ab}\nabla_a\nabla_b(A_0)-2n^{-1}\nabla_bn\nabla_b(\pr_0(A_0))-n^{-1}\Delta n\pr_0A_0+2n^{-1}\nabla_an k^{ab}\nabla_bA_0
\eeaa
Together with the constraint equation \eqref{constk} for $k$ and the elliptic equation \eqref{eqlapsen} for $n$, we infer
\beaa
[\De, \pr_0](A_0) &=& 2\nab_a(k^{ab}\nabla_b(A_0))-2n^{-1}\nabla_bn\nabla_b(\pr_0(A_0))-|k|^2\pr_0A_0+2n^{-1}\nabla_an k^{ab}\nabla_bA_0\\
&=& \pr(A\pr A_0)+n^{-1}\nab n \, \pr(\pr_0A_0)+(A, n^{-1}\nab n)A\pr_0A_0+\A^2\pr A_0.
\eeaa
We finally obtain
\beaa
&&\De (\pr_0((A_0)_{jl}))\\
 &=& \pr_0(\De ((A_0)_{jl}))+[\De, \pr_0]((A_0)_{jl})\\
&=&  \pr\Big((A, n^{-1}\nab n)(\R, \prb \A)+A_0(\pr_0 A, \R)+\A^3\Big)+\pr A_0 (\R, \pr_0 A)\\
&&+(A, n^{-1}\nabla n) (\pr(\pr_0A_0), A\pr_0A_0)+(\pr(\pr_0n), \A^2)(\R, \prb A,\pr A_0, \A^2).
\eeaa
This concludes the proof of the lemma.
\end{proof} 
 
 We also record below  the following useful computation.
\begin{lemma}\lab{prop:projectionoperator} 
We have the following symbolic identity:
 \bea
\label{curlcurl}
\curl(\curl(A))_j= \pr_j (\div A)-\Delta  (A_j)+A\pr A.
\eea
 
\end{lemma}

\begin{proof}
To prove \eqref{curlcurl} we write, using the fact that $[\pr_i, \pr_j]=A\pr$ in view of \eqref{notcoordinate}, and the definition \eqref{def:div} of $\div$:
\beaa
\curl(\curl(A))_j&=&\in_{jli}\pr_l (\in_{imn}\pr_m(A_n))\\
&=&\in_{jli}\in_{imn}\pr_l (\pr_m(A_n))\\
&=&(\de_{jm}\de_{ln}-\de_{jn}\de_{lm})\pr_l (\pr_m(A_n))\\
&=&\pr_l(\pr_j(A_l))-\pr_l(\pr_l(A_j))\\
&=&\pr_j(\div A)-\lap (A_j)+A\pr A.
\eeaa
which is \eqref{curlcurl}. This concludes the proof of the lemma.
\end{proof}

\section{Preliminaries}\label{sec:preliminaries}

\subsection{The initial slice}\label{initial.slice}

By the assumptions of Theorem \ref{th:mainter}, we have:
\bea\lab{estinit6}
\norm{R}_{L^2(\Si_0)}\leq\ep,
\eea
\bea\lab{estinit0}
\norm{k}_{L^2(\Si_0)}+\norm{\nabla k}_{L^2(\Si_0)}\leq \ep,
\eea 
and:
\bea\lab{ikea}
r_{vol}(\Si_0,1)\geq \frac{1}{2}.
\eea
\eqref{estinit0}, \eqref{estinit6} and \eqref{ikea} together with the estimates in \cite{param3} (see section  4.4 in that paper) yields:
\bea\label{estinit7}
\norm{n-1}_{L^\infty(\Si_0)}+   \norm{\nabla n}_{L^2(\Si_0)}+     \norm{\nabla^2n}_{L^2(\Si_0)}\lesssim\ep.
\eea

Also, we record the following Sobolev embeddings and elliptic estimates on $\Si_0$ that where derived under the assumptions \eqref{estinit6} and \eqref{ikea} in \cite{param3} (see section 3.5 in that paper).
\begin{lemma}[Calculus inequalities on $\Si_0$ \cite{param3}]\lab{lemma:embeddingSi0}
Assume that \eqref{estinit6} and \eqref{ikea} hold. We have on $\Si_0$ the following Sobolev embedding for any tensor $F$:
\bea\lab{ikea1}
\norm{F}_{L^6(\Si_0)}\lesssim \norm{\nabla F}_{L^2(\Si_0)}.
\eea 
Also, we define the operator $(-\Delta)^{-\frac{1}{2}}$ acting on tensors on $\Si_0$ as:
$$(-\Delta)^{-\frac{1}{2}}F=\frac{1}{\Gamma\left(\frac{1}{4}\right)}\int_0^{+\infty}\tau^{-\frac{3}{4}}U(\tau)Fd\tau,$$
where $\Gamma$ is the Gamma function, and where $U(\tau)F$ is defined using the heat flow on $\Si_0$:
$$(\pr_\tau-\Delta)U(\tau)F=0,\, U(0)F=F.$$
We have the following Bochner estimates:
\bea\lab{ikea2}
\norm{\nabla(-\Delta)^{-\frac{1}{2}}}_{\LL(L^2(\Si_0))}\les 1\textrm{ and }\norm{\nabla^2(-\Delta)^{-1}}_{\LL(L^2(\Si_0))}\les 1,
\eea
where $\mathcal{L}(L^2(\Si_0))$ denotes the set of bounded linear operators on $L^2(\Si_0)$. \eqref{ikea2} together with the Sobolev embedding \eqref{ikea1} yields:
\bea\lab{ikea3}
\norm{(-\Delta)^{-\frac{1}{2}}F}_{L^2(\Si_0)}\lesssim \norm{F}_{L^{\frac{6}{5}}(\Si_0)}.
\eea 
\end{lemma}

We will also need the following elliptic estimate for div-curl systems on $\Sigma_0$.
\begin{lemma}\lab{lemma:divcurllemmainitialslice}
Let $X$ a vectorfield on $\Sigma_0$. Then, we have 
\beaa
\norm{X}_{L^2(\Sigma_0)}\les \norm{\div X}_{L^{\frac{6}{5}}(\Sigma_0)}+\norm{\curl X}_{L^{\frac{6}{5}}(\Sigma_0)}+\norm{X}_{L^6(\Sigma_0)} 
\eeaa
\end{lemma}

The proof Lemma \ref{lemma:divcurllemmainitialslice} is postponed to Appendix \ref{appendix:proofdivcurllemmainitialslice}.

\subsubsection{The Uhlenbeck type lemma}

In order to exhibit a frame $e_1, e_2, e_3$ such that together with $e_0=T$ we obtain a connection $\A$ satisfying our Coulomb type gauge on the initial slice $\Si_0$, we will need the following result in the spirit of the Uhlenbeck lemma\footnote{Note that our smallness  assumptions   on $\tilde{A}$      make the proof of the Lemma    much simpler than  the original result of Uhlenbeck.}     \cite{Uhl}.

\begin{lemma}\lab{lemma:uhlenbeck}
Let $(M,g)$ be a 3 dimensional Riemannian asymptotically flat manifold. Let $R$ denote its curvature tensor and $r_{vol}(M,1)$ its volume radius on scales $\leq 1$. Let\footnote{ In the context of this paper the lemma is applied  to  the  purely spatial part  of $\A$ i.e.   $  (A_i)_{jk}$.}  $\tilde{A}$    be   a connection on $M$ corresponding to an orthonormal frame $\tilde{e}_1,  \tilde{e}_2, \tilde{e}_3   $.  Assume the following bounds:
\bea
\|\tilde{A}\|_{L^2(M)}+ \|\nab \tilde{A}\|_{L^2(M) }+ \norm{R}_{L^2(M)}\leq \delta\qquad \textrm{ and}\,\, r_{vol}(M,1)\geq \frac{1}{4},  \lab{ass:tildeA}
\eea
where $\delta>0$ is a small enough constant. Assume also that $\tilde{A}$ and $\nabla\tilde{A}$ belong to $L^2(M)$. Then, there is another connection $A$ on $M$ satisfying he Coulomb  gauge condition $\nab^l (A_l)=0$, and such that
\bea
\| A\|_{L^2(M)}+ \|\nab  A\|_{L^2(M)} \les \de
\eea
 Furthermore, if $\nabla^2\tilde{A}$ belongs to $L^2(M)$, then $\nabla^2A$ belongs to $L^2(M)$.
\end{lemma}
\begin{proof}
This is a straightforward adaptation, in a simpler situation,  of   \cite{Uhl}.
Note   that   in  the new  frame $e_1, e_2, e_3$,  defined by  
$
e_i=O_i^j\,  \tilde{e}_j
$, 
  with  $O$ in the orthogonal group $O(3)$,  we have,
  \beaa
  A_m &=& O \tilde{A}_m O^{-1}+(\pr_m O) O^{-1}.
  \eeaa
Our   Coulomb gauge condition  leads to the elliptic equation for $O$,
  \bea
  \label{gauge2}
  \nab^m\big((\pr_m O) O^{-1}+ O \tilde{A}_m O^{-1}\big)=0, \qquad O \c O^t=I,
  \eea
which, in view of the smallness assumptions and the boundary condition $O\to 1$ at infinity along $M$, admits the unique solution. We leave the remaining details to the 
reader. 
\end{proof}

\subsubsection{Control of $A$, $A_0$ and $B=\Delta^{-1}\curl(A)$ on the initial slice}\lab{sec:bobo}

Let us first deduce from the Uhlenbeck type Lemma \ref{lemma:uhlenbeck} the existence of a connection $(A_i)_{jk}$ on $\Si_0$ satisfying the Coulomb gauge condition \eqref{eq:YM18}. In view of Theorem \ref{th:coordharm}, the bound on $R$ in $L^2(\Si_0)$ and on $r_{vol}(\Si_0,1)$ assumed in Theorem \ref{th:mainter} yields the existence of a system of harmonic coordinates. Furthermore, let $\tilde{e}_1, \tilde{e}_2, \tilde{e}_3$ an orthonormal frame obtained from $\pr_{x_1}, \pr_{x_2}, \pr_{x_3}$ by a standard orthonormalisation procedure, and let $\tilde{A}$ the corresponding connection. Then, the estimates of Theorem \ref{th:coordharm} yield the fact that $\tilde{A}$ and $\nabla\tilde{A}$ belong to $L^2(M)$. Together with the estimates \eqref{estinit6} on $R$ and \eqref{ikea} on $r_{vol}(\Si_0, 1)$, and the Uhlenbeck type Lemma \ref{lemma:uhlenbeck}, we obtain the existence of a connection $(A_i)_{jk}$ on $\Si_0$ satisfying the Coulomb gauge condition \eqref{eq:YM18}.

Next, using   both   the Coulomb gauge condition for   $(A_i)_{jk}$, see  \eqref{eq:YM18}, and    the   constraint equation  $\nabla^jk_{ij}=0$,   we deduce    \eqref{eq:YM19}.   Thus,    using also the estimates \eqref{estinit6} \eqref{estinit0} and the estimates of Lemma \ref{lemma:embeddingSi0} and Lemma \ref{lemma:divcurllemmainitialslice} on the initial slice $\Si_0$, we may estimate $A, A_0$ and $B=\Delta^{-1}\curl(A)$. This is done in the following proposition.

\begin{proposition}\lab{lemma:initialslice}
We have the following estimate for $A, A_0$ and $B=\Delta^{-1}\curl(A)$ on the initial slice $\Si_0$:
$$\norm{A}_{L^2(\Si_0)}+\norm{\prb A}_{L^2(\Si_0)}+\norm{\pr A_0}_{L^2(\Si_0)}+\norm{\pr B}_{L^2(\Si_0)}+\norm{\pr\prb B}_{L^2(\Si_0)}\les \ep.$$
\end{proposition}

\begin{remark}
In view of Proposition \ref{lemma:initialslice}, $A$ is in $L^2(\Si_0)$. Note that the corresponding estimate for $A_0$ is false and fortunately not needed. 
\end{remark}

\begin{proof}
We estimate separately the components $(A_i)_{j0}, (A_i)_{jl}, (A_0)_{i0}$ and $(A_0)_{ij}$. We start with $(A_i)_{j0}$. Recall that $(A_i)_{0j}=-k_{ij}$. Together with \eqref{estinit0}, we obtain:
\bea\lab{estinit1}
\norm{(\pr A)_{0j}}_{L^2(\Si_0)}\les \norm{\nabla k}_{L^2(\Si_0)}+\norm{A^2}_{L^2(\Si_0)}\lesssim \ep+\norm{A}^2_{L^4(\Si_0)}.
\eea
Also, $(A_i)_{jl}=g(\D_ie_j, e_l)=g(\nabla_ie_j,e_l)$. A computation similar to \eqref{linkRA} yields:
$$R(e_i, e_j , e_l, e_m)=\pr_l (\A_m)_{ij}-
\pr_m (\A_l)_{ij}+(\A_m)_i\,^n  (\A_l)_{nj}-
(\A_l)_i\,^n  (\A_m)_{nj}.$$
Thus, we have schematically:
$$(\curl A)_{ij}=R+A^2.$$
On the other hand, we have from the Coulomb gauge condition:
$$\div A=A^2.$$
Using \eqref{curlcurl}, we obtain, writing again schematically:
\bea\lab{estinit1bis}
(\Delta A)_{ij}=\nabla R+A\pr  A +A^3,
\eea
which after multiplication by $A_{ij}$ and integration by parts yields:
\bea\lab{estinit2}
\norm{(\pr A)_{ij}}^2_{L^2(\Si_0)}&\lesssim& (\norm{R}_{L^2(\Si_0)}+\norm{A}^2_{L^4(\Si_0)})\norm{\pr A}_{L^2(\Si_0)}+\norm{A}^4_{L^4(\Si_0)}\\
\nn&\les & (\ep+\norm{A}^2_{L^4(\Si_0)})\norm{\pr A}_{L^2(\Si_0)}+\norm{A}^4_{L^4(\Si_0)},
\eea
where we used \eqref{estinit6} in the last inequality. Now, recall $(A_i)_{00}=0$, which together with  \eqref{estinit1} and \eqref{estinit2} yields:
$$\norm{\pr A}_{L^2(\Si_0)}\lesssim \ep+\norm{A}^2_{L^4(\Si_0)}.$$
Together with the Sobolev embedding \eqref{ikea1}, this implies:
\bea\lab{estinit4:temporary}
\norm{\pr A}_{L^2(\Si_0)}\lesssim \ep+\norm{A}^2_{L^2(\Si_0)}+\norm{\pr A}^2_{L^2(\Si_0)}.
\eea

Next, we estimate $\norm{A}_{L^2(\Sigma_0)}$. Again, we estimate separately the components $(A_i)_{j0}, (A_i)_{jl}, (A_0)_{i0}$ and $(A_0)_{ij}$. As noticed above, we have
$$(\curl A)_{ij}=R+A^2,\,\,\div A=A^2.$$
Together with Lemma \ref{lemma:divcurllemmainitialslice}, this yields
\beaa
\norm{(A)_{ij}}_{L^2(\Sigma_0)}&\les& \norm{\div A}_{L^{\frac{6}{5}}(\Sigma_0)}+\norm{\curl A}_{L^{\frac{6}{5}}(\Sigma_0)}+\norm{A}_{L^6(\Sigma_0)}\\
&\les& \norm{\R}_{L^{\frac{6}{5}}(\Sigma_0)}+\norm{A^2}_{L^{\frac{6}{5}}(\Sigma_0)}+\norm{A}_{L^6(\Sigma_0)}\\ 
&\les& \norm{\R}_{L^{\frac{6}{5}}(\Sigma_0)}+\norm{A}^2_{L^2(\Sigma_0)}+\norm{A}^2_{L^6(\Sigma_0)}+\norm{A}_{L^6(\Sigma_0)}\\
&\les& \norm{\R}_{L^{\frac{6}{5}}(\Sigma_0)}+\norm{A}^2_{L^2(\Sigma_0)}+\norm{\pr A}^2_{L^2(\Sigma_0)}+\norm{\pr A}_{L^2(\Sigma_0)}
\eeaa
where we used the Sobolev embedding \eqref{ikea1} in the last inequality. In view of \eqref{estinit6} and the asymptotic behavior of $\R$ inherited from the asymptotic flatness of $\Sigma_0$, we have
$$\norm{\R}_{L^{\frac{6}{5}}(\Sigma_0)}\les\ep$$
and hence
\beaa
\norm{(A)_{ij}}_{L^2(\Sigma_0)} &\les& \ep+\norm{A}^2_{L^2(\Sigma_0)}+\norm{\pr A}^2_{L^2(\Sigma_0)}+\norm{\pr A}_{L^2(\Sigma_0)}.
\eeaa
In view of the fact that $(A_i)_{0j}=-k_{ij}$, $(A_i)_{00}=0$, and in view of the estimate \eqref{estinit0} for $k$, we obtain
\beaa
\norm{A}_{L^2(\Sigma_0)} &\les& \ep+\norm{A}^2_{L^2(\Sigma_0)}+\norm{\pr A}^2_{L^2(\Sigma_0)}+\norm{\pr A}_{L^2(\Sigma_0)}.
\eeaa
Together with \eqref{estinit4:temporary}, we infer
\beaa
\norm{A}_{L^2(\Si_0)}+\norm{\pr A}_{L^2(\Si_0)}\lesssim \ep+\norm{A}^2_{L^2(\Si_0)}+\norm{\pr A}^2_{L^2(\Si_0)}.
\eeaa
and thus 
\bea\lab{estinit4}
\norm{A}_{L^2(\Si_0)}+\norm{\pr A}_{L^2(\Si_0)}\lesssim \ep.
\eea

Next, we estimate $\nabla_0k$. Recall \eqref{eq:structfol1}:
$$\nabla_0k_{ab}= \R_{a\,0\,b\,0}-n^{-1}\nabla_a\nabla_bn-k_{ac}k_b\,^c.$$
Also recall Gauss equation \eqref{eq:structfol3}: 
$$\R_{a\,0\,b\,0}=R_{ab}-k_a\,^c k_{cb}.$$
Thus, we have:
\bea\lab{estinit5}
\nabla_0k= R-n^{-1}\nabla^2n+A^2.
\eea
\eqref{estinit6}, \eqref{estinit7}, \eqref{estinit4}, \eqref{estinit5} and the Sobolev embedding \eqref{ikea1} imply:
\bea\lab{estinit8}
\norm{\nabla_0k}_{L^2(\Si_0)}\les \ep.
\eea
Now, $(A_j)_{0i}=k_{ij}$, and thus:
$$(\pr_0A)_{0i}=\nabla_0k+AA_0,$$
which together with \eqref{estinit8}, \eqref{estinit4} and the Sobolev embedding \eqref{ikea1} yields:
\bea\lab{estinit8bis}
\norm{(\pr_0A)_{0i}}_{L^2(\Si_0)}&\les& \norm{\nabla_0k}_{L^2(S_0)}+ \norm{A}_{L^4(\Si_0)}\norm{A_0}_{L^4(\Si_0)}\\
\nn&\les& \ep+\ep\norm{\pr A_0}_{L^2(\Si_0)}.
\eea

Next, we estimate $(\pr_0A)_{ij}$. In view of \eqref{linkRA}, we have:
$$\R(e_i, e_j , e_0, e_l)=(\pr_0A_l)_{ij}-
(\pr_lA_0)_{ij}+A_0A.$$
Furthermore, we have:
$$\R_{0\,l\,i\,j}=(\pr_iA_j)_{0l}-(\pr_jA_i)_{0l}+A^2=\pr A+A^2.$$
Using the symmetry of the curvature tensor $\R_{i\,j\,0\,l}=\R_{0\,l\,i\,j}$, we obtain:
$$(\pr_0A_l)_{ij}=\pr A_0+\pr A+A\A,$$
which together with \eqref{estinit4} and the Sobolev embedding \eqref{ikea1} yields:
\bea\lab{estinit9}
\norm{(\pr_0A)_{ij}}_{L^2(\Si_0)}&\les& \norm{\pr A_0}_{L^2(\Si_0)}+\norm{\pr A}_{L^2(\Si_0)}+\norm{A}_{L^3(\Si_0)}\norm{\A}_{L^6(\Si_0)}\\
\nn&\les& \ep+\ep\norm{A_0}_{L^6(\Si_0)}\\
\nn&\les& \ep+\ep\norm{\pr A_0}_{L^2(\Si_0)}.
\eea
Since $A_{00}=0$, \eqref{estinit8bis} and \eqref{estinit9} yield:
\bea\lab{estinit11}
\norm{\pr_0A}_{L^2(\Si_0)}&\les& \ep+\ep\norm{\pr A_0}_{L^2(\Si_0)}.
\eea

Next, we estimate $A_0$. Recall \eqref{eq:YM21}:
\beaa
\De A_0=(A, n^{-1}\nab n) (\pr  \A+\pr_0 A)+ (A, n^{-1}\nab n) \c \A^2.
\eeaa
After multiplication by $A_0$ and integration by parts, and together with \eqref{estinit4}, \eqref{estinit11} and  the Sobolev embedding \eqref{ikea1}, 
this yields:
\beaa
\norm{\pr A_0}^2_{L^2(\Si_0)}&\les& (\norm{A}_{L^3(\Si_0)}+\norm{n^{-1}\nab n}_{L^3(\Si_0)}) (\norm{\pr  \A}_{L^2(\Si_0)}+\norm{\pr_0 A}_{L^2(\Si_0)})\norm{A_0}_{L^6(\Si_0)}\\
&&+ (\norm{A}_{L^2(\Si_0)}+\norm{n^{-1}\nab n}_{L^2(\Si_0)}) \norm{\A}^2_{L^6(\Si_0)}\norm{A_0}_{L^6(\Si_0)}\\
&\les& \ep (\ep+\norm{\pr A_0}_{L^2(\Si_0)}+\norm{\pr_0 A}_{L^2(\Si_0)}+\norm{A_0}^2_{L^6(\Si_0)})\norm{A_0}_{L^6(\Si_0)}\\
&\les& \ep (\ep+\norm{\pr A_0}_{L^2(\Si_0)}+\norm{\pr A_0}^2_{L^2(\Si_0)})\norm{\pr A_0}_{L^2(\Si_0)}
\eeaa
which implies:
\bea\lab{estinit12}
\norm{\pr A_0}_{L^2(\Si_0)}\les \ep.
\eea
Together with \eqref{estinit11}, we also obtain 
\bea\lab{estinit13}
\norm{\pr_0A}_{L^2(\Si_0)}\les \ep.
\eea

Finally, we estimate $B$ on the initial slice $\Si_0$ using the estimates for $\A$ \eqref{estinit4}, \eqref{estinit12} and \eqref{estinit13}. This will be done on $\Si_t$ in Proposition \ref{lemma:estB}. Arguing as in Proposition \ref{lemma:estB} for $t=0$ together with \eqref{estinit4}, \eqref{estinit12}, \eqref{estinit13},  the Sobolev embeddings \eqref{ikea1} and \eqref{ikea3} on $\Si_0$, the Bochner inequality on $\Si_0$ \eqref{ikea2}, we   immediately obtain:
$$\norm{\pr B}_{L^2(\Si_0)}+\norm{\pr\prb B}_{L^2(\Si_0)}\les \ep.$$
This concludes the proof of the proposition.
\end{proof}

\section{Strategy of the proof of theorem  \ref{th:mainter}}\label{sec:strategyoftheproof}

\subsection{Classical local existence}

We will need the following well-posedness result for the Cauchy problem for the Einstein equations \eqref{EVE} in the maximal foliation.

\begin{theorem}[Well-posedness for the Einstein equation in the maximal foliation]\lab{th:classicalwp}
Let $(\Si_0 ,g , k)$ be asymptotically flat and satisfying the constraint equations \eqref{const}, with $\textrm{Ric}$, $\nabla \textrm{Ric}$, $k$, $\nabla k$ and  $\nabla^2k$ in $L^2(\Si_0)$, and $r_{vol}(\Si_0,1)>0$. Then, there exists a unique asymptotically flat solution $(\MM, {\bf g})$ to the Einstein vacuum equations \eqref{EVE} corresponding to this initial data set, together with a maximal foliation by space-like hypersurfaces $\Si_t$ defined as level hypersurfaces of a time function $t$. Furthermore, there exists a time 
$$T_*=T_*(\norm{\nabla^{(l)}\textrm{Ric}}_{L^2(\Si_0)}, 0\leq l\leq 1, \norm{\nabla^{(j)}k}_{L^2(\Si_0)}, 0\leq j\leq 2, r_{vol}(\Si_0,1))>0$$
such that the maximal foliation exists for on $0\leq t\leq T_*$ with a corresponding control in $L^\infty_{[0,T_*]}L^2(\Si_t)$ for $\textrm{Ric}$, $\nabla\textrm{Ric}$, $k$, $\nabla k$ and $\nabla^2k$. 
\end{theorem}

Theorem \ref{th:classicalwp} requires two more derivatives both for $R$ and $k$ with respect to the main Theorem \ref{th:main}. Its proof is standard and relies solely on energy estimates (as opposed to Strichartz estimates of bilinear estimates). We refer the reader to \cite{ChKl} chapter 10 for a related statement.

\begin{remark} In the  proof  of  our main theorem, the result above will be used only in the context of an extension and continuity arguments (see \textit{Step 1} and \textit{Step 3} in section \ref{sec:proofconj}).
\end{remark}
\subsection{ Weakly regular  null  hypersurfaces}  We shall be working with  null hyper surfaces in $\MM$ verifying 
 a set of    assumptions, described below. These assumptions will be easily verified  by the level hyper surfaces  $\HH_u$   solutions $u$ of the eikonal equation $\g^{\mu\nu} \pr_\mu u\,\pr_\nu u=0$ discussed in  section \eqref{sec:parametrix}. The regularity of the eikonal  equation is  studied in detail  in \cite{param3}.

\begin{definition}[Weakly regular  null  hypersurfaces] 
Let $\HH$ be a  null hypersurface  with future  null    normal  $L$ verifying  ${\bf g}(L,T)=-1$. Let also $N=L-T$.  We denote by $\nabb$ the induced connection along the $2$-surfaces $\HH\cap\Si_t$. We say that $\HH$ is 
weakly regular provided that, 
\bea\lab{assumptionH1}
\norm{\D L}_{L^3(\HH)}+\norm{\D N}_{L^3(\HH)}\les 1,
\eea
and the following Sobolev embedding holds for any scalar function $f$ on $\HH$:
\bea\lab{assumptionH2}
\norm{f}_{L^6(\HH)}\les \norm{\nabb f}_{L^2(\HH)}+\norm{L(f)}_{L^2(\HH)}+\norm{f}_{L^2(\HH)}.
\eea
\end{definition}

\subsection{Main bootstrap assumptions}\lab{sec:mainbootassss}

Let $M\geq 1$ a large enough constant to be chosen later in terms only of universal constants. By choosing $\ep>0$ sufficiently small, we can also ensure $M\ep$ is small enough.     From now on, we assume the following bootstrap assumptions  hold true on a fixed  interval  $[0, T^*]$,  for some $0<T^*\le 1$.  Note that  $\HH$ denotes an arbitrary  weakly regular   null hypersurface,   with future directed normal  $L$,  normalized by the condition ${\bf g}(L,T)=-1$.

\begin{itemize}
\item   \textit{Bootstrap curvature assumptions}
\bea\label{bootR}
\norm{\R}_{\lsit{2}}\leq M\ep.
\eea
Also,
\bea\lab{bootcurvatureflux}
\norm{\R\c L}_{L^2(\HH)}\leq M\ep,
\eea
where $\R\c L$ denotes any component of $\R$ such 
that at least one index is contracted with $L$. 
\item \textit{Bootstrap  assumptions for the connection $\A$.}
We  also assume  that there exist  $\A=(A_0, A)$   verifying   our Coulomb type condition on  $ [0,T^*]$ ,   
such that,
\bea\label{bootA}
\norm{A}_{\lsit{2}}+\norm{\prb A}_{\lsit{2}}\leq M\ep,
\eea
and:
\bea\label{bootA0}
\nn\norm{A_0}_{L^\infty_t L^4(\Sigma_t)}+\norm{\pr A_0}_{\lsit{2}}+\norm{A_0}_{\lsitt{2}{\infty}}+\norm{\pr A_0}_{\lsit{3}}&&\\
+\norm{\pr\pr A_0}_{\lsit{\frac{3}{2}}}+\norm{\pr_0A_0}_{L^2_tL^{\frac{42}{13}}(\Sigma_t)}+\norm{\pr\pr_0A_0}_{L^2_tL^{\frac{14}{9}}(\Sigma_t)}&\leq& M\ep.
\eea
\end{itemize}

\begin{remark}
Compared to the bootstrap assumptions for $\pr A_0$, one would expect the following bootstrap assumptions for $\pr_0A_0$
$$ \norm{\pr_0 A_0}_{\lsit{3}},\,\,\norm{\pr\pr_0 A_0}_{\lsit{\frac{3}{2}}}$$
instead of the quantities
$$\norm{\pr_0A_0}_{L^2_tL^{\frac{42}{13}}(\Sigma_t)},\,\,\norm{\pr\pr_0A_0}_{L^2_tL^{\frac{14}{9}}(\Sigma_t)}$$
appearing in \eqref{bootA0}. However, such bounds could not be recovered as they correspond to the control the Laplace equation with a right-hand side in $L^1(\Si_t)$ which is borderline. Instead, relying on the elliptic equations for $\pr_0A_0$ in Proposition \ref{prop:equationpr0A0}, one could either derive bounds for 
$$ \norm{\pr_0 A_0}_{\lsit{3}+\lsit{p}},\,\,\norm{\pr\pr_0 A_0}_{\lsit{\frac{3}{2}}+\lsit{\frac{3p}{p+3}}},\,\, 3<p\leq 6$$
or with the help of the non sharp Strichartz estimate \eqref{bootstrich} bounds for
$$\norm{\pr_0A_0}_{L^2_tL^p(\Sigma_t)},\,\,\norm{\pr\pr_0A_0}_{L^2_tL^{\frac{3p}{p+3}}(\Sigma_t)},\,\, 3<p\leq\frac{42}{13}.$$
In \eqref{bootA0}, we choose for convenience the second type of bounds with $p=42/13$. This issue is purely technical and will require in estimate involving $\prb A_0$ to distinguish between $\pr A_0$ and $\pr_0A_0$. 
\end{remark}

\begin{remark}
Together with the estimates in \cite{param3} (see section 4.4 in that paper),  the bootstrap  assumption \eqref{bootR}
 yields:
\bea\label{bootk}
\norm{k}_{\lsit{2}}+\norm{\nabla k}_{\lsit{2}}\lesssim M\ep.
\eea
Furthermore, the bootstrap assumption \eqref{bootcurvatureflux} together with the estimates in \cite{param3} (see section 4.2 in that paper) yields:
\bea\lab{bootvolumeradius}
\inf_tr_{vol}(\Si_t,1)\geq \frac{1}{4}.
\eea
\end{remark}

In addition we make the following bilinear estimates assumptions for $\A$ and $\R$.:
\begin{itemize}
\item \textit{Bilinear assumptions I}.
Assume,
\bea\lab{bil1}
\norm{A^j\pr_jA}_{L^2(\MM)}\lesssim M^3\ep^2.
\eea
Also,  for  $B=(-\Delta)^{-1}\curl(A)$ (see   \eqref{defineB}  and  the  accompanying explanations):
\bea\lab{bil5}
\norm{ A^j\pr_j(\prb B )}_{L^2(\MM)}\lesssim M^3\ep^2,
\eea
and:
\bea\lab{bil6}
\norm{\R_{\c\,\c\,j\,0}\pr^jB}_{L^2(\MM)}\lesssim M^3\ep^2.
\eea
Finally,  for any weakly regular  null hypersurface $\HH$ and  any smooth  scalar function $\phi$ on $\MM$,
\bea\lab{bil8}
\norm{A_j\pr^j\phi}_{L^2(\MM)}\les M^2\ep\sup_{\HH}\norm{\nabb\phi}_{L^2(\HH)},
\eea
where the supremum is taken over all null hypersurfaces $\HH$. 
\begin{remark}
 Note that the  bootstrap assumption  \eqref{bil8}  applies   both to $(A_j)_{mn}$, with spatial external indices, 
as well as to  $  k_{j\,\cdot}=   - (A_j)_{ 0\,  \cdot } $  with a temporal  external  component.   
\end{remark}
\item  \textit{Bilinear assumptions II}. We assume,
\bea\lab{bil3}
\norm{(-\Delta)^{-\frac12}(Q_{ij}(A,A))}_{L^2(\MM)}\lesssim M^3\ep^2,
\eea
where the bilinear form $Q_{ij}$ is given by $Q_{ij}(\phi,\psi)=\pr_i\phi\pr_j\psi-\pr_j\phi\pr_i\psi$. Furthermore, we also have:
\bea\lab{bil4}
\norm{(-\Delta)^{-\frac12}(\pr(A^l)\pr_l(A))}_{L^2(\MM)}\lesssim M^3\ep^2.
\eea
\item \textit{Non-sharp Strichartz assumptions}
\bea\lab{bootstrich}
\norm{A}_{\lsitt{2}{7}}\les M^2 \ep.
\eea
and, for $B=(-\De)^{-1}\curl A$, (see   \eqref{defineB}  and  the  accompanying explanations):
\bea
\norm{\pr B}_{\lsitt{2}{7}}\lesssim M^2\ep. \lab{bootstrichB}
\eea
\begin{remark}
Note that the Strichartz estimate for  $\norm{A}_{\lsitt{2}{7}}$ is far from being sharp. Nevertheless, this estimate will be sufficient for the proof as it will only be used to deal with lower order terms.
\end{remark}

\end{itemize}
Finally we also need a trilinear bootstrap  assumption.  For  this we need to introduce the Bel-Robinson tensor,
 \bea\label{def:bellrobinson}
 Q_{\a\b\ga\de}=\R_\a\,^\la\,\ga\,^\si \R_{\b\,\la\,\de\,\si}+\dual \R_\a\,^\la\,\ga\,^\si\dual \R_{\b\,\la\,\de\,\si}
 \eea
\begin{itemize}
\item
\textit{Trilinear bootstrap assumption.} 
  We assume the following,
\bea
\left|\int_{\MM}Q_{ij\ga\de}k^{ij}e_0^\ga e_0^\de\right| &\les& M^4 \ep^3.\label{trilinearboot}
\eea
\end{itemize}
 We conclude this section by showing  that the  bootstrap assumptions are verified  for some positive $T^*$.
\begin{proposition}\lab{prop:continuity}
The above bootstrap assumptions are verified on $0\leq t\leq T^*$ for a sufficiently small $T^*>0$.
\end{proposition} 
 
\begin{proof}
The   only challenge  here   is to prove the existence of   the desired  connection $\A$, all  other  estimates    follow  trivially    from   our initial  bounds and    the local existence theorem above, for sufficiently small  $T^*$.  More precisely we need  
to exhibit a frame $e_1, e_2, e_3$ such that, together with $e_0=T$,  we obtain a connection $\A$ satisfying our Coulomb type gauge on the slice $\Si_t$. To achieve this we start   on $\Si_0$  with  the orthonormal  frame $e_1, e_2, e_3$,   discussed    in section \ref{initial.slice}\footnote{ such that  the corresponding  connection $\A$ verify the Coulomb gauge condition  \eqref{eq:YM18} and the estimates of  proposition \ref{lemma:initialslice}} and  transport  it 
to  an orthonormal   frame on $\Si_t$, $0\le t\le T^*$,  according to  the equation,
\beaa
\D_T(\tilde{e}_j)=0  , \qquad \tilde{e}_j(0)=e_j,\quad  j=1, 2, 3.
\eeaa
Differentiating, we obtain schematically the following transport equation for $\tilde{A}$:
$$\D_T(\tilde{A})=\R,              \qquad \tilde{A}(0)=A.$$
We can then rely on the estimates of the local existence theorem, for sufficiently small $T^*$,
to derive  $L^\infty_{[0, T^*]}L^2(\Si_t)$  bounds for $\tilde{A}$, $\prb\tilde{A}$ and
 $\prb^2\tilde{A}$.  Since all the bounds for $\tilde{A}$ and  $R$   are  controlled 
 from the initial data, for small $T^*$ (thus  proportional to $\ep$), we  are in a position to apply\footnote{Strictly speaking Uhlenbeck's lemma takes care   of   the purely spatial  part of $A$;  one needs to also use  the constraint equation  for  $k$  to derive
  \eqref{eq:YM19}.} Uhlenbeck's  lemma \ref{lemma:uhlenbeck} on $\Si_t$  to produce the desired  connection $A$. Furthermore,
  differentiating \eqref{gauge2} twice with respect to $\D_T$, and applying standard elliptic estimates,  we finally obtain the fact that $A$, $\prb A$ and $\prb^2A$ are  also controlled in  $L^\infty_{[0, T_*]}L^2(\Si_t)$ in conformity with our bootstrap  assumptions.
\end{proof} 
  
\subsection{Proof of the bounded $L^2$ curvature conjecture}\lab{sec:proofconj}

In the following two propositions, we state the improvement of our bootstrap assumptions. 
\begin{proposition}\label{prop:improve1}
Let us assume that all bootstrap assumptions of the previous section hold for $0\leq t\leq T^*$. If $\ep>0$ is sufficiently small, then the following improved estimates hold true on $0\leq t\leq T^*$:
\bea\label{bootRimp}
\norm{\R}_{\lsit{2}}\les \ep+M^2\ep^{\frac{3}{2}}+M^3\ep^2,
\eea
\bea\lab{bootcurvaturefluximp}
\norm{\R\c L}_{L^2(\HH)}\les \ep+M^2\ep^{\frac{3}{2}}+M^3\ep^2,
\eea
\bea\label{bootAimp}
\norm{A}_{\lsit{2}}+\norm{\prb A_i}_{\lsit{2}}\les \ep+M^2\ep^{\frac{3}{2}}+M^3\ep^2,
\eea
\bea\label{bootA0imp}
\nn\norm{A_0}_{L^\infty_t L^4(\Sigma_t)}+\norm{\pr A_0}_{\lsit{2}}+\norm{A_0}_{\lsitt{2}{\infty}}\\
\nn+\norm{\pr A_0}_{\lsit{3}}+\norm{\pr\pr A_0}_{\lsit{\frac{3}{2}}}+\norm{\pr_0A_0}_{L^2_tL^{\frac{42}{13}}(\Sigma_t)}\\
+\norm{\pr\pr_0A_0}_{L^2_tL^{\frac{14}{9}}(\Sigma_t)}&\les& \ep+M^2\ep^{\frac{3}{2}}+M^3\ep^2.
\eea
\end{proposition}

\begin{proposition}\label{prop:improve2}
Let us assume that all bootstrap assumptions of the previous section hold for $0\leq t\leq T^*$. If $\ep>0$ is sufficiently small, then the following improved estimates hold true on $0\leq t\leq T^*$:
\bea\lab{bil1imp}
\norm{A^j\pr_jA}_{L^2(\MM)}\lesssim M^2\ep^2,
\eea
\bea\lab{bil5imp}
\norm{A^j\pr_j(\prb B)}_{L^2(\MM)}\lesssim M^2\ep^2,
\eea
and
\bea\lab{bil6imp}
\norm{\R_{\c\,\c\,j\,0}\pr^jB}_{L^2(\MM)}\lesssim M^2\ep^2.
\eea
\bea\lab{bil8imp}
\norm{A^j\pr_j\phi}_{L^2(\MM)}\les M\ep\left(\sup_{\HH}\norm{\nabb\phi}_{L^2(\HH)}+\norm{\pr\phi}_{\lsit{2}}\right),
\eea
where the supremum is taken over all (weakly regular)  null hypersurfaces $\HH$. Finally, we have:
\bea\lab{bil3imp}
\norm{(-\Delta)^{-\frac12}(Q_{ij}(A,A))}_{L^2(\MM)}\lesssim M^2\ep^2,
\eea
\bea\lab{bil4imp}
\norm{(-\Delta)^{-\frac12}(\pr(A^l)\pr_l(A))}_{L^2(\MM)}&\lesssim& M^2\ep^2.
\eea
Also,
\bea\lab{bootstrichimp}
\norm{A}_{\lsitt{2}{7}}&\les &M \ep\\
\norm{\pr B}_{\lsitt{2}{7}}&\lesssim& M\ep. \lab{bootstrichBimp}
\eea
and
\bea
\left|\int_{\MM}Q_{ij\ga\de}k^{ij}e_0^\ga e_0^\de\right| &\les& M^3 \ep^3.\label{trilinearbootimp}
\eea
\end{proposition}

The proof of Proposition \ref{prop:improve1} is postponed to section \ref{sec:improve1}, while the proof of Proposition \ref{prop:improve2} is postponed to sections \ref{sec:improve2} and \ref{sec:improve2B}. We also need a proposition on the propagation of higher regularity.  

\begin{proposition}\lab{prop:propagreg}
Let us assume that the estimates corresponding to all bootstrap assumptions of the previous section hold for $0\leq t\leq T^*$ with a universal constant $M$. Then for any $t\in [0,T^*)$ and for $\ep>0$ sufficiently small, the following propagation of higher regularity holds:
$$
\|\D\R\|_{L^\infty_tL^2(\Sigma_t)}\le 2\left(\norm{\textrm{Ric}}_{L^2(\Sigma_0)}+\norm{\nabla\textrm{Ric}}_{L^2(\Sigma_0)}+\norm{k}_{L^2(\Sigma_0)}+\norm{\nabla k}_{L^2(\Sigma_0)}+\norm{\nabla^2k}_{L^2(\Sigma_0)}\right).$$
\end{proposition}

The proof of Proposition \ref{prop:propagreg} is postponed to section \ref{sect:propag}. Next, let us show how Propositions  \ref{prop:continuity}, \ref{prop:improve1}, \ref{prop:improve2} and \ref{prop:propagreg} imply our main theorem \ref{th:mainter}. We proceed, by the standard  bootstrap  method , along the following steps:
\begin{itemize}
\item[\textit{Step 1}.] We show that all bootstrap assumptions are verified  for a sufficiently small  final value $T^*$. 

\item[\textit{Step2}.]  Assuming that  all bootstrap assumptions hold  for   fixed values of $0< T^*\le 1$ and $M$  sufficiently large    we   show  that,  for $\ep>0$ sufficiently small, we may improve on the constant $M$ in our bootstrap assumptions. 

\item[\textit{Step 3}.] Using the  estimates derived in step 2   we can extend  the time of existence $T^*$ to $T^*+\de$  such that  all the bootstrap assumptions remain true.
\end{itemize}

Now, \textit{Step 1} follows from Proposition \ref{prop:continuity}. \textit{Step 2} follows from Proposition \ref{prop:improve1} and Proposition \ref{prop:improve2}. In view of \textit{Step 2}, the estimates corresponding to all bootstrap assumptions of the previous section hold for $0\leq t\leq T^*$ with a universal constant $M$. Thus the conclusion of Proposition \ref{prop:propagreg} holds, and arguing as in the proof of Proposition \ref{prop:continuity}, we obtain \textit{Step 3}. Thus, the bootstrap assumptions hold on $0\leq t\leq 1$ for a universal constant $M$. In particular, this yields together with \eqref{bootk}:
\bea\lab{saionara2}
\norm{\R}_{\lsit{2}}\les \ep\textrm{ and }\norm{k}_{\lsit{2}}\les \ep\textrm{ for all }0\leq t\leq 1.
\eea
In view of  \eqref{bootvolumeradius}, we also obtain the following control on the volume radius:
\bea\lab{bootvolumeradius:bis}
\inf_{0\leq t\leq 1}r_{vol}(\Si_t,1)\geq \frac{1}{4}.
\eea
Furthermore, Proposition \ref{prop:propagreg} yields the following propagation of higher regularity
\bea\lab{gggggggggggggg}
\sum_{|\a|\le m}\|\D^{(\a)} \R\|_{L^\infty_{[0,1]}L^2(\Sigma_t)}  \leq C_m
  \bigg[\|\nab^{(i)} \textrm{Ric}\|_{L^2(\Sigma_t)} + \|\nab^{(i)} \nab k\|_{L^2(\Sigma_t)}\bigg]
\eea
where $C_m$ only depends on $m$.

\begin{remark}
Note that Proposition \ref{prop:propagreg} only yields the case $m=1$ in \eqref{gggggggggggggg}. The fact that \eqref{gggggggggggggg} also holds for higher derivatives $m\geq 2$ follows from the standard propagation of regularity for the classical local existence result of Theorem \ref{th:classicalwp} and the bound \eqref{gggggggggggggg} with $m=1$ coming from Proposition \ref{prop:propagreg}. 
\end{remark}

Finally, \eqref{saionara2}, the control on the volume radius \eqref{bootvolumeradius:bis} and the propagation of higher regularity \eqref {gggggggggggggg} yield the conclusion of Theorem \ref{th:mainter}. Together with the reduction to small initial data performed in section \ref{sec:reductionsmall}, this concludes the proof of the main Theorem \ref{th:main}.\\

  The rest of the paper deals with the proofs of propositions \ref{prop:improve1}, \ref{prop:improve2} and \ref{prop:propagreg}. 
  The core of the proof  is to control $A$,  the spatial part of the  connection $\A$. 
  As explained     in  the introduction  we need to project  our equation for   the spatial components 
   $A$  onto divergence free vectorfields. This is needed for two reasons, to eliminate the term 
   $\pr_i (\pr_0 A_0)$ on the left  hand side of \eqref{eq:YM22} and to   obtain, on the right hand side,
   terms which exhibit the crucial null structure we need to implement our proof. Rather than work with 
   the projection $\PP$, which is too complicated, we      rely  instead   on   the new  variable,
   \bea
   B=(-\Delta)^{-1}\curl(A) \label{defineB}
   \eea
    for which  we derive a     wave equation. 
   Since we have (see Lemma \ref{recoverA}):
$$A=\curl(B)+l.o.t$$
 it suffices  to obtain estimates for $B$  which  lead  us  to an  improvement of the bootstrap assumption \eqref{bootA} on $A$.  In section \ref{sec:bobo5},   we derive space-time  estimates  for $\square B$ and its derivatives.     Proposition  \ref{prop:improve1}, which does not require  a parametrix representation, is  proved in \ref{sec:improve1}.   Proposition  \ref{prop:improve2} is proved 
 in sections \ref{sec:improve2} and \ref{sec:improve2B} based on the  representation formula  of theorem \ref{lemma:parametrixconstruction} derived in section \ref{sec:parametrix}. Finally, Proposition \ref{prop:propagreg} is proved in section \ref{sect:propag}.

\section{Simple  consequences of the bootstrap assumptions}\label{sec:simpleconsequencesboot}
 
 In this section, we discuss elliptic estimates on $\Si_t$, we derive estimates for $B$ from the bootstrap assumptions on $A$, and we show how to recover $A$ from $B$.

\subsection{Sobolev embeddings and elliptic estimates on $\Si_t$.}

First, we derive estimates for the lapse $n$ on $\Si_t$. The bootstrap assumption on $R$ \eqref{bootR} and the estimate for $k$ \eqref{bootk} together with the estimates in \cite{param3} (see section 4.4 in that paper) yield:
\bea\label{bootn}
\norm{n-1}_{L^\infty(\MM)}+\norm{\nabla n}_{\lsit{2}}+\norm{\nabla n}_{L^\infty(\MM)}+\norm{\nabla^2n}_{\lsit{2}}+\norm{\nabla^2n}_{\lsit{3}}&&\\
\nn+\norm{\nabla(\pr_0n)}_{\lsit{2}}+\norm{\nabla(\pr_0n)}_{\lsit{3}}+\norm{\nabla^3n}_{\lsit{\frac{3}{2}}}+\norm{\nabla^2(\pr_0n)}_{\lsit{\frac{3}{2}}}&\lesssim& M\ep.
\eea

\begin{remark}\lab{rem:tempura}
Recall from \eqref{eq:YM11} that:
$$(A_0)_{0i}=-n^{-1}\nabla_in.$$
Thus, the estimates \eqref{bootn} for $n$ could in principle be deduced from the bootstrap assumptions \eqref{bootA0} for $A_0$. However, notice that $\nabla n\in L^\infty(\MM)$ in view of \eqref{bootn}, while $A_0$ is only in $\lsitt{2}{\infty}$ according to \eqref{bootA0}. This improvement for the components $(A_0)_{0i}$ of $A_0$ will turn out to be crucial    and subtle\footnote{Using   the lapse equation   
$\Delta n=n|k|^2$ and  $k, \nabla k\in \lsit{2}$, see \eqref{bootk},  together  with  the Sobolev embedding \eqref{sobineqsit}  we only deduce $k\in \lsit{6}$ from which $\Delta n\in \lsit{3}$. This  would  yield $\nabla^2n\in\lsit{3}$, and thus $\nabla n$ misses  to be in $L^\infty(\MM)$ by a log divergence.  However, one can overcome this loss by exploiting the Besov improvement with respect to the Sobolev embedding \eqref{sobineqsit}. We refer the reader to  section 4.4 in \cite{param3} for the details.} (see remark \ref{rem:tempura1}).
\end{remark}

Next, we record   the following Sobolev embeddings and elliptic estimates on $\Si_t$ that where derived under the assumptions \eqref{bootcurvatureflux} and \eqref{bootR} in \cite{param3} (see sections 3.5 and 4.2 in that paper).
\begin{lemma}[Calculus inequalities on $\Si_t$ \cite{param3}]\lab{lemma:estimatesit}
Assume that the assumptions \eqref{bootcurvatureflux} and \eqref{bootR} hold, and assume that the volume radius at scales $\leq 1$ on $\Si_0$ is bounded from below by a universal constant.  
Let $\de>0$. Then, there exists $r_0(\de)>0$ and a finite covering of $\Si_t$ by geodesic balls  of radius $r_0(\de)$ such that each geodesic ball in the covering admits a system of harmonic coordinates $x=(x_1,x_2,x_3)$ relative to which we have
\bea\label{coorharmth1bis}
(1+\de)^{-1}\de_{ij}\leq g_{ij}\leq (1+\de)\de_{ij},
\eea
and
\bea\label{coorharmth2bis}
r_0(\de)\int_{B_{r_0}(p)}|\partial^2g_{ij}|^2\sqrt{|g|}dx\leq \de.
\eea

Furthermore, we have on $\Si_t$ the following estimates for any tensor $F$:
\bea\lab{eq:gnirenbergsit}
\norm{F}_{L^{3}(\Si_t)}\lesssim \norm{\nabla F}_{L^{\frac{3}{2}}(\Si_t)},
\eea 
\bea\lab{sobineqsit}
\norm{F}_{L^{6}(\Si_t)}\lesssim \norm{\nabla F}_{L^{2}(\Si_t)},
\eea 
\bea\lab{sobinftysit}
\norm{F}_{L^\infty(\Si_t)}\lesssim \norm{\nabla F}_{L^p(\Si_t)}+\norm{F}_{L^p(\Si_t)}\,\forall p>3,
\eea
and:
\bea\lab{prop:bochsit2}
\norm{\nabla^2F}_{L^{\frac{3}{2}}(\Si_t)}\lesssim \norm{\Delta F}_{L^{\frac{3}{2}}(\Si_t)}+\norm{\nabla F}_{L^2(\Si_t)}.
\eea

Finally, we define the operator $(-\Delta)^{-\frac{1}{2}}$ acting on tensors on $\Si_t$ as:
$$(-\Delta)^{-\frac{1}{2}}F=\frac{1}{\Gamma\left(\frac{1}{4}\right)}\int_0^{+\infty}\tau^{-\frac{3}{4}}U(\tau)Fd\tau,$$
where $\Gamma$ is the Gamma function, and where $U(\tau)F$ is defined using the heat flow on $\Si_t$:
$$(\pr_\tau-\Delta)U(\tau)F=0,\, U(0)F=F.$$
We have the following Bochner estimates:
\bea\lab{bochnerestimates}
\norm{\nabla(-\Delta)^{-\frac{1}{2}}}_{\LL(L^2(\Si_t))}\les 1\textrm{ and }\norm{\nabla^2(-\Delta)^{-1}}_{\LL(L^2(\Si_t))}\les 1,
\eea
where $\mathcal{L}(L^2(\Si_t))$ denotes the set of bounded linear operators on $L^2(\Si_t)$. \eqref{bochnerestimates} together with the Sobolev embedding \eqref{sobineqsit} yields:
\bea\lab{sob}
\norm{(-\Delta)^{-\frac{1}{2}}F}_{L^2(\Si_t)}\lesssim \norm{F}_{L^{\frac{6}{5}}(\Si_t)}.
\eea 
\end{lemma}

\begin{remark}\lab{covornotcov}
Note that $\pr^2f=\nabla^2f+A\pr f$ for any scalar function $f$ on $\Si_t$. Thus, in view of the bootstrap assumption \eqref{bootA} for $A$, we may replace $\nabla^2$ with $\pr^2$ in the Bochner inequality \eqref{bochnerestimates} when applied to a scalar function.
\end{remark}

\subsection{ Elliptic  estimates for $B$}
Here we derive  estimates for $B$ using  the bootstrap
 assumptions \eqref{bootA} \eqref{bootA0} for $A$ and $A_0$.

\begin{proposition}\lab{lemma:estB}
Let $B_i=(-\Delta)^{-1}(\curl(A)_i)$. Then, we have:
$$\norm{\pr(B_i)}_{\lsit{2}}+\norm{\pr^2(B_i)}_{\lsit{2}}+\norm{\pr(\pr_0(B_i))}_{\lsit{2}}\les M\ep.$$
\end{proposition}

\begin{proof}
Using the B\"ochner inequality on $\Si_t$ \eqref{bochnerestimates} together with Remark \ref{covornotcov}, and from the bootstrap assumption \eqref{bootA} on $A$, we have:
\bea\lab{estb1}
\norm{\pr(B_i)}_{\lsit{2}}+\norm{\pr^2(B_i)}_{\lsit{2}}\les \norm{A}_{\lsit{2}}+\norm{\pr A}_{\lsit{2}}\les M\ep.
\eea

Next, we estimate $\pr(\pr_0(B_i))$. In view of the definition of $B$, we have:
\beaa
\pr_0(B_i)&=&(-\Delta)^{-1}(\curl(\pr_0(A)))+[\pr_0,(-\Delta)^{-1}]\curl(A)+(-\Delta)^{-1}([\pr_0,\curl]A)\\
\nn&=&(-\Delta)^{-1}(\curl(\pr_0(A)))-(-\Delta)^{-1}[\pr_0,\Delta](-\Delta)^{-1}\curl(A)+(-\Delta)^{-1}([\pr_0,\curl]A)\\
\nn&=&(-\Delta)^{-1}(\curl(\pr_0(A)))-(-\Delta)^{-1}[\pr_0,\Delta]B+(-\Delta)^{-1}([\pr_0,\curl]A).
\eeaa
Thus, in view of the bootstrap assumption \eqref{bootA} for $A$, the Bochner inequality on $\Si_t$ \eqref{bochnerestimates} and the Sobolev embedding on $\Si_t$ \eqref{sob}, we have:
\bea\lab{estb2}
&&\norm{\pr\pr_0(B_i)}_{\lsit{2}}\\
\nn&\les& \norm{\pr(-\Delta)^{-1}(\curl(\pr_0(A)))}_{\lsit{2}}+\norm{\pr(-\Delta)^{-1}[\pr_0,\Delta]B}_{\lsit{2}}\\
\nn&&+\norm{\pr(-\Delta)^{-1}([\pr_0,\curl]A)}_{\lsit{2}}\\
\nn&\les& \norm{\pr_0A}_{\lsit{2}}+\norm{(-\Delta)^{-\frac12}[\pr_0,\Delta]B}_{\lsit{2}}+\norm{(-\Delta)^{-\frac12}([\pr_0,\curl]A)}_{\lsit{2}}\\
\nn&\les& M\ep+\norm{(-\Delta)^{-\frac12}[\pr_0,\Delta]B}_{\lsit{2}}+\norm{(-\Delta)^{-\frac12}([\pr_0,\curl]A)}_{\lsit{2}}\\
\nn&\les& M\ep+\norm{[\pr_0,\Delta]B}_{\lsit{\frac{6}{5}}}+\norm{(-\Delta)^{-\frac12}([\pr_0,\curl]A)}_{\lsit{2}}.
\eea

Next, we estimate the right-hand side of \eqref{estb2}. Recall the commutator formula 
\eqref{commutdeltapr0}
\beaa
[\pr_0,\Delta](B_l)=-2k^{ab}\nabla_a\nabla_b(B_l)+2n^{-1}\nabla_bn\nabla_b(\pr_0(B_l))+n^{-1}\Delta n\pr_0(B_l)-2n^{-1}\nabla_an k^{ab}\nabla_b(B_l).
\eeaa
Together with the Sobolev embedding on $\Si_t$ \eqref{sobineqsit}, the estimate \eqref{bootk} for $k$, the estimate \eqref{bootn} for $n$, and the estimate \eqref{estb1} for $B_i$, this yields:
\bea\lab{estb3}
&&\norm{[\pr_0,\Delta](B_i)}_{\lsit{\frac{6}{5}}}\\
\nn&\les & \norm{k}_{\lsit{3}}\norm{\pr^2(B_i)}_{\lsit{2}}+\norm{\nabla n}_{\lsit{3}}\norm{\pr(\pr_0(B_i))}_{\lsit{2}}\\
\nn &&+\norm{\Delta n}_{\lsit{\frac{3}{2}}}\norm{\pr_0(B_i)}_{\lsit{6}}+\norm{\nabla n}_{\lsit{2}}\norm{k}_{\lsit{6}}\norm{\pr(B_i)}_{\lsit{6}}\\
\nn &\les & M^2\ep^2+M\ep\norm{\pr(\pr_0(B_i))}_{\lsit{2}}.
\eea

Next, we estimate the last term in the right-hand side of \eqref{estb2}. In view of the commutator formulas \eqref{commutnablapr0:s} and \eqref{notcoordinate}, and in view of the definition of $\curl$, we have schematically:
$$[\pr_0,\curl]A=k\nabla A+n^{-1}\nabla n\pr_0A+\A\pr A=(A, n^{-1}\nab n)(\prb A, \pr A_0)+\pr(A_0A),$$
which together with the Bochner inequality on $\Si_t$ \eqref{bochnerestimates}, the Sobolev embedding on $\Si_t$ \eqref{sob}, the bootstrap assumption \eqref{bootA} for $A$ and the bootstrap assumption \eqref{bootA0} for $A_0$ yields:
\bea\lab{estb4}
&&\norm{(-\Delta)^{-\frac12}([\pr_0,\curl]A)}_{\lsit{2}}\\
\nn &\les& \norm{(A, n^{-1}\nab n)(\prb A, \pr A_0)}_{\lsit{\frac{6}{5}}}+\norm{A_0A}_{\lsit{2}}\\
\nn&\les& (\norm{A}_{\lsit{3}}+\norm{n^{-1}\nab n}_{\lsit{3}})(\norm{\prb A}_{\lsit{2}}+\norm{\pr A_0}_{\lsit{2}})\\
\nn&&+\norm{A_0}_{\lsit{4}}\norm{A}_{\lsit{4}}\\
\nn&\les& M^2\ep^2.
\eea
Finally, \eqref{estb2}-\eqref{estb4} imply:
\beaa
\norm{\pr\pr_0(B_i)}_{\lsit{2}}&\les& M\ep+\norm{[\pr_0,\Delta]B}_{\lsit{\frac{6}{5}}}+\norm{(-\Delta)^{-\frac12}([\pr_0,\curl]A)}_{\lsit{2}}\\
\nn&\les& M\ep+M\ep\norm{\pr(\pr_0(B_i))}_{\lsit{2}}
\eeaa
which yields:
$$\norm{\pr\pr_0(B_i)}_{\lsit{2}}\les M\ep.$$
Together with \eqref{estb1}, this concludes the proof of the proposition.
\end{proof}

\subsection{A decomposition for $A$}

Recall that $B=(-\Delta)^{-1}(\curl(A))$. We show how to recover $A$ from $B$:
\begin{lemma}\lab{recoverA}
We have the following estimate:
$$A=\curl(B)+E$$
where $E$ satisfies:
$$\norm{\pr E}_{\lsit{2}} +\norm{\pr E}_{\lsit{3}}+\norm{\pr^2E}_{\lsit{\frac{3}{2}}}+\norm{E}_{\lsitt{2}{\infty}}\les M^2\ep^2.$$
\end{lemma}

\begin{proof}
In view of Lemma \ref{prop:projectionoperator}, we have:
\beaa
 A=(-\Delta)^{-1}\curl(\curl(A)+(-\Delta)^{-1} (A\pr  A+A^3).
\eeaa
This yields:
\beaa
A &=& \curl(-\Delta)^{-1}(\curl(A))+[(-\Delta)^{-1},\curl]\curl(A)+(-\Delta)^{-1} (A \pr A+A^3)\\
&=& \curl(B)-(-\Delta)^{-1}[\Delta,\curl](-\Delta)^{-1}\curl(A)+(-\Delta)^{-1} (A \pr  A+A^3)\\
&=& \curl(B)-(-\Delta)^{-1}[\Delta,\curl]B+(-\Delta)^{-1} (A \pr  A+A^3),
\eeaa
which implies:
$$E=-(-\Delta)^{-1}[\Delta,\curl]B+(-\Delta)^{-1} (A \pr  A+A^3).$$
Now, we have 
$$[\Delta,\pr]\phi=R\pr \phi+\pr A\,\pr\phi+A\,\pr^2\phi$$
for any scalar function $\phi$ in $\Si_t$ where the curvature tensor $R$ on $\Si_t$ is related to $\R$ through the Gauss equation which can be written schematically: 
$$R=\R+A^2.$$
Thus, we obtain:
$$[\Delta,\curl]B=\R\pr B+\pr A\pr B+A\pr^2 B+A^2\pr B.$$
This yields:
\bea\lab{eel1}
E=-(-\Delta)^{-1}(\R\pr B+\pr A\pr B+A\pr^2B+A^2\pr B)+(-\Delta)^{-1} (A \pr  A+A^3).
\eea

In view \eqref{eel1}, and using the elliptic estimate \eqref{bochnerestimates} and the Sobolev embedding \eqref{sob} together with $\nab(-\De)^{-1}=\nab(-\De)^{-\frac{1}{2}}(-\De)^{-\frac{1}{2}}$, we have
\bea\lab{estmiatebasicL2E}
\norm{\pr E}_{\lsit{2}} 
\nn&\les& \norm{\R\pr B}_{\lsit{\frac{6}{5}}}+\norm{\pr B \pr A}_{\lsit{\frac{6}{5}}}+\norm{A\pr^2B}_{\lsit{\frac{6}{5}}}\\
\nn&&+\norm{A^2\pr B}_{\lsit{\frac{6}{5}}}+\norm{A\pr A}_{\lsit{\frac{6}{5}}}+\norm{A^3}_{\lsit{\frac{6}{5}}}\\
\nn&\les& \norm{\R}_{\lsit{2}}\norm{\pr B}_{\lsit{3}}+\norm{A}_{\lsit{3}}\norm{\pr^2B}_{\lsit{2}}\\
\nn&&+(\norm{A}_{\lsit{3}}+\norm{\pr B}_{\lsit{3}})\norm{\pr A}_{\lsit{2}}\\
\nn&&+\norm{A}^2_{\lsit{6}}(\norm{\pr B}_{\lsit{3}}+\norm{A}_{\lsit{3}})\\
\nn&\les & M^2\ep^2,
\eea
where in the last step we interpolate $L^3$ between $L^2$ and $L^6$ to estimate $\norm{A}_{\lsit{3}}$, $\norm{\pr B}_{\lsit{3}}$ and use
\eqref{bootA} and Proposition \ref{lemma:estB}.

Also, we have
\beaa
\norm{\Delta E}_{\lsit{\frac{3}{2}}}&\les& \norm{\R\pr B}_{\lsit{\frac{3}{2}}}+\norm{\pr B \pr A}_{\lsit{\frac{3}{2}}}+\norm{A\pr^2B}_{\lsit{\frac{3}{2}}}\\
&&+\norm{A^2\pr B}_{\lsit{\frac{3}{2}}}+\norm{A\pr A}_{\lsit{\frac{3}{2}}}+\norm{A^3}_{\lsit{\frac{3}{2}}}\\
\nn&\les& \norm{\R}_{\lsit{2}}\norm{\pr B}_{\lsit{6}}+\norm{A}_{\lsit{6}}\norm{\pr^2B}_{\lsit{2}}\\
&&+(\norm{A}_{\lsit{6}}+\norm{\pr B}_{\lsit{6}})\norm{\pr A}_{\lsit{2}}\\
&&+\norm{A}^2_{\lsit{6}}(\norm{\pr B}_{\lsit{6}}+\norm{A}_{\lsit{6}})\\
\nn&\les & M^2\ep^2,
\eeaa
where we used in the last inequality the Sobolev embedding \eqref{sobineqsit} on $\Si_t$, the bootstrap estimates \eqref{bootA} for $A$, the bootstrap estimate \eqref{bootR} for $\R$ and the estimates \eqref{estb1} for $B$. Using the elliptic estimate  \eqref{prop:bochsit2}  on $\Si_t$ and \eqref{estmiatebasicL2E}, we deduce:
\beaa
\norm{\pr^2E}_{\lsit{\frac{3}{2}}}&\les& \norm{\Delta E}_{\lsit{\frac{3}{2}}}+\norm{\pr E}_{\lsit{2}}\\
&\les& M^2\ep^2.
\eeaa
Together with the Sobolev embedding \eqref{eq:gnirenbergsit} on $\Si_t$, we finally obtain:
$$\norm{\pr E}_{\lsit{3}}+\norm{\pr^2E}_{\lsit{\frac{3}{2}}}\les M^2\ep^2.$$

Next, we estimate $\norm{E}_{\lsitt{2}{\infty}}$. We first claim the following non sharp embedding on $\Si_t$. For any scalar function $v$ on $\Si_t$, we have:
\bea\lab{eel3}
\norm{(-\Delta)^{-1}v}_{\lt{\infty}}\les \norm{v}_{\lt{\frac{14}{9}}}+\norm{v}_{\lt{\frac{13}{9}}}.
\eea
The proof of \eqref{eel3} requires the use of Littlewood-Paley projections on $\Si_t$ and is postponed to Appendix A. We now come back to the estimate of $\norm{E}_{\lsitt{2}{\infty}}$. Using \eqref{eel1} and \eqref{eel3}, we have:
\beaa
&&\norm{E}_{\lsitt{2}{\infty}}\\
&\les& \norm{\R\pr B}_{\lsitt{2}{\frac{14}{9}}}+\norm{A\pr A}_{\lsitt{2}{\frac{14}{9}}}+\norm{A^3}_{\lsitt{2}{\frac{14}{9}}}\\
&&+\norm{\R\pr B}_{\lsitt{2}{\frac{13}{9}}}+\norm{A\pr A}_{\lsitt{2}{\frac{13}{9}}}+\norm{A^3}_{\lsitt{2}{\frac{13}{9}}}\\
&\les& \norm{\R}_{\lsit{2}}(\norm{\pr B}_{\lsitt{2}{7}}+\norm{\pr B}_{\lsit{\frac{26}{5}}})\\
&&+(\norm{A}_{\lsitt{2}{7}}+\norm{A}_{\lsit{\frac{26}{5}}})\norm{\pr A}_{\lsit{2}}+\norm{A}^3_{\lsit{\frac{14}{3}}}+\norm{A}^3_{\lsit{\frac{13}{3}}}\\
&\les & M^2\ep^2,
\eeaa
where we used in the last inequality the bootstrap assumptions \eqref{bootA} for $A$, the 
bootstrap assumption \eqref{bootR} for $\R$, the  bootstrap   Strichartz estimate for $B, see  $ \eqref{bootstrich}, and Proposition \ref{lemma:estB}. This concludes the proof of the lemma.
\end{proof}

\section{Estimates  for $\square B $}\lab{sec:bobo5}

The goal of this section is to derive estimates  for  $\square B$, with   $B=\Delta^{-1}\curl(A)$ using the wave equation \eqref{eq:YM22} satisfied by each component of $A_i$.
 We  provide  the proof of two important  propositions concerning  estimates  for $\square \curl A$    and $\square  B$,  with   $B=\Delta^{-1}\curl(A)$. The proofs makes use of the special structure of various  bilinear expressions and thus  is based    not only on the bootstrap assumptions for $A_0, A$, $k$ and $\R$  but also  some of our  bilinear bootstrap assumptions.

We  will  need the  following straightforward commutation lemma. 
\begin{lemma}\lab{lemma:commutation}
Let $\phi$ be a $so(3,1)$  scalar function on $\MM$. We have, schematically,
\bea\lab{commsquare2}
\pr_j(\square\phi)-\square (\pr_j\phi)=
2(\A^\la)_j\,^{\mu}\,\, \pr_\la   \pr_\mu\phi+\pr_0(n^{-1}\nab n)\pr_0\phi +\pr_0A_0\pr\phi + \A^2\prb\phi.
\eea
We also have:
\bea\lab{commsquare3}
[\square, \Delta]\phi
&=&-4k^{ab}\nabla_a\nabla_b(\pr_0\phi)+4n^{-1}\nabla_bn\nabla_b(\pr_0(\pr_0\phi))
 -2\nabla_0k^{ab}\nabla_a\nabla_b\phi\\
 &+&F^{(1)}\prb^2\phi+F^{(2)}\prb\phi,\nn\\
 F^{(1)}&=& n^{-1}\nab(\pr_0n)+\pr A_0+\A^2,\nn\\
 F^{(2)}&=& n^{-1}\nab^2(\pr_0n)+\pr\pr A_0+\A\Big(n^{-1}\nab(\pr_0n),\R, \pr\A, \prb A\Big)+\A^3,\nn
\eea
where $\nabla_a$ and $\nabla_b$ denote induced covariant derivatives on $\Si_t$ applied to the scalars $\phi$, $\pr_0\phi$ and $\pr_0(\pr_0\phi)$.
\end{lemma}
\begin{remark}
The derivation of \eqref{commsquare3} involves the full use of the Einstein equations. The resulting structure of the terms on the right hand 
side of  \eqref{commsquare3} is crucial to the strategy of the proof of the main result and reflects ``hidden" null cancellations. We refer the reader to \eqref{csd6}-\eqref{csd13}, where this structure allows us to use bilinear estimates.
\end{remark}

\begin{proof}
 We start with the following general covariant calculation for any scalar function $\phi$ on $\MM$:
\bea\lab{commsquare1}
[\D_\mu,\square]\phi=0.
\eea
This follows trivially from 
 the vanishing of the  spacetime Ricci curvature, i.e.
$$[\D_\mu,\square]\phi=\R_{\nu\,\mu}\,^\nu\,^\la \D_\la\phi=0.$$
On the other hand, Lemma \ref{lemma:scalarize} yields:
\beaa
(e_j)^\mu\square(\D\phi)_\mu=\square(\pr_j\phi)-2(\A^\la)_j\,^{\mu}\,\, \D_\la   \pr_\mu\phi-\D^\la(\A_\la)_j\,^\ga\pr_\ga\phi-
(\A_\lambda)_j\, ^\gamma (\A^\lambda)_\gamma\, ^\sigma \pa_\sigma\phi.
\eeaa
In view of our Coulomb like gauge condition, we obtain for $\pr_j\phi $, $j=1,2,3$: 
$$\pr_j(\square\phi)-\square (\pr_j\phi)=
2(\A^\la)_j\,^{\mu}\,\, \pr_\la   \pr_\mu\phi+ \pr_0 (A_0)^\ga\,_j \, \pr_\ga \phi 
+ \A^2\prb\phi,$$
which together with the identity \eqref{eq:YM11} proves the first part  of the lemma. The proof of  the second part of Lemma \ref{lemma:commutation} is postponed to Appendix C.
\end{proof}

\subsection{Estimates  for $\square\,  \curl(A)$}

\begin{proposition}\lab{prop:commutsquarecurl}
The following estimate holds true, 
$$\sum_{i=1}^3\norm{(-\Delta)^{-\frac12}      \square(\curl(A)_i)        }_{L^2(\MM)}\lesssim M^2\ep^2.$$
\end{proposition}

\begin{proof}
We have:
\bea\lab{csc1}
\square(\pr_j(A_i)-\pr_i(A_j))=\pr_j(\square(A_i))-\pr_i(\square(A_j))+[\square,\pr_j](A_i)-[\square,\pr_i](A_j).
\eea
We evaluate the first term on the right-hand side of \eqref{csc1} by differentiating \eqref{eq:YM22}. We obtain:
\bea\lab{csc2}
\pr_j(\square(A_i))=-\pr_j(\pr_i (\pr_0 A_0))+\pr_j(A^l \pr_i A_l)+\pr_j(h^{(1)}_i),
\eea
where $h^{(1)}_i$ is given by:
$$h^{(1)}_i=A^l\pr_lA_i+A_0\prb\A+A\prb A_0+\A^3=A^l\pr_lA_i+A_0(\prb A, \pr A_0)+\A\pr_0A_0+A\pr A_0+\A^3.$$
We estimate $h^{(1)}_i$ using the bootstrap assumptions \eqref{bootA} and \eqref{bootA0} for $A$ and $A_0$, the Sobolev embedding on $\Si_t$ \eqref{sob}, and the Bochner inequality \eqref{bochnerestimates} on $\Si_t$:
\bea\lab{csc3}
&&\norm{(-\Delta)^{-\frac12}\pr_j(h^{(1)}_i)}_{L^2(\MM)}\\
\nn&\les& \norm{h^{(1)}_i}_{L^2(\MM)}\\
\nn&\les& \norm{A^l\pr_l(A_i)}_{L^2(\MM)}+\norm{A_0}_{\lsitt{2}{\infty}}(\norm{\prb A}_{\lsit{2}}+\norm{\pr A_0}_{\lsit{2}})\\
\nn&&+\norm{\A}_{\lsit{\frac{21}{4}}}\norm{\pr_0A_0}_{L^2_tL^{\frac{42}{13}}(\Sigma_t)}+\norm{A}_{\lsit{6}}\norm{\pr A_0}_{\lsit{3}}+\norm{\A}^3_{\lsit{6}}\\
\nn&\les& M^2\ep^2.
\eea
In view of \eqref{csc2}, we have:
\bea\lab{csc4}
\pr_j(\square(A_i))-\pr_i(\square(A_j))&=& -\pr_j(\pr_i (\pr_0 A_0))+\pr_i(\pr_j (\pr_0 A_0))\\
\nn&& +\pr_j(A^l \pr_i A_l)-\pr_i(A^l \pr_j A_l)+\pr_j(h^{(1)}_i)-\pr_i(h^{(1)}_j)\\
\nn&=& h^{(2)}_{ij},
\eea
where $h^{(2)}_{ij}$ is given by:
$$h^{(2)}_{ij}=Q_{ij}(A^l,A_l)+ A\pr(\pr_0A_0)+A^2(\prb A, \pr\A)+A^2\pr_0A_0+\pr_j(h^{(1)}_i)-\pr_i(h^{(1)}_j),$$
and where the quadratic form $Q_{ij}$ is defined as $Q_{ij}(\phi, \psi)=\pr_i\phi\pr_j\psi-\pr_j\phi\pr_i\psi$. Note that the most dangerous term in $h^{(2)}_{ij}$ is $Q_{ij}(A^l,A_l)$. Using the bilinear assumption  \eqref{bil3}, the Sobolev embeddings on $\Si_t$ \eqref{sob} and \eqref{sobineqsit}, the bootstrap assumptions \eqref{bootA} and \eqref{bootA0} for $A$ and $A_0$, and the estimate \eqref{csc3}, we have:
\bea\lab{csc5}
&&\norm{(-\Delta)^{-\frac12}(h^{(2)}_{ij})}_{L^2(\MM)}\\
\nn&\les&  \norm{(-\Delta)^{-\frac12}(Q_{ij}(A^l,A_l))}_{L^2(\MM)}+\norm{A\pr(\pr_0A_0)}_{\lsitt{2}{\frac{6}{5}}}+\norm{A^2(\prb A, \pr A_0)}_{\lsitt{2}{\frac{6}{5}}}\\
\nn&&+\norm{A^2\pr_0A_0}_{\lsitt{2}{\frac{6}{5}}}+\norm{(-\Delta)^{-\frac12}\pr_j(h^{(1)}_i)}_{L^2(\MM)}+\norm{(-\Delta)^{-\frac12}\pr_i(h^{(1)}_j)}_{L^2(\MM)}\\
\nn&\les& M^3\ep^2+\norm{A}_{\lsit{\frac{21}{4}}}\norm{\pr(\pr_0A_0)}_{\lsitt{2}{\frac{14}{9}}}\\
\nn&&+\norm{A}^2_{\lsit{6}}(\norm{\prb A}_{\lsit{2}}+\norm{\pr A_0}_{\lsit{2}})+\norm{A}^2_{\lsit{\frac{42}{11}}}\norm{\pr_0A_0}_{\lsitt{2}{\frac{42}{13}}}+M^2\ep^2\\
\nn&\les& M^3\ep^2.
\eea

Next, we consider the commutator terms in the right-hand side of \eqref{csc1}. In view of \eqref{commsquare2}, we have:
\bea\lab{csc6}
[\square,\pr_j](A_i)=2(A^\la)_j\,^{\mu}\,\, \pr_\la   \pr_\mu(A_i)+h^{(3)}_{ij},
\eea
where $h^{(3)}_{ij}$ is given by:
\beaa
h^{(3)}_{ij} &=& \pr_0(n^{-1}\nab n)\pr_0 A+\pr_0 A_0\pr A+ \A^2\prb A\\
&=& \pr_0(n^{-1}\nab n)\pr_0A+\pr(A\pr_0 A_0)+A_i\pr\pr_0A_0+ \A^2\prb A.
\eeaa
Using the Sobolev embeddings on $\Si_t$ \eqref{sob} and \eqref{sobineqsit}, the Bochner inequality \eqref{bochnerestimates} on $\Si_t$, the estimate \eqref{bootn} for $n$, and the bootstrap assumptions \eqref{bootA} and \eqref{bootA0} for $A$ and $A_0$, we have:
\bea\lab{csc7}
&&\norm{(-\Delta)^{-\frac12}(h^{(3)}_{ij})}_{L^2(\MM)}\\
\nn&\les&  \norm{\pr_0(n^{-1}\nab n)\pr_0A}_{\lsit{\frac{6}{5}}}+\norm{A\pr_0A_0}_{L^2(\MM)}+\norm{A\pr\pr_0A_0}_{\lsit{\frac{6}{5}}}+\norm{\A^2\prb A}_{\lsit{\frac{6}{5}}}\\
\nn&\les& \norm{\pr_0(n^{-1}\nab n)}_{\lsit{3}}\norm{\pr_0A}_{\lsit{2}}+\norm{A}_{\lsit{\frac{21}{4}}}\norm{\pr_0A_0}_{\lsitt{2}{\frac{42}{13}}}\\
\nn&&+\norm{A}_{\lsit{\frac{21}{4}}}\norm{\pr\pr_0A_0}_{\lsitt{2}{\frac{14}{9}}}+\norm{\A}^2_{\lsit{6}}\norm{\prb A}_{\lsit{2}}\\
\nn&\les& M^2\ep^2.
\eea

Next, we consider the term $(A^\la)_j\,^{\mu}\,\, \pr_\la   \pr_\mu(A_i)$. We have:
\bea\lab{csc8}
(A^\la)_j\,^{\mu}\,\, \pr_\la   \pr_\mu(A_i)&=&-(A_0)_j\,^l\,\, \pr_0\pr_l(A_i)+(A_0)_{j\,0}\,\, \pr_0\pr_0(A_i)\\
\nn &+& (A^l)_j\,^m\,\, \pr_l \pr_m(A_i)-(A^l)_{j\,0}\,\, \pr_l \pr_0(A_i)+{A}^2\prb A.
\eea
Note that the most dangerous terms in \eqref{csc8} are the third and the fourth one. They will both  require the use of bilinear estimates.

We deal with each term in the right-hand side of \eqref{csc8}, starting with the first one. We have:
\beaa
(A_0)_j\,^l\,\, \pr_0\pr_l(A_i)&=& (A_0)_j\,^l\,\, \pr_l(\pr_0(A_i))+\A^2\prb A\\
&=& \pr_l (A_0\prb A)+\pr_l(A_0)\prb A+\A^2\prb A,
\eeaa
which together with the Sobolev embeddings on $\Si_t$ \eqref{sob} and \eqref{sobineqsit}, and the bootstrap assumptions \eqref{bootA} and \eqref{bootA0} for $A$ and $A_0$ yields:
\bea\lab{csc10}
&&\norm{(-\Delta)^{-\frac12}((A_0)_j\,^l\,\, \pr_0\pr_l(A_i))}_{L^2(\MM)}\\
\nn&\les& \norm{A_0\prb A}_{L^2(\MM)}+\norm{\pr_l(A_0)\prb A}_{\lsit{\frac{6}{5}}}+\norm{\A^2\prb A}_{\lsit{\frac{6}{5}}}\\
\nn&\les& \norm{A_0}_{\lsitt{2}{\infty}}\norm{\prb A}_{\lsit{2}}+\norm{\pr A_0}_{\lsit{3}}\norm{\prb A}_{\lsit{2}}\\
&&\nn+\norm{\A}^2_{\lsit{6}}\norm{\prb A}_{\lsit{2}}\\
\nn&\les& M^2\ep^2.
\eea 

Next, we consider the second term in the right-hand side of \eqref{csc8}. For that term, we would like to factorize the $\pr_0$ derivative in order to get two terms of the type $\pr_0(A_0\pr_0(A))$ and $\pr_0(A_0)\pr_0(A)$, and then conclude using elliptic estimates and Sobolev embeddings on $\Si_t$. A similar strategy worked for the first term in the right-hand side of \eqref{csc8}. But it does not work directly for this term since $(-\Delta)^{-\frac12}\pr_0$ is not necessarily bounded on $L^2(\Si_t)$. Thus, we first start by showing how one may replace one $\pr_0$ with $\pr$. Using the identity  \eqref{linkRA} relating $\A$ and $\R$, we have:
\beaa
(A_0)_{j\,0}\,\, \pr_0\pr_0(A_i)&=&(A_0)_{j\,0}\,\, \pr_0(\pr_0(A_i))+\A^2\prb A\\
&=&(A_0)_{j\,0}\,\, \pr_0(\pr_i(A_0))+A_0\pr_0(\R_{0\,i\,\c\,\c})+\A^2\prb A+A\A\pr_0A_0\\
&=& A_0\pr_i(\pr_0(A_0))+A_0\D_0\R_{0\,i\,\c\,\c}+\A^2\R+\A^2\prb A+A\A\pr_0A_0.
\eeaa
Using the Bianchi identities for $\R$, we have:
$$\D_0\R_{0\,i\,\c\,\c}=\D_l\R_{l\,i\,\c\,\c}.$$
Thus we obtain:
\beaa
(A_0)_{j\,0}\,\, \pr_0\pr_0(A_i)&=& A_0\pr_i(\pr_0(A_0))+A_0\D_l\R_{l\,i\,\c\,\c}+\A^2\R+\A^2\prb A+A\A\pr_0A_0\\
&=& A_0\pr_i(\pr_0(A_0))+\pr_l(A_0\R)+\pr_l(A_0)\R+\A^2\R+\A^2\prb A+A\A\pr_0A_0.
\eeaa
Using the Sobolev embeddings on $\Si_t$ \eqref{sob} and \eqref{sobineqsit}, the bootstrap assumptions \eqref{bootA} and \eqref{bootA0} for $A$ and $A_0$, and the bootstrap assumption \eqref{bootR} for $\R$, we obtain:
\bea\lab{csc9}
&&\norm{(-\Delta)^{-\frac12}((A_0)_{j\,0}\,\, \pr_0\pr_0(A_i))}_{L^2(\MM)}\\
\nn&\les& \norm{A_0\pr_i(\pr_0(A_0))}_{\lsitt{2}{\frac{6}{5}}}+\norm{A_0\R}_{L^2(\MM)}+\norm{\pr_l(A_0)\R}_{\lsit{\frac{6}{5}}}\\
\nn&&+\norm{\A^2\R}_{\lsit{\frac{6}{5}}}+\norm{\A^2\prb A}_{\lsit{\frac{6}{5}}}+\norm{A\A\pr_0A_0}_{\lsitt{2}{\frac{6}{5}}}\\
\nn&\les&\norm{A_0}_{\lsit{6}}\norm{\pr_i(\pr_0(A_0))}_{\lsit{\frac{3}{2}}}+\norm{A_0}_{\lsitt{2}{\infty}}\norm{\R}_{\lsit{2}}\\
\nn&&+\norm{\pr_l(A_0)}_{\lsit{3}}\norm{\R}_{\lsit{2}}+\norm{A}_{\lsit{\frac{14}{5}}}\norm{\A}_{\lsit{6}}\norm{\pr_0A_0}_{\lsitt{2}{\frac{42}{13}}}\\
\nn&&+\norm{\A}^2_{\lsit{6}}(\norm{\R}_{\lsit{2}}+\norm{\prb A}_{\lsit{2}})\\
\nn&\les & M^2\ep^2.
\eea

Next, we consider the third term in the right-hand side of \eqref{csc8}. We have:
\beaa
(A^l)_j\,^m\,\, \pr_l \pr_m(A_i)&=& A^l\pr_m(\pr_l(A))+\A^2\prb A\\
\nn &=& \pr(A^l\pr_l(A))+\pr(A^l)\pr_l(A)+\A^2\prb A.
\eeaa
Together with the bilinear  assumptions  \eqref{bil1} and \eqref{bil4}, the Sobolev embeddings on $\Si_t$ \eqref{sob} and \eqref{sobineqsit}, and the bootstrap assumptions \eqref{bootA} and \eqref{bootA0} for $A$ and $A_0$, we obtain:
\bea\lab{csc13}
&&\norm{(-\Delta)^{-\frac12}((A^l)_j\,^m\,\, \pr_l \pr_m(A_i))}_{L^2(\MM)}\\
\nn&\les& \norm{A^l\pr_l(A)}_{L^2(\MM)}+\norm{(-\Delta)^{-\frac12}(\pr(A^l)\pr_l(A))}_{L^2(\MM)}+\norm{\A^2\prb A}_{\lsit{\frac{6}{5}}}\\
\nn&\les & M^3\ep^2.
\eea

Finally, we consider the fourth term in the right-hand side of \eqref{csc8}. We would like to factorize the $\pr_0$ derivative in order to get two terms of the type $\pr_0(A^l\pr_l(A))$ and $\pr_0(A^l)\pr_l(A)$, and then conclude using the bilinear assumptions \eqref{bil1} and \eqref{bil4}. A similar strategy worked for the third term in the right-hand side of \eqref{csc8}. But it does not work directly for this term since $(-\Delta)^{-\frac12}\pr_0$ is not necessarily bounded on $L^2(\Si_t)$. Thus, as for the second term, we first start by showing how one may replace $\pr_0$ with $\pr$. Using the identity \eqref{linkRA} relating $\A$ and $\R$, we have schematically:
$$\pr_0(A_i)-\pr_i(A_0)+\A^2=\R_{0\,i\,\c\,\c}$$
which yields:
\beaa
(A^l)_{j\,0}\,\, \pr_l \pr_0(A_i)&=& (A^l)_{j\,0}\,\, \pr_l(\pr_0(A_i))+\A^2\prb A\\
&=& (A^l)_{j\,0}\,\, \pr_l(\pr_i(A_0))+(A^l)_{j\,0}\,\,\pr_l(\R_{0\,i\,\c\,\c})+\A^2(\prb A, \pr A_0)\\
&=& \pr_l((A^l)_{j\,0}\,\,\R_{0\,i\,\c\,\c})+\pr_l(A^l)\R+A\pr^2A_0+\A^2(\prb A, \pr A_0)\\
&=& \pr_l((A^l)_{j\,0}\,\,\R_{0\,i\,\c\,\c})+\A^2\R+\A^2(\prb A, \pr A_0)+A\pr^2A_0,
\eeaa
where we used in the last inequality our Coulomb like gauge choice which yields $\pr_l(A^l)=\nabla_l(A^l)=A^2$. Thus, we have:
\bea\lab{csc11}
(A^l)_{j\,0}\,\, \pr_l \pr_0(A_i)=\pr_l((A^l)_{j\,0}\,\,\R_{0\,i\,\c\,\c})+h^{(4)}_{ij},
\eea
where $h^{(4)}_{ij}$ is given by:
$$h^{(4)}_{ij}=\A^2\R+\A^2(\prb A, \pr A_0)+A\pr^2A_0.$$
Using the Sobolev embeddings on $\Si_t$ \eqref{sob} and \eqref{sobineqsit}, the bootstrap assumptions \eqref{bootA} and \eqref{bootA0} for $A$ and $A_0$, and the bootstrap assumption \eqref{bootR} for $\R$, we obtain:
\bea\lab{csc12}
&&\norm{(-\Delta)^{-\frac12}(h^{(4)}_{ij})}_{L^2(\MM)}\\
\nn&\les& \norm{\A^2\R}_{\lsit{\frac{6}{5}}}+\norm{\A^2(\prb A, \pr A_0)}_{\lsit{\frac{6}{5}}}+\norm{\A\pr^2A_0}_{\lsit{\frac{6}{5}}}\\ 
\nn&\les& \norm{\A}^2_{\lsit{6}}(\norm{\R}_{\lsit{2}}+\norm{\prb A}_{\lsit{2}}+\norm{\pr A_0}_{\lsit{2}})\\
\nn&&+\norm{A}_{\lsit{6}}\norm{\pr^2A_0}_{\lsit{\frac{3}{2}}}\\
\nn&\les& M^2\ep^2.
\eea
Next, we estimate the first term on the right-hand side of \eqref{csc11}. Since $\R_{0\,i\,0\,0}=0$, the terms $\R_{0\,i\,\c\,\c}$ are of two types: $\R_{0\,i\,m\,n}$ or $\R_{0\,i\,0\,m}$. Now, from the symmetries of $\R$ and the Einstein equations, we have:
$$\R_{0\,i\,m\,n}=\R_{m\,n\,0\,i}\textrm{ and }\R_{0\,i\,0\,m}=-\R_{n\,i\,n\,m}.$$
Also, in view of the link between $\R$ and $\A$ \eqref{linkRA}, we have schematically:
$$\R_{m\,n\,0\,i}=\pr_m(A_n)-\pr_n(A_m)+A^2\textrm{ and }\R_{n\,i\,n\,m}=\pr_n(A_i)-\pr_i(A_n)+A^2.$$
Thus, we obtain schematically:
$$\R_{0\,i\,\c\,\c}=\pr A+A^2.$$
which, using the Coulomb gauge, yields:
\beaa
\pa_l \left((A^l)_{j\,0}\,\,\R_{0\,i\,\c\,\c}\right)&=&A^l\pr_l\pr(A)+A^2\pr A\\
&=&\pr(A^l\pr_l(A))+\pr(A^l)\pr_l(A)+A^2\pr A.
\eeaa
Together with the bilinear assumptions  \eqref{bil1} and \eqref{bil4}, the Sobolev embeddings on $\Si_t$ \eqref{sob} and \eqref{sobineqsit}, and the bootstrap assumptions \eqref{bootA} and \eqref{bootA0} for $A$ and $A_0$, we obtain:
\bea\lab{csc12bis}
&&\norm{(-\Delta)^{-\frac12}(\pr_l((A^l)_{j\,0}\,\,\R_{0\,i\,\c\,\c}))}_{L^2(\MM)}\\
\nn&\les& \norm{A^l\pr_l(A)}_{L^2(\MM)}+\norm{(-\Delta)^{-\frac12}(\pr(A^l)\pr_l(A))}_{L^2(\MM)}+\norm{A^2\pr A}_{\lsit{\frac{6}{5}}}\\
\nn&\les & M^3\ep^2.
\eea
Now, \eqref{csc11}-\eqref{csc12bis} imply:
\bea\lab{csc12ter}
\norm{(-\Delta)^{-\frac12}((A^l)_{j\,0}\,\, \pr_l \pr_0(A_i))}_{L^2(\MM)}\les M^3\ep^2.
\eea
Finally, \eqref{csc8}-\eqref{csc12ter} imply:
\bea\lab{csc15}
\norm{(-\Delta)^{-\frac12}((A^\la)_j\,^{\mu}\,\, \pr_\la   \pr_\mu(A_i))}_{L^2(\MM)}\les M^3\ep^2.
\eea

In the end, \eqref{csc1}, \eqref{csc4}-\eqref{csc7}, and \eqref{csc15} yield:
$$\norm{(-\Delta)^{-\frac12}\square(\pr_j(A_i))}_{L^2(\MM)}\les M^3\ep^2.$$
This implies:
$$\norm{(-\Delta)^{-\frac12}[\square(\pr_j(A_i)-\pr_i(A_j))]}_{L^2(\MM)}\les M^3\ep^2,$$
which concludes the proof of the proposition.
\end{proof}

\subsection{Estimates for  $\square B$}
 Here we derive a wave equation for each component of $B=\Delta^{-1}\curl(A)$ and prove  the following, 
\begin{proposition}[Estimates  for $\square B$]\lab{prop:waveeqB}
The components  $B_i=(-\Delta)^{-1}(\curl(A)_i)$ verify the following estimate,
\bea
\sum_{i=1}^3\norm{\pr \,\square B_i}_{L^2(\MM)}  \lesssim M^2\ep^2.\label{spacetimeB}
\eea

We also have,
\bea
\sum_{i=1}^3\norm{\pr_0(\pr_0(B_i))}_{L^2(\MM)}\lesssim M\ep.
\eea
\end{proposition}

\begin{proof}  The estimates for $\square B$  are simpler than those for $\pr  \square B$ and $\square \pr B$.
We prove first the estimates for  $\pr  \square B$ and derive those for  $\square \pr B $ using the commutation formula 
\eqref{commsquare2}. 

We have:
\bea
\nn\square(B_i)&=&[\square,(-\Delta)^{-1}](\curl(A)_i)+(-\Delta)^{-1}(\square(\curl(A)_i))\\
\nn&=& (-\Delta)^{-1}[\square,\Delta](-\Delta)^{-1}(\curl(A)_i)+(-\Delta)^{-1}(\square(\curl(A)_i))\\
\lab{eq:fondwaveeqB}&=& (-\Delta)^{-1}[\square,\Delta](B_i)+(-\Delta)^{-1}(\square(\curl(A)_i)).
\eea
Thus, we obtain:
\bea\lab{csd1}
&&\norm{\pr\square(B_i)}_{L^2(\MM)}\\
\nn&\les & \norm{\pr(-\Delta)^{-1}[\square,\Delta](B_i)}_{L^2(\MM)}+\norm{\pr(-\Delta)^{-1}(\square(\curl(A)_i))}_{L^2(\MM)}\\
\nn&\les & \norm{(-\Delta)^{-\frac12}[\square,\Delta](B_i)}_{L^2(\MM)}+\norm{(-\Delta)^{-\frac12}(\square(\curl(A)_i))}_{L^2(\MM)}\\
\nn&\les & \norm{(-\Delta)^{-\frac12}[\square,\Delta](B_i)}_{L^2(\MM)}+M^3\ep^2,
\eea
where we used Proposition \ref{prop:commutsquarecurl} in the last inequality.

In view of \eqref{csd1}, we need to estimate $\norm{(-\Delta)^{-\frac12}[\square,\Delta](B_i)}_{L^2(\MM)}$. Recall the commutator formula \eqref{commsquare3}:
\bea\lab{csd2}
[\square, \Delta]\phi
&=&-4k^{ab}\nabla_a\nabla_b(\pr_0\phi)+4n^{-1}\nabla_bn\nabla_b(\pr_0(\pr_0\phi))
 -2\nabla_0k^{ab}\nabla_a\nabla_b\phi\\
 &+&F^{(1)}\prb^2\phi+F^{(2)}\prb\phi,\nn\\
 F^{(1)}&=& n^{-1}\nab(\pr_0n)+\pr A_0+\A^2,\nn\\
 F^{(2)}&=& n^{-1}\nab^2(\pr_0n)+\pr\pr A_0+\A\Big(n^{-1}\nab(\pr_0n),\R, \pr\A, \prb A\Big)+\A^3,\nn
\eea
Using the estimate \eqref{bootn} for $n$ and the bootstrap assumptions \eqref{bootR} for $\R$, \eqref{bootA} for $A$ and \eqref{bootA0} for $A_0$, we have:
\bea\lab{csd3}
\norm{F^{(1)}}_{\lsit{3}}\les \norm{n^{-1}\nab(\pr_0n)}_{\lsit{3}}+ \norm{\pr A_0}_{\lsit{3}}+\norm{\A}^2_{\lsit{6}}\les M\ep,
\eea
and:
\bea\lab{csd4}
&&\norm{F^{(2)}}_{\lsit{\frac{3}{2}}}\\
\nn&\les& \norm{n^{-1}\nab^2(\pr_0n)}_{\lsit{\frac{3}{2}}}+\norm{\pr\pr A_0}_{\lsit{\frac{3}{2}}}+\norm{\A}_{\lsit{6}}^3\\
\nn&&+\norm{\A}_{\lsit{6}}(\norm{n^{-1}\nab(\pr_0n)}_{\lsit{2}}+\norm{\R}_{\lsit{2}}+\norm{\pr\A}_{\lsit{2}}+\norm{\prb A}_{\lsit{2}})\\
\nn&\les& M\ep.
\eea

Using \eqref{csd2}-\eqref{csd4} together with the estimates of Proposition \ref{lemma:estB}, we obtain:
\bea\lab{csd6}
&&\norm{(-\Delta)^{-\frac12}[\square, \Delta](B_l)}_{L^2(\MM)}\\
\nn&\les& \norm{(-\Delta)^{-\frac12}[k^{ab}\nabla_a\nabla_b(\pr_0(B_l))]}_{L^2(\MM)}+ \norm{(-\Delta)^{-\frac12}[n^{-1}\nabla_bn\nabla_b(\pr_0(\pr_0(B_l)))]}_{L^2(\MM)}\\
\nn&& + \norm{(-\Delta)^{-\frac12}[\nabla_0k^{ab}\nabla_a\nabla_b(B_l)]}_{L^2(\MM)}+\norm{F^{(1)}}_{\lsit{3}}\norm{\prb^2(B_l)}_{L^2(\MM)}\\
\nn&&+\norm{F^{(2)}}_{\lsit{\frac{3}{2}}}\norm{\prb(B_l)}_{\lsit{6}}\\
\nn&\les& \norm{(-\Delta)^{-\frac12}[k^{ab}\nabla_a\nabla_b(\pr_0(B_l))]}_{L^2(\MM)}+ \norm{(-\Delta)^{-\frac12}[n^{-1}\nabla_bn\nabla_b(\pr_0(\pr_0(B_l)))]}_{L^2(\MM)}\\
\nn&& + \norm{(-\Delta)^{-\frac12}[\nabla_0k^{ab}\nabla_a\nabla_b(B_l)]}_{L^2(\MM)}+M^2\ep^2+M\ep\norm{\pr_0(\pr_0(B_l)))}_{L^2(\MM)}.
\eea

Next, we estimate the various terms in the right-hand side of \eqref{csd6}. The first and the third will require bilinear estimates, while the second will require the estimate $\nabla n\in L^\infty(\MM)$. 
We start with the first one. We have:
\beaa
k^{ab}\nabla_a\nabla_b(\pr_0(B_l))&=& \nabla_a[k^{ab}\nabla_b(\pr_0(B_l))]-\nabla_ak^{ab}\nabla_b(\pr_0(B_l))\\
&=& \nabla_a[k^{ab}\nabla_b(\pr_0(B_l))],
\eeaa
where we used the constraint equations \eqref{constk} for $k$ in the last equality. Together with the Bochner inequality on $\Si_t$ \eqref{bochnerestimates} and the bilinear  assumption  \eqref{bil5} , we obtain:
\bea\lab{csd7}
\norm{(-\Delta)^{-\frac12}[k^{ab}\nabla_a\nabla_b(\pr_0(B_l))]}_{L^2(\MM)}\les \norm{k^{ab}\pr_b(\pr_0(B_l))]}_{L^2(\MM)}\les M^3\ep^2.
\eea

Next, we estimate the second term in the right-hand side of \eqref{csd6}. We have:
$$n^{-1}\nabla_bn\nabla_b(\pr_0(\pr_0(B_l)))=\nabla^b[n^{-1}\nabla_bn\pr_0(\pr_0(B_l))]-(n^{-1}\Delta n-n^{-2}|\nabla n|^2)\pr_0(\pr_0(B_l)).$$
Together with the estimates \eqref{bootn} for the lapse $n$ and the Sobolev embedding on $\Si_t$ \eqref{sob}, this yields:
\bea\lab{csd8}
&&\norm{(-\Delta)^{-\frac12}[n^{-1}\nabla_bn\nabla_b(\pr_0(\pr_0(B_l)))]}_{L^2(\MM)}\\
\nn&\les& \norm{n^{-1}\nabla_bn\pr_0(\pr_0(B_l))}_{L^2(\MM)}+\norm{(n^{-1}\Delta n-n^{-2}|\nabla n|^2)\pr_0(\pr_0(B_l))}_{\lsitt{2}{\frac{6}{5}}}\\
\nn&\les& (\norm{\nabla n}_{L^\infty}+\norm{n^{-1}\Delta n-n^{-2}|\nabla n|^2}_{\lsit{3}})\norm{\pr_0(\pr_0(B_l))}_{L^2(\MM)}\\
\nn&\les& M\ep\norm{\pr_0(\pr_0(B_l))}_{L^2(\MM)}.
\eea

\begin{remark}\lab{rem:tempura1}
Note that there is no room in the estimate \eqref{csd8}. In particular, the estimate $\norm{\nabla n}_{L^\infty(\MM)}\les M\ep$ given by \eqref{bootn} is crucial as emphasized in remark \ref{rem:tempura}.
\end{remark}

Finally, we consider the third term in the right-hand side of \eqref{csd6}. Recall from \eqref{eq:structfol1} that the second fundamental form satisfies the following equation:
\bea\lab{csd9}
\nabla_0k_{ab}= E_{ab}+F^{(3)}_{ab},
\eea
where $E$ is the 2-tensor on $\Si_t$ defined as:
$$E_{ab}=\R_{a\,0\,b\,0},$$
and where $F^{(3)}_{ab}$ is given by:
$$F^{(3)}_{ab}=-n^{-1}\nabla_a\nabla_bn-k_{ac}k_b\,^c.$$
In view of the estimates \eqref{bootk} for $k$ and \eqref{bootn} for $n$, $F^{(3)}_{ab}$ satisfies the estimate:
\bea\lab{csd10}
\norm{F^{(3)}_{ab}}_{\lsit{3}}\les \norm{\nabla^2n}_{\lsit{3}}+\norm{k}^2_{\lsit{6}}\les M\ep.
\eea
Next, we consider the term involving $E$ in the right-hand side of \eqref{csd9}. Using the maximal foliation assumption, the Bianchi identities and the symmetries of $\R$, we obtain:
$$\nabla^aE_{ab}=\D^a\R_{a\,0\,b\,0}+\A\R=-\D^0\R_{0\,0\,b\,0}+\A\R=\pr_0(\R_{0\,0\,b\,0})+\A\R=\A\R$$
which together with the bootstrap assumptions \eqref{bootA} for $A$ and \eqref{bootA0} for $A_0$, and the bootstrap assumption \eqref{bootR} for $\R$ yields:
\bea\lab{csd11}
\norm{\nabla^aE_{ab}}_{\lsit{\frac{3}{2}}}\les \norm{\A}_{\lsit{6}}\norm{\R}_{\lsit{2}}\les M^2\ep^2.
\eea
Now, we have:
$$E_{ab}\nabla_a\nabla_b(B_l)=\nabla^a[E_{ab}\nabla_b(B_l)]-\nabla^aE_{ab}\nabla_b(B_l)$$
which together with the bilinear estimate \eqref{bil6}, the estimates of Proposition \ref{lemma:estB} for $B$  and \eqref{csd11} yields:
\bea\lab{csd12}
&&\norm{(-\Delta)^{-\frac12}[E_{ab}\nabla_a\nabla_b(B_l)]}_{L^2(\MM)}\\
\nn&\les& \norm{(-\Delta)^{-\frac12}\nabla^a[E_{ab}\nabla_b(B_l)]}_{L^2(\MM)}+\norm{(-\Delta)^{-\frac12}[\nabla^aE_{ab}\pr_b(B_l)]}_{L^2(\MM)}\\
\nn&\les& \norm{\R_{a\,0\,b\,0}\pr_b(B_l)}_{L^2(\MM)}+\norm{\nabla^aE_{ab}}_{\lsit{\frac{3}{2}}}\norm{\pr(B_l)}_{\lsit{6}}\\
\nn&\les& M^2\ep^2.
\eea
\eqref{csd9}, \eqref{csd10}, \eqref{csd12} and the estimates of Proposition  \ref{lemma:estB} for $B$ yield:
\bea\lab{csd13}
&&\norm{(-\Delta)^{-\frac12}[\nabla_0k^{ab}\nabla_a\nabla_b(B_l)]}_{L^2(\MM)}\\
\nn&\les & \norm{(-\Delta)^{-\frac12}[E_{ab}\nabla_a\nabla_b(B_l)]}_{L^2(\MM)}+\norm{(-\Delta)^{-\frac12}[F^{(3)}_{ab}\nabla_a\nabla_b(B_l)]}_{L^2(\MM)}\\
\nn&\les & M^2\ep^2+\norm{F^{(3)}_{ab}}_{\lsit{3}}\norm{\pr^2(B)}_{\lsit{2}}\\
\nn&\les & M^2\ep^2.
\eea

Finally, \eqref{csd6}-\eqref{csd8} and \eqref{csd13} yield:
$$\norm{(-\Delta)^{-\frac12}[\square, \Delta](B_i)}_{L^2(\MM)}\les M^2\ep^2+M\ep\norm{\pr_0(\pr_0(B_i)))}_{L^2(\MM)},$$
which together with \eqref{csd1} implies:
\bea\lab{csd14}
\norm{\pr\square(B_i)}_{L^2(\MM)}\les  M^2\ep^2+M\ep\norm{\pr_0(\pr_0(B_i)))}_{L^2(\MM)}.
\eea
Recalling \eqref{eq:YM16}, we have:
$$\pr_0(\pr_0(B_i))=-\square(B_i)+\Delta(B_i)+n^{-1}\nabla n\cdot\nabla(B_i),$$
which together with the estimates of Lemma \ref{lemma:estB} for $B$, the estimates \eqref{bootn} for $n$,  and \eqref{csd14} yields:
\bea\lab{csd15}
&&\norm{\pr_0(\pr_0(B_i)))}_{L^2(\MM)}\\
\nn&\les& \norm{\square(B_i)}_{L^2(\MM)}+\norm{\Delta(B_i)}_{L^2(\MM)}+\norm{\nabla n\cdot\nabla(B_i)}_{L^2(\MM)}\\ 
\nn&\les& M^2\ep^2+M\ep\norm{\pr_0(\pr_0(B_i)))}_{L^2(\MM)}+\norm{\pr^2(B_i)}_{\lsit{2}}+\norm{\nabla n}_{L^\infty}\norm{\pr(B_i)}_{\lsit{2}}\\ 
\nn&\les& M^2\ep^2+M\ep\norm{\pr_0(\pr_0(B_i)))}_{L^2(\MM)}.
\eea
Choosing $\ep>0$ such that $M\ep$ is small enough to absorb the term $\norm{\pr_0(\pr_0(B_i)))}_{L^2(\MM)}$ in the right-hand side, \eqref{csd14} and \eqref{csd15} 
gives the desired  estimate  for both  $\|\pr\square B\|_{L^2(\MM)}$ and $\norm{\pr_0(\pr_0(B_i)))}_{L^2(\MM)}$ of the lemma. 

\end{proof}

\section{Energy estimate for the wave equation on a curved background with bounded $L^2$ curvature}\lab{sec:bobo6}

Recall that $e_0=T$, the future unit normal to the $\Si_t$ foliation. Let $\pi$ be the deformation tensor of $e_0$, that is the symmetric 2-tensor on $\MM$ defined as:
$$\pi_{\a\b}=\D_\a T_\b+\D_\b T_\a.$$
In view of the definition of the second fundamental form $k$ and the lapse $n$, we have:
\bea\lab{defpi}
\pi_{ab}=-2k_{ab},\, \pi_{a0}=\pi_{0a}=n^{-1}\nabla_an,\,\pi_{00}=0.
\eea

In what follows $\HH$ denotes an arbitrary  weakly regular null  hypersurface\footnote{i.e. it   satisfies assumptions \eqref{assumptionH1} and \eqref{assumptionH2}}  with future normal  $L$  verifying  ${\bf g}(L,T)=-1$.  We denote by $\nabb$     the induced connection on the $2$-surfaces  $\HH\cap \Si_t$.  

We have the following energy estimate for the scalar wave equation:
\begin{lemma}\lab{lemma:energyestimate}
Let $F$ a scalar function on $\MM$, and let $\phi_0$ and $\phi_1$ two scalar functions on $\Si_0$. Let $\phi$ the solution of the following wave equation on $\MM$:
\bea\lab{eq:wave}
\left\{\begin{array}{l}
\square\phi=F,\\
\phi|_{\Si_0}=\phi_0,\, \pr_0\phi|_{\Si_0}=\phi_1. 
\end{array}\right.
\eea
Then, $\phi$ satisfies the following energy estimate:
\bea
&&\norm{\prb\phi}_{\lsit{2}}+\sup_{\HH}(\norm{\nabb\phi}_{L^2(\HH)}+\norm{L(\phi)}_{L^2(\HH)})\nn\\
&\les& \norm{\nabla\phi_0}_{L^2(\Si_0)}+\norm{\phi_1}_{L^2(\Si_0)}+\norm{F}_{L^2(\MM)},\lab{energy0}
\eea
where the supremum is taken over all null hypersurfaces $\HH$ satisfying assumptions \eqref{assumptionH1} and \eqref{assumptionH2}.
\end{lemma}

\begin{proof}
We introduce the energy momentum tensor $Q_{\a\b}$ on $\MM$ given by:
$$Q_{\a\b}=Q_{\a\b}[\phi]=\pr_\a\phi\pr_\b\phi-\frac{1}{2}\g_{\a\b}\left(\g^{\mu\nu}\pr_\mu\phi\pr_\nu\phi\right).$$
In view of the equation \eqref{eq:wave} satisfied by $\phi$, we have:
$$\D^\a Q_{\a\b}=F\pr_\b\phi.$$
Now, we form the 1-tensor $P$:
$$P_\a=Q_{\a 0},$$
and we obtain:
$$\D^\a P_\a=\D^\a Q_{\a 0}+Q_{\a\b}\D^\a T^\b=F\pr_0\phi+\frac12 Q_{\a\b}\pi^{\a\b},$$
where $\pi$ is the deformation tensor of $e_0$. Integrating a specifically chosen region of $\MM$, bounded by $\Sigma_0, \Sigma_t$ and $\HH$, we obtain:
\bea\lab{nrj1}
&&\norm{\prb\phi}^2_{\lsit{2}}+\sup_{\HH}\norm{\nabb\phi}^2_{L^2(\HH)}\\
\nn&\les& \norm{\nabla\phi_0}^2_{L^2(\Si_0)}+\norm{\phi_1}^2_{L^2(\Si_0)}+\left|\int_{\MM}F\pr_0\phi d\MM\right|+\left|\int_{\MM}Q_{\a\b}\pi^{\a\b}d\MM\right|\\
\nn&\les& \norm{\nabla\phi_0}^2_{L^2(\Si_0)}+\norm{\phi_1}^2_{L^2(\Si_0)}+\norm{F}_{L^2(\MM)}\norm{\pr_0\phi}_{L^2(\MM)}+\left|\int_{\MM}Q_{\a\b}\pi^{\a\b}d\MM\right|.
\eea

Next, we deal with the last term in the right-hand side of \eqref{nrj1}. In view of \eqref{defpi}, we have:
\beaa
&&\int_{\MM}Q_{\a\b}\pi^{\a\b}d\MM\\
&=&-2\int_{\MM}Q_{ab}k^{ab}d\MM+\int n^{-1}\nabla^inQ_{0i}d\MM\\
&=&-2\int_{\MM}\pr_a\phi\pr_b\phi k^{ab}d\MM+\int_{\MM}\tr_g k\left(\g^{\mu\nu}\pr_\mu\phi\pr_\nu\phi\right)d\MM+\int n^{-1}\nabla^an\pr_a\phi\pr_0\phi d\MM\\
&=&-2\int_{\MM}\pr_a\phi\pr_b\phi k^{ab}d\MM+\int n^{-1}\nabla^an\pr_a\phi\pr_0\phi d\MM,
\eeaa
where we used in the last inequality the maximal foliation assumption. Together with the bilinear bootstrap  assumption  and the estimates \eqref{bootn} for the lapse $n$, this yields:
\beaa
\left|\int_{\MM}Q_{\a\b}\pi^{\a\b}d\MM\right|&\les&\norm{k_{a\,\c}\pr^a\phi}_{L^2(\MM)}\norm{\pr\phi}_{L^2(\MM)}+\norm{\nabla n}_{L^\infty(\MM)}\norm{\prb\phi}^2_{L^2(\MM)}\\
&\les& M^2\ep\left(\sup_{\HH}\norm{\nabb\phi}_{L^2(\HH)}\right)\norm{\prb\phi}_{L^2(\MM)}+M\ep\norm{\prb\phi}^2_{L^2(\MM)},
\eeaa
which together with \eqref{nrj1} concludes the proof of the lemma.
\end{proof}

\begin{remark}
The most dangerous term in the right-hand side of the previous inequality is $\norm{k_{a\,\c}\pr^a\phi}_{L^2(\MM)}$. Usually, when deriving energy estimates for the wave equation, this term is typically estimated by:
$$\norm{k_{a\,\c}\pr^a\phi}_{L^2(\MM)}\les \norm{k}_{\lsitt{2}{\infty}}\norm{\pr\phi}_{\lsit{2}}$$
which requires a Strichartz estimate for $k$. This Strichartz estimate fails under the assumptions of Theorem \ref{th:main}, and we need to rely instead on our  bilinear estimate \eqref{bil8}.
\end{remark}

We have the following higher order energy estimate for the scalar wave equation:
\begin{lemma}\lab{lemma:energyestimatebis}
Let $F$ a scalar function on $\MM$, and let $\phi_0$ and $\phi_1$ two scalar functions on $\Si_0$. Let $\phi$ the solution of the wave equation \eqref{eq:wave} on $\MM$. Then, $\phi$ satisfies the following energy estimate:
\bea
&&\norm{\pr(\prb\phi)}_{\lsit{2}}+\norm{\pr_0(\pr_0\phi)}_{L^2(\MM)}+\sup_{\HH}\left(\norm{\nabb(\pr\phi)}_{L^2(\HH)}+\norm{L(\pr\phi)}_{L^2(\HH)}\right)\nn\\
&\les& \norm{\nabla^2\phi_0}_{L^2(\Si_0)}+\norm{\nabla\phi_0}_{L^2(\Si_0)}+\norm{\nabla\phi_1}_{L^2(\Si_0)}+\norm{\nabla F}_{L^2(\MM)},
\lab{energy1}
\eea
where the supremum is taken over all null hypersurfaces $\HH$ satisfying assumption \eqref{assumptionH1} and \eqref{assumptionH2}. Furthermore, $\square(\pr_j\phi)$ satisfies the following estimate:
$$\norm{\square(\pr_j\phi)}_{L^2(\MM)}\les M\ep(\norm{\nabla^2\phi_0}_{L^2(\Si_0)}+\norm{\nabla\phi_0}_{L^2(\Si_0)}+\norm{\nabla\phi_1}_{L^2(\Si_0)}+                \norm{ F}_{L^2(\MM)}+          \norm{\pr F}_{L^2(\MM)}).$$
\end{lemma}

\begin{proof}
We derive an equation for $\pr_j\phi$. Differentiating \eqref{eq:wave}, we obtain:
\bea\lab{nrj2}
\left\{\begin{array}{l}
\square(\pr_j\phi)=\pr_jF+[\square,\pr_j](\phi),\\
\pr_j\phi|_{\Si_0}=\pr_j\phi_0,\, \pr_0(\pr_j\phi)|_{\Si_0}=\left(\pr_j(\pr_0\phi)+[\pr_0,\pr_j]\phi\right)|_{\Si_0}=\pr_j\phi_1+n^{-1}\nabla n\phi_1+\A\nabla\phi_0,
\end{array}\right.
\eea
where we used the commutator estimate \eqref{notcoordinate} in the last equality. Applying the energy estimate of Lemma \ref{lemma:energyestimate} to \eqref{nrj2}, we obtain:
\beaa
&&\norm{\prb(\pr_j\phi)}_{\lsit{2}}+\sup_{\HH}(\norm{\nabb(\pr_j\phi)}_{L^2(\HH)}+\norm{L(\pr_j\phi)}_{L^2(\HH)})\\
&\les& \norm{\nabla(\pr_j\phi_0)}_{L^2(\Si_0)}+\norm{\pr_j\phi_1+n^{-1}\nab n \phi_1+\A\nabla\phi_0}_{L^2(\Si_0)}+\norm{\pr_jF+[\square,\pr_j]\phi}_{L^2(\MM)}.
\eeaa
which after taking the supremum over $j=1, 2, 3$ yields:
\beaa
&&\norm{\pr(\prb\phi)}_{\lsit{2}}+\sup_{\HH}\left(\norm{\nabb(\pr\phi)}_{L^2(\HH)}+\norm{L(\pr\phi)}_{L^2(\HH)}\right)\\
\nn&\les& \norm{\nabla^2\phi_0}_{L^2(\Si_0)}+\norm{\nabla\phi_1}_{L^2(\Si_0)}+\norm{\nabla F}_{L^2(\MM)}+\norm{n^{-1}\nab n}_{L^3(\Si_0)}\norm{\phi_1}_{L^6(\Si_0)}\\
\nn&&+\norm{\A}_{L^6(\Si_0)}\norm{\nabla\phi_0}_{L^3(\Si_0)}+\norm{n^{-1}\nab n\pr_0\phi+\A\pr\phi}_{\lsit{2}}+\sup_j\norm{[\square,\pr_j]\phi}_{L^2(\MM)},
\eeaa
where the term $n^{-1}\nab n\pr_0\phi+\A\pr\phi$ in the last inequality comes from the commutator formula \eqref{notcoordinate} applied to $[\pr_0, \pr_j]\phi$. Together with the Sobolev embedding  \eqref{ikea1} on $\Si_0$, the estimates for $\A$ on $\Si_0$ of Proposition \ref{lemma:initialslice}, the bootstrap assumptions \eqref{bootA} and \eqref{bootA0} for $\A$, the estimate \eqref{bootn} for $n$, and the Sobolev embedding \eqref{sobineqsit} on $\Si_t$, we obtain
\beaa
&&\norm{\pr(\prb\phi)}_{\lsit{2}}+\sup_{\HH}\left(\norm{\nabb(\pr\phi)}_{L^2(\HH)}+\norm{L(\pr\phi)}_{L^2(\HH)}\right)\\
\nn&\les& \norm{\nabla^2\phi_0}_{L^2(\Si_0)}+\norm{\nabla\phi_0}_{L^2(\Si_0)}+\norm{\nabla\phi_1}_{L^2(\Si_0)}+\norm{\nabla F}_{L^2(\MM)}\\
\nn&&+\norm{n^{-1}\nab n}_{\lsit{3}}\norm{\pr_0\phi}_{\lsit{6}}+\norm{\A}_{\lsit{6}}\norm{\pr\phi}_{\lsit{3}}+\sup_j\norm{[\square,\pr_j]\phi}_{L^2(\MM)}\\
\nn&\les& \norm{\nabla^2\phi_0}_{L^2(\Si_0)}+\norm{\nabla\phi_0}_{L^2(\Si_0)}+\norm{\nabla\phi_1}_{L^2(\Si_0)}+\norm{\nabla F}_{L^2(\MM)}\\
\nn&&+M\ep\norm{\pr\prb\phi}_{\lsit{2}}+M\ep\norm{\pr\phi}_{\lsit{2}}+\sup_j\norm{[\square,\pr_j]\phi}_{L^2(\MM)}.
\eeaa
In view of the fact that we may choose $\ep$ such that $M\ep$ is small enough, we infer
\bea
\lab{nrj3}&&\norm{\pr(\prb\phi)}_{\lsit{2}}+\sup_{\HH}\left(\norm{\nabb(\pr\phi)}_{L^2(\HH)}+\norm{L(\pr\phi)}_{L^2(\HH)}\right)\\
\nn&\les& \norm{\nabla^2\phi_0}_{L^2(\Si_0)}+\norm{\nabla\phi_0}_{L^2(\Si_0)}+\norm{\nabla\phi_1}_{L^2(\Si_0)}+\norm{\nabla F}_{L^2(\MM)}+M\ep\norm{\pr\phi}_{\lsit{2}}\\
\nn&&+\sup_j\norm{[\square,\pr_j]\phi}_{L^2(\MM)},
\eea

Next, we estimate the last term in the right-hand side of \eqref{nrj3}. In view of the commutator formula \eqref{commsquare2}, we have:
\bea\lab{freddy}
[\square,\pr_j]\phi&=&2(\A^\la)_j\,^{\mu}\,\, \pr_\la   \pr_\mu\phi+\pr_0(n^{-1}\nab n)\pr_0\phi +\pr_0A_0\pr\phi+ \A^2\prb\phi\\
\nn&=& A^i\pr_i(\pr_l\phi)+A^i\pr_i(\pr_0\phi)+A^0\pr_0(\pr_l\phi)+(A^0)_{j0}\pr_0(\pr_0\phi)+\pr_0(n^{-1}\nab n)\pr_0\phi\\
\nn&& +\pr_0A_0\pr\phi+\A^2\prb\phi\\
\nn&=& A^i\pr_i(\pr_l\phi)+A^i\pr_i(\pr_0\phi)+h_j,
\eea
where $h_j$ is defined in view of the identity \eqref{eq:YM11} as:
$$h_j=A^0\pr_0(\pr_l\phi)+n^{-1}\nabla_jn \pr_0(\pr_0\phi)+\pr_0(n^{-1}\nab n)\pr_0\phi +\pr_0A_0\pr\phi+\A^2\prb\phi.$$
We estimate the various terms in the right-hand of \eqref{freddy} starting with $h_j$. The Sobolev embedding on $\Si_t$ \eqref{sobineqsit}, the bootstrap estimates \eqref{bootA0} for $A_0$ and \eqref{bootA} for $A$, and the estimate \eqref{bootn} for the lapse $n$ yield:
\bea\lab{freddy1}
&&\norm{h_j}_{L^2(\MM)}\\
\nn&\les& \norm{A^0}_{\lsitt{2}{\infty}}\norm{\pr_0(\pr_l\phi)}_{\lsit{2}}+\norm{\nabla n}_{L^\infty(\MM)}\norm{\pr_0(\pr_0\phi)}_{L^2(\MM)}\\
\nn&&+\norm{\pr_0(n^{-1}\nab n)}_{\lsit{3}}\norm{\pr_0\phi}_{\lsit{6}}+\norm{\pr_0 A_0}_{\lsitt{2}{\frac{42}{13}}}\norm{\pr\phi}_{\lsit{\frac{21}{4}}}\\
\nn&&+\norm{\A}^2_{\lsit{6}}\norm{\prb\phi}_{\lsit{6}}\\
\nn&\les& M\ep(\norm{\pr_0(\pr_0\phi)}_{L^2(\MM)}+\norm{\pr(\prb\phi)}_{\lsit{2}}+\norm{\pr\phi}_{\lsit{2}}).
\eea
Note again in view of the previous inequality that the estimate $\nabla n\in L^\infty(\MM)$ is crucial as emphasized by Remarks \ref{rem:tempura} and \ref{rem:tempura1}. Next, we deal with the first and the second term in the right-hand of \eqref{freddy}. Using the bilinear estimate \eqref{bil8}, we have:
\beaa
\norm{A^i\pr_i(\pr_l\phi)}_{L^2(\MM)}+\norm{A^i\pr_i(\pr_0\phi)}_{L^2(\MM)}&\les& M\ep\left(\sup_{\HH}(\norm{\nabb(\pr_l\phi)}_{L^2(\HH)}+\norm{\nabb(\pr_0\phi)}_{L^2(\HH)}\right),
\eeaa
which together with \eqref{freddy} and \eqref{freddy1} yields:
\bea\lab{nrj4}
&&\norm{[\square,\pr_j]\phi}_{L^2(\MM)}\\
\nn&\les& M\ep\left(\sup_{\HH}(\norm{\nabb(\pr_l\phi)}_{L^2(\HH)}+\norm{\nabb(\pr_0\phi)}_{L^2(\HH)}\right)+M\ep(\norm{\pr_0(\pr_0\phi)}_{L^2(\MM)}+\norm{\pr(\prb\phi)}_{\lsit{2}}).
\eea
It remains to estimate the term $\norm{\nabb(\pr_0\phi)}_{L^2(\HH)}$. Let us define the vectorfield $N=L-e_0$. Since ${\bf g}(L,e_0)=-1$, and since $L$ is null, $N$ is tangent to $\Si_t$. Decomposing $e_0=L-N$, we obtain schematically:
\bea\lab{decompositionpr0}
|\nabb(\pr_0\phi)| &\leq& |\nabb(\nabla_N\phi)|+|\nabb(L(\phi))|\\
\nn&\les& |\nabb(N_j\pr_j\phi)|+|\pr(L(\phi))|\\
\nn&\les& |\nabb(\pr\phi)|+|L(\pr\phi)|+|(\D N)(\pr\phi)|+|(\D L)(\pr\phi)|+|A\pr\phi|
\eea
which together with the assumptions \eqref{assumptionH1} and \eqref{assumptionH2} for $\HH$, and the embedding \eqref{HversusSitbis} on $\HH$ yields:
\bea\lab{nrj6}
&&\norm{\nabb(\pr_0\phi)}_{L^2(\HH)}\\
\nn&\les& \norm{\nabb(\pr\phi)}_{L^2(\HH)}+\norm{L(\pr\phi)}_{L^2(\HH)}+\norm{(\D N)(\pr\phi)}_{L^2(\HH)}+\norm{(\D\L)(\pr\phi)}_{L^2(\HH)}+\norm{A\pr\phi}_{L^2(\HH)}\\
\nn&\les& \norm{\nabb(\pr\phi)}_{L^2(\HH)}+\norm{L(\pr\phi)}_{L^2(\HH)}+(\norm{\D N}_{L^3(\HH)}+\norm{\D\L}_{L^3(\HH)}+\norm{A}_{L^3(\HH)})\norm{\pr\phi}_{L^6(\HH)}\\
\nn &\les& (\norm{\nabb(\pr\phi)}_{L^2(\HH)}+\norm{L(\pr\phi)}_{L^2(\HH)})(1+\norm{\nabla A}_{\lsit{2}})\\
\nn &\les& \norm{\nabb(\pr\phi)}_{L^2(\HH)}+\norm{L(\pr\phi)}_{L^2(\HH)},
\eea 
where we used the bootstrap assumptions \eqref{bootA} for $A$ in the last inequality. Finally, \eqref{nrj4}-\eqref{nrj6} yield:
\bea\lab{nrj7}
\sup_j\norm{[\square,\pr_j](\phi)}_{L^2(\MM)}&\les& M\ep\left(\sup_{\HH}(\norm{\nabb(\pr\phi)}_{\HH}+\norm{L(\pr\phi)}_{L^2(\HH)}\right)\\
\nn&&+M\ep(\norm{\pr_0(\pr_0\phi)}_{L^2(\MM)}+\norm{\pr(\prb\phi)}_{\lsit{2}}+\norm{\pr\phi}_{\lsit{2}}).
\eea

Now, \eqref{nrj3} and \eqref{nrj7} imply:
\beaa
&&\norm{\pr(\prb\phi)}_{\lsit{2}}+\sup_{\HH}\left(\norm{\nabb(\pr\phi)}_{L^2(\HH)}+\norm{L(\pr\phi)}_{L^2(\HH)}\right)\\
\nn&\les& \norm{\nabla^2\phi_0}_{L^2(\Si_0)}+\norm{\nabla\phi_0}_{L^2(\Si_0)}+\norm{\nabla\phi_1}_{L^2(\Si_0)}+\norm{\nabla F}_{L^2(\MM)}\\
\nn&& +M\ep\left(\sup_{\HH}(\norm{\nabb(\pr\phi)}_{\HH}+\norm{L(\pr\phi)}_{L^2(\HH)}\right)\\
\nn&&+M\ep(\norm{\pr_0(\pr_0\phi)}_{L^2(\MM)}+\norm{\pr(\prb\phi)}_{\lsit{2}}+\norm{\pr\phi}_{\lsit{2}}).
\eeaa
which together with the fact that we may choose $\ep$ such that $M\ep$ is small enough,  yields:
\bea
\lab{nrj8preliminary}&&\norm{\pr(\prb\phi)}_{\lsit{2}}+\sup_{\HH}\left(\norm{\nabb(\pr\phi)}_{L^2(\HH)}+\norm{L(\pr\phi)}_{L^2(\HH)}\right)\\
\nn&\les& \norm{\nabla^2\phi_0}_{L^2(\Si_0)}+\norm{\nabla\phi_0}_{L^2(\Si_0)}+\norm{\nabla\phi_1}_{L^2(\Si_0)}+\norm{\pr F}_{L^2(\MM)}+M\ep\norm{\pr_0(\pr_0\phi)}_{L^2(\MM)}\\
\nn&&+M\ep\norm{\pr\phi}_{\lsit{2}}.
\eea
Also, note that 
\beaa
\norm{\pr\phi}_{\lsit{2}} &\les & \norm{\nab\phi_0}_{L^2(\Si_0)}+\norm{\pr_0\pr\phi}_{\lsit{2}}\\
&\les&  \norm{\nab\phi_0}_{L^2(\Si_0)}+\norm{\pr\pr_0\phi}_{\lsit{2}}+\norm{n^{-1}\nab n\pr_0\phi+\A\pr\phi}_{\lsit{2}}\\
&\les&  \norm{\nab\phi_0}_{L^2(\Si_0)}+\norm{\pr\pr_0\phi}_{\lsit{2}}+\norm{n^{-1}\nab n}_{\lsit{3}}\norm{\pr_0\phi}_{\lsit{6}}\\
&&+\norm{\A}_{\lsit{6}}\norm{\pr\phi}_{\lsit{3}}\\
&\les&  \norm{\nab\phi_0}_{L^2(\Si_0)}+\norm{\pr\prb\phi}_{\lsit{2}}+M\ep\norm{\pr\phi}_{\lsit{2}}
\eeaa
where we used in the last inequality the Sobolev embedding on $\Si_t$ \eqref{sobineqsit}, the bootstrap estimates \eqref{bootA0} for $A_0$ and \eqref{bootA} for $A$, and the estimate \eqref{bootn} for the lapse $n$. Since we may choose $\ep$ such that $M\ep$ is small enough, we infer
\beaa
\norm{\pr\phi}_{\lsit{2}} &\les&  \norm{\nab\phi_0}_{L^2(\Si_0)}+\norm{\pr\prb\phi}_{\lsit{2}}
\eeaa
which together with \eqref{nrj8preliminary} and the fact that we may choose $\ep$ such that $M\ep$ is small enough implies
\bea
\lab{nrj8}&&\norm{\pr(\prb\phi)}_{\lsit{2}}+\sup_{\HH}\left(\norm{\nabb(\pr\phi)}_{L^2(\HH)}+\norm{L(\pr\phi)}_{L^2(\HH)}\right)\\
\nn&\les& \norm{\nabla^2\phi_0}_{L^2(\Si_0)}+\norm{\nabla\phi_0}_{L^2(\Si_0)}+\norm{\nabla\phi_1}_{L^2(\Si_0)}+\norm{\pr F}_{L^2(\MM)}+M\ep\norm{\pr_0(\pr_0\phi)}_{L^2(\MM)}.
\eea

In view of \eqref{nrj8}, we need an estimate for $\pr_0(\pr_0\phi)$. Proceeding as in \eqref{csd15}, we have:
\bea\lab{nrj9}
\norm{\pr_0(\pr_0(\phi))}_{L^2(\MM)}&\les& \norm{\square(\phi)}_{L^2(\MM)}+\norm{\Delta\phi}_{L^2(\MM)}+\norm{\nabla n\cdot\nabla(\phi)}_{L^2(\MM)}\\ 
\nn&\les& \norm{F}_{L^2(\MM)}+\norm{\nabla^2\phi}_{\lsit{2}}+\norm{\nabla n}_{L^\infty}\norm{\nabla\phi}_{\lsit{2}}\\ 
\nn&\les& \norm{F}_{L^2(\MM)}+\norm{\pr^2\phi}_{\lsit{2}}.
\eea
Finally, \eqref{nrj2} and \eqref{nrj7}-\eqref{nrj9} yields:
\beaa
&&\norm{\square(\pr_j\phi)}_{L^2(\MM)}\\
&\les& \norm{\pr_jF}_{L^2(\MM)}+\norm{[\square,\pr_j]\phi}_{L^2(\MM)}\\
&\les & \norm{\pr F}_{L^2(\MM)}+M\ep\left(\sup_{\HH}(\norm{\nabb(\pr\phi)}_{\HH}+\norm{L(\pr\phi)}_{L^2(\HH)})\right)\\
\nn&&+M\ep(\norm{\pr_0(\pr_0\phi)}_{L^2(\MM)}+\norm{\pr(\prb\phi)}_{\lsit{2}})\\
&\les & M\ep(\norm{\nabla^2\phi_0}_{L^2(\Si_0)}+\norm{\nabla\phi_0}_{L^2(\Si_0)}+\norm{\nabla\phi_1}_{L^2(\Si_0)}+             \norm{ F}_{L^2(\MM)} +\norm{\pr F}_{L^2(\MM)})
\eeaa
which together with \eqref{nrj8} and \eqref{nrj9} concludes the proof of the lemma.
\end{proof}

\section{Proof of Proposition \ref{prop:improve1}}\lab{sec:improve1}

Here we derive estimates for $\R, A_0$ and $A$ and thus improve 
 the basic bootstrap assumptions  \eqref{bootR}, \eqref{bootcurvatureflux}, \eqref{bootA} and \eqref{bootA0}.
\subsection{Curvature estimates} We derive the curvature estimates using the Bel-Robinson tensor,
 \beaa
 Q_{\a\b\ga\de}=\R_\a\,^\la\,\ga\,^\si \R_{\b\,\la\,\de\,\si}+\dual \R_\a\,^\la\,\ga\,^\si\dual \R_{\b\,\la\,\de\,\si}
 \eeaa
 Let 
 $$P_\a=Q_{\a\b\ga\de}e_0^\b e_0^\ga e_0^\de.$$ 
Then, we have:
\bea\lab{lybia}
D^\a P_\a=3Q_{\a\b\ga\de}\pi^{\a\b}e_0^\ga e_0^\de,
\eea
where $\pi$ is the deformation tensor of $e_0$. We introduce the  Riemannian metric,
\bea
h_{\a\b}&=&g_{\a\b}+2(e_0)_\a (e_0)_\b
\eea
 and use it to define  the following  space-time  norm for  tensors $U$:
$$|U|^2=U_{\a_1\cdots\a_k}U_{\a_1'\cdots\a'_k} h^{\a_1\a_1'}\cdots h^{\a_k\a'_k}.$$
Given two space-time tensors $U, V$ we denote by $U\c V$ a given contraction between
the two tensors and by $|U\c V|$ the  norm of the contraction according to the above definition.

Let $\HH$ be a weakly regular  null hypersurface with future normal  $L$ such that  ${\bf g}(L,T)=-1$.  Integrating \eqref{lybia} over a space-time region, 
bounded by $\Sigma_0, \Sigma_t$ and $\HH$, and using well-known properties of the Bel-Robinson tensor,
we have:
$$\int_{\Si_t}|\R|^2+\int_{\HH}|\R\c L|^2\les \norm{\R}^2_{L^2(\Si_0)}+\left|\int_{\MM}Q_{\a\b\ga\de}\pi^{\a\b}e_0^\ga e_0^\de\right|\les \ep^2+\left|\int_{\MM}Q_{\a\b\ga\de}\pi^{\a\b}e_0^\ga e_0^\de\right|.$$
We need to estimate the term in the right-hand side of the previous inequality. Note that since $\pi_{00}=0$, $\pi_{0j}=n^{-1}\nabla_jn$, and $\pi_{ij}=k_{ij}$, the bootstrap assumption \eqref{bootR} for $\R$, and the estimates \eqref{bootn} for $n$ yield:
\beaa
\int_{\Si_t}|\R|^2+\int_{\HH}|\R\c L|^2&\les& \ep^2+\norm{\nabla n}_{L^\infty}\norm{\R}^2_{\lsit{2}}+\left|\int_{\MM}Q_{ij\ga\de} k^{ij}e_0^\ga e_0^\de\right|\\
&\les& \ep^2+(M\ep)^3+\left|\int_{\MM}Q_{ij\ga\de}k^{ij}e_0^\ga e_0^\de\right|.
\eeaa

The last term on the right-hand side of the previous inequality is dangerous. Schematically it has the form 
$\left|\int_{\MM}k\R^2\right|.$ 
Typically  this term is  estimated by:
\beaa
\left|\int_{\MM}k\R^2\right|
&\les& \norm{k}_{\lsitt{2}{\infty}} \norm{\R}^2_{\lsit{2}},
\eeaa
requiring   a   Strichartz estimate for $k$ which is false even in flat space.
It is for this reason that we need  the trilinear  bootstrap assumption \eqref{trilinearboot}.
Using it we derive,
\bea\lab{camarche}
\int_{\Si_t}|\R|^2+\int_{\HH}|\R\c L|^2&\les& \ep^2+M^4 \ep^3.
\eea
which, for small $\ep$,   improves  the bootstrap assumptions \eqref{bootR} and  \eqref{bootcurvatureflux}.

\subsection{Improvement of the bootstrap assumption for $A_0$}

Recall the equation for $A_0$, \eqref{eq:YM21}
\bea\lab{tues5}
\De A_0=(A, n^{-1}\nab n) (\pr  \A+\pr_0 A)+ (A, n^{-1}\nab n) \c \A^2.
\eea
After multiplication by $A_0$ and integration by parts, and together with the estimate \eqref{bootn} for $n$, the bootstrap assumptions \eqref{bootA} for $A$ and \eqref{bootA0} for $A_0$,  and  the Sobolev embedding \eqref{sobineqsit}, 
this yields:
\bea\lab{wed0}
&&\norm{\pr A_0}^2_{\lsit{2}}\\
\nn&\les& (\norm{A}_{\lsit{3}}+\norm{n^{-1}\nab n}_{\lsit{3}}) (\norm{\pr  \A}_{\lsit{2}}+\norm{\pr_0 A}_{\lsit{2}})\norm{A_0}_{\lsit{6}}\\
\nn&&+ (\norm{A}_{\lsit{2}}+\norm{n^{-1}\nab n}_{\lsit{2}}) \norm{\A}^2_{\lsit{6}}\norm{A_0}_{\lsit{6}}\\
\nn&\les& M^3\ep^3.
\eea

Also, using the elliptic estimate \eqref{prop:bochsit2} together with \eqref{tues5}, we have:
\bea\lab{wed1}
&&\norm{\pr A_0}_{\lsit{3}}+\norm{\pr^2A_0}_{\lsit{\frac{3}{2}}}\\
\nn&\les& 
\norm{\Delta A_0}_{\lsit{\frac{3}{2}}}  +     \norm{\pr A_0}_{\lsit{2} }  \\
\nn&\les& \norm{\A}_{\lsit{6}}(\norm{\prb A}_{\lsit{2}}+\norm{\pr\A}_{\lsit{2}})+\norm{(A, n^{-1}\nab n)}_{\lsit{3}}    \norm{\A}^2_{\lsit{6}}           \\
\nn&\les& M^2\ep^2
\eea
where we used in the last inequality the estimate \eqref{bootn} for $n$, the bootstrap assumptions \eqref{bootA} on $A$ and \eqref{bootA0} on $A_0$, and the Sobolev embedding \eqref{sobineqsit}.

We will need the following basic $L^p$ elliptic estimate on $\Sigma_t$.
\begin{lemma}\lab{lemma:moser}
Let $\phi$ a scalar function solution of 
$$
\De\phi=f
$$
on $\Sigma_t$, vanishing at infinity. Then for any $3<p<\infty$
\bea\label{eq:ell}
\|\phi\|_{L^p(\Sigma_t)}\les \|f\|_{L^{\frac{3p}{2p+3}}(\Sigma_t)}.
\eea
\end{lemma}

\begin{proof}
Splitting $f$ into its positive and negative parts, if necessary, we may assume, by the maximum principle, that $\phi\ge 0$.
We multiply the equation $\De\phi=f$ by $\phi^{p/3-1}$ and integrate over $\Sigma_t$.
\beaa
 \frac{12(p-3)}{p^2}\int_{\Sigma_t} |\nabla (\phi^{\frac p6})|^2 &=& -\int_{\Sigma_t} \De\phi \phi^{\frac p3-1}\\
 &=&  -\int_{\Sigma_t} f \phi^{\frac p3-1}\\
 & \les &  \|f\|_{L^{\frac{3p}{2p+3}}(\Sigma_t)}  \|\phi\|^{\frac p3-1}_{L^{p}(\Sigma_t)}.
\eeaa
In view of the Sobolev embedding \eqref{sobineqsit}, we have
$$
\|\phi\|_{L^p(\Sigma_t)}^{\frac p3}\les \int_{\Sigma_t} |\nabla (\phi^{\frac p6})|^2
$$
and the desired inequality follows.
\end{proof}

Using Lemma \ref{lemma:moser} with $p=4$  together with \eqref{tues5}, we infer
\bea\lab{wed1bis}
\norm{A_0}_{\lsit{4}} &\les& \norm{\Delta A_0}_{\lsit{\frac{12}{11}}}\\
\nn&\les& \norm{(A, n^{-1}\nab n) (\pr  \A,\pr_0 A)+ (A, n^{-1}\nab n) \c \A^2}_{\lsit{\frac{12}{11}}}\\
\nn&\les& \norm{(A, n^{-1}\nab n)}_{\lsit{\frac{12}{5}}}( \norm{(\pr  \A,\pr_0 A)}_{\lsit{2}}+ \norm{\A}^2_{\lsit{4}})\\
\nn&\les& M^2\ep^2.
\eea

Next, using the Sobolev embedding \eqref{eel3}  together with \eqref{tues5}, we have:
\bea\lab{wed2}
\norm{A_0}_{\lsitt{2}{\infty}}&=& \norm{(-\Delta)^{-1}((A, n^{-1}\nab n) (\pr  \A+\pr_0 A)+ (A, n^{-1}\nab n) \c \A^2)}_{\lsitt{2}{\infty}}\\
\nn&\les & \norm{\A(\pr  \A,\pr_0 A)}_{\lsitt{2}{\frac{14}{9}}}+\norm{\A^3}_{\lsitt{2}{\frac{14}{9}}}+\norm{\A(\pr  \A,\pr_0 A)}_{\lsitt{2}{\frac{13}{9}}}\\
\nn&&+\norm{\A^3}_{\lsitt{2}{\frac{13}{9}}}\\\
\nn&\les & (\norm{\A}_{\lsitt{2}{7}}+\norm{\A}_{\lsit{4}})\norm{(\pr  \A,\pr_0 A)}_{\lsit{2}}\\
\nn&&+\norm{\A}^3_{\lsit{4}}+\norm{\A}^3_{\lsit{6}}\\
\nn&\les& M^3\ep^2,
\eea
where we used the bootstrap assumptions on $\A$, the Sobolev embedding \eqref{sobineqsit} and interpolation.

Next, we consider $\pr_0A_0$. Recall from Proposition \ref{prop:equationpr0A0} that 
\bea\label{wed3}
\Delta\Big((\pr\pr_0n, \pr_0((A_0)_{jl})\Big) &=&  f_1+\pr f_2,
\eea
where $f_1$ is given by:
$$f_1=\pr A_0 (\R, \pr_0 A)+(A, n^{-1}\nabla n) (\pr(\pr_0A_0), A\pr_0A_0)+(\pr(\pr_0n), \A^2)(\R, \prb A,\pr A_0, \A^2)$$
and where $f_2$ is given by:
$$f_2=(A, n^{-1}\nab n)(\R, \prb \A)+A_0(\pr_0 A, \R)+\A^3.$$
In view of the bootstrap assumptions \eqref{bootA} for $A$ and \eqref{bootA0} for $A_0$, we have:
\bea\lab{wed4}
&&\norm{f_1}_{\lsitt{2}{\frac{42}{41}}}+\norm{f_2}_{\lsitt{2}{\frac{14}{9}}}\\
\nn&\les&    \norm{\pr A_0}_{\lsit{\frac{21}{10}}} \norm{(\R, \pr_0 A)}_{\lsit{2}}\\
\nn&&+\norm{(A, n^{-1}\nabla n)}_{\lsit{3}} \norm{(\pr(\pr_0A_0), A\pr_0A_0)}_{\lsitt{2}{\frac{14}{9}}}\\
\nn&&+\norm{(\pr(\pr_0n), \A^2)}_{\lsit{\frac{21}{10}}}\norm{(\R, \prb A,\pr A_0, \A^2)}_{\lsit{2}}\\
\nn&&+\norm{(A, n^{-1}\nab n)}_{\lsitt{2}{7}}\norm{(\R, \prb A, \pr\A)}_{\lsit{2}}\\
\nn&&+\norm{(A, n^{-1}\nab n)}_{\lsit{3}}\norm{\pr_0A_0}_{\lsitt{2}{\frac{42}{13}}}\\
\nn&&+\norm{A_0}_{\lsitt{2}{7}}\norm{(\pr_0 A, \R)}_{\lsit{2}}+\norm{\A}^3_{\lsit{\frac{14}{3}}}\\
\nn&\les& M^3\ep^2.
\eea

We will use the following elliptic estimates on $\Si_t$:
\begin{lemma}\lab{lemma:maru}
Let $v$ a scalar function solution of 
$$\Delta v=f$$
on $\Sigma_t$, vanishing at infinity.
Then, for any $3<p<\infty$ we have the following estimate
$$\norm{v}_{L^p(\Si_t)}+\norm{\pr v}_{L^{\frac{3p}{p+3}}(\Si_t)}\les \norm{f}_{L^{\frac{3p}{2p+3}}(\Si_t)}.$$
\end{lemma}

\begin{lemma}\lab{lemma:marubis}
Let $v$ a scalar function solution of 
$$\Delta v=\pr f$$
on $\Sigma_t$, vanishing at infinity.
Then, for any $3<p<\infty$ we have the following estimate
$$\norm{v}_{L^p(\Si_t)}+\norm{\pr v}_{L^{\frac{3p}{p+3}}(\Si_t)}\les \norm{f}_{L^{\frac{3p}{p+3}}(\Si_t)}.$$
\end{lemma}

The proof of Lemma \ref{lemma:maru} is postponed to Appendix \ref{sec:prooflemmamaru} while the proof of Lemma \ref{lemma:marubis} is postponed to Appendix \ref{sec:prooflemmamarubis}. We now come back to the estimate of $\pr_0A_0$. In view of \eqref{wed3}, 
Lemma \ref{lemma:maru}, Lemma \ref{lemma:marubis} and the estimate \eqref{wed4}, we have:
\beaa
\norm{(\pr\pr_0n, \pr_0((A_0)_{jl})}_{\lsitt{2}{\frac{42}{13}}}+\norm{\pr(\pr\pr_0n, \pr_0((A_0)_{jl})}_{\lsitt{2}{\frac{14}{9}}}\les M^3\ep^2.
\eeaa
Together with the identity \eqref{eq:YM11} and the estimates \eqref{bootn} for $n$, we infer 
\bea\lab{wed8}
\norm{\pr_0A_0}_{\lsitt{2}{\frac{42}{13}}}+\norm{\pr\pr_0A_0}_{\lsitt{2}{\frac{14}{9}}}\les M^3\ep^2.
\eea
Finally, \eqref{wed0}, \eqref{wed1}, \eqref{wed1bis}, \eqref{wed2} and \eqref{wed8} lead to an improvement of the bootstrap assumption \eqref{bootA0} for $A_0$.

\subsection{Improvement of the bootstrap assumption for $A$}

Using the estimates for $\square B_i$ derived in Lemma \ref{prop:waveeqB},   the estimates for $B$  on the initial slice $\Si_0$ obtained in Proposition  \ref{lemma:initialslice}, and the energy estimate   \eqref{energy1} derived in Lemma \ref{lemma:energyestimatebis}, we have:

\bea\lab{tuesday2}
\norm{\pr^2B}_{\lsit{2}}\les \ep+M^2\ep^2.
\eea

Using \eqref{tuesday2} with Lemma \ref{recoverA}, we obtain:
\bea\lab{tuesday3}
\norm{\pr A}_{\lsit{2}}\les \norm{\pr^2B}_{\lsit{2}}+\norm{\pr E}_{\lsit{2}}\les \ep+M^2\ep^2.
\eea

Next, we estimate $\pr_0(A)$. Recall that:
$$\pr_0(A_j)=\pr_j(A_0)+\R_{0j\c\c}+\A^2.$$
Thus, we have:
$$\norm{\pr_0A}_{\lsit{2}}\les \norm{\pr A_0}_{\lsit{2}}+\norm{\R}_{\lsit{2}}+\norm{\A}^2_{\lsit{4}},$$
which together with the bootstrap estimates \eqref{bootA} for $A$, and the improved estimates for $\R$ and $A_0$ yields:
\bea\lab{tuesday4}
\norm{\pr_0A}_{\lsit{2}}\les \ep+ M^2\ep^{\frac{3}{2}}.
\eea
In view of the estimate of Proposition \ref{lemma:initialslice} for $A$ on the initial slice $\Si_0$, we infer
\bea\lab{tuesday4bis}
\norm{A}_{\lsit{2}}&\les& \norm{A_{|_{t=0}}}_{L^2(\Si_0)}+\norm{\pr_0A}_{\lsit{2}}\\
\nn&\les& \ep+ M^2\ep^{\frac{3}{2}}.
\eea
\eqref{tuesday3}, \eqref{tuesday4} and \eqref{tuesday4bis} lead to an improvement of the bootstrap assumption \eqref{bootA} for $A$.

Finally, \eqref{camarche}, \eqref{wed0}, \eqref{wed1}, \eqref{wed1bis}, \eqref{wed2}, \eqref{wed8},  \eqref{tuesday3}, \eqref{tuesday4} and \eqref{tuesday4bis} yield the improved estimates \eqref{bootRimp}, \eqref{bootcurvaturefluximp}, \eqref{bootA0imp} and \eqref{bootAimp}. This concludes the proof of Proposition \ref{prop:improve1}.

\section{Parametrix for the wave equation}\lab{sec:parametrix}

Let $u_\pm$ be two families of scalar functions  defined on the space-time $\MM$ and indexed by $\om\in\SSS^2$,  satisfying the eikonal equation 
$\g^{\a\b} \pa_\a u_\pm\, \pa_\b u_\pm=0$
for each $\om\in\SSS^2$. We also denote $\uom_\pm(t,x)=u_\pm(t,x,\om)$. We have the freedom of choosing $\uom_\pm$ on the initial slice $\Si_0$, and in order for the results in \cite{param2} , \cite{param4} to apply, we need to initialize $\uom_\pm$ on $\Si_0$ as in \cite{param1}. The dependence of $\uom_\pm$ on $\omega$ 
is manifested in particular through the requirement that on $\Sigma_0$ the behavior of $u_\pm$ asymptotically
approaches that of $x\cdot\omega$.

Let $\HH_{\uom_\pm}$ denote the {\it null} level hypersurfaces of $\uom_\pm$. Let $\Lom_\pm$ be their null normals, fixed by the condition 
${\bf g}(\Lom_\pm, T)=\mp 1$. Let the vectorfield tangent to $\Si_t$ $\Nom_\pm$ be defined such as to satisfy:
$$\Lom_\pm=\pm e_0+\Nom_\pm.$$
We pick $(\eom_\pm)_A,\, A=1, 2$ vectorfields in $\Si_t$ such that together with $\Nom_\pm$ we obtain an orthonormal basis of $\Si_t$. Finally, we denote by $\nabb_\pm$          derivatives in the directions $(\eom_\pm)_A,\, A=1, 2$.

\begin{remark}
Note that  from the results in \cite{param3}, $\HH_{\uom_\pm}$ satisfy assumptions \eqref{assumptionH1} and \eqref{assumptionH2} (\eqref{assumptionH1} is a straightforward consequence of Theorem 2.18 in \cite{param3} while \eqref{assumptionH2} corresponds to (3.55) in \cite{param3}).
\end{remark}
We record  the following Sobolev embedding/trace type inequality  on $\HH_u$ for functions  defined on $\MM$,    derived  in \cite{param3} (see sections 3.5 in that paper).
\begin{lemma}[An embedding on $\HH$ \cite{param3}]
For any null hypersurface $\HH_u$, defined as above,  and for any $\Si_t$-tangent tensor $F$, we have:
\bea\lab{HversusSit}
\norm{F}_{L^2(\HH)}\lesssim \norm{\nabla F}_{\lsit{\frac{3}{2}}}+\norm{F}_{\lsit{2}},
\eea
and for any $2\leq p\leq 4$:
\bea\lab{HversusSitbis}
\norm{F}_{L^p(\HH)}\lesssim \norm{\nabla F}_{\lsit{2}}+\norm{F}_{\lsit{2}}.
\eea
\end{lemma}

For any pair of functions $f_\pm$ on $\RRR^3$, we define the following scalar function on $\MM$:
$$\psi[f_+, f_-](t,x)=\int_{\SSS^2} \int_0^\infty  e^{i \la \uom_+(t,x)}f_+(\la\om)\la^2d\la d\om
 +\int_{\SSS^2} \int_0^\infty e^{i\la \uom_-(t,x)}f_-(\la\om)\la^2d\la d\om.$$
We appeal to the following result from \cite{param2} \cite{param4}:
\begin{theorem}[Theorem 2.10 in \cite{param2} and Theorem 2.15 in  \cite{param4}]\lab{prop:estparam}
Let $\phi_0$ and $\phi_1$ two scalar functions on $\Si_0$. Then, there is a unique pair of functions $(f_+, f_-)$ such that: 
$$\psi[f_+, f_-]|_{\Si_0}=\phi_0\textrm{ and }\pr_0(\psi[f_+, f_-])|_{\Si_0}=\phi_1.$$
Furthermore, $f_\pm$ satisfy the following estimates:
$$\norm{\la f_+}_{L^2(\RRR^3)}+\norm{\la f_-}_{L^2(\RRR^3)}\les \norm{\nabla\phi_0}_{L^2(\Si_0)}+\norm{\phi_1}_{L^2(\Si_0)},$$
and:
$$\norm{\la^2f_+}_{L^2(\RRR^3)}+\norm{\la^2f_-}_{L^2(\RRR^3)}\les \norm{\nabla^2\phi_0}_{L^2(\Si_0)}+\norm{\nabla\phi_1}_{L^2(\Si_0)}.$$
Finally, $\square\psi[f_+, f_-]$ satisfies the following estimates:
$$\norm{\square\psi[f_+, f_-]}_{L^2(\MM)}\les M\ep(\norm{\nabla\phi_0}_{L^2(\Si_0)}+\norm{\phi_1}_{L^2(\Si_0)}),$$
and:
$$\norm{\pr\square\psi[f_+, f_-]}_{L^2(\MM)}\les M\ep(\norm{\nabla^2\phi_0}_{L^2(\Si_0)}+\norm{\nabla\phi_1}_{L^2(\Si_0)}).$$
\end{theorem}

\begin{remark}
The content of Theorem 10.3 is a deep statement about existence 
of a generalized Fourier transform and its inverse on $\Sigma_0$, and  existence and accuracy of a
parametrix for the scalar wave equation on $\MM$, with merely $L^2$ curvature bounds assumptions on the
ambient geometry.
The existence of $f_\pm$ and the first two estimates of Theorem \ref{prop:estparam} are proved in \cite{param2}, while the last two estimates in Theorem \ref{prop:estparam} are proved in \cite{param4}.
\end{remark}

\begin{remark}
As $\uom_\pm$ satisfy the eikonal equation, note that 
\bea\lab{formulainparam4}
\square\psi[f_+, f_-] &=& i\int_{\SSS^2} \int_0^\infty \square\uom_+(t,x) e^{i \la \uom_+(t,x)}f_+(\la\om)\la^3d\la d\om\\
\nn&& +i\int_{\SSS^2} \int_0^\infty \square\uom_-(t,x) e^{i\la \uom_-(t,x)}f_-(\la\om)\la^3d\la d\om
 \eea
 which is a sum of two Fourier integral operators with phase $\uom_\pm$ and symbol $\square\uom_\pm$. The last two estimates in Theorem \ref{prop:estparam} are proved in \cite{param4} by considering one of the two half waves in \eqref{formulainparam4}, as the estimate is identical for both half waves. Note also the identity (2.28) in \cite{param4} which allows to rewrite the symbol of the Fourier integral operator \eqref{formulainparam4} under the form appearing in Theorem 2.15 of \cite{param4}. 
\end{remark}

We associate to any pair of functions $\phi_0, \phi_1$ on $\Si_0$ the function $\Psi_{om}[\phi_0,\phi_1]$ defined for $(t,x)\in\MM$ as:
$$\Psi_{om}[\phi_0,\phi_1]=\psi[f_+, f_-]$$
where $(f_+, f_-)$ is defined in view of Theorem \ref{prop:estparam} as the unique pair of functions associated to $(\phi_0, \phi_1)$. In particular, we obtain:
$$\norm{\la f_+}_{L^2(\RRR^3)}+\norm{\la f_-}_{L^2(\RRR^3)}\les \norm{\nabla\phi_0}_{L^2(\Si_0)}+\norm{\phi_1}_{L^2(\Si_0)},$$
$$\norm{\la^2f_+}_{L^2(\RRR^3)}+\norm{\la^2f_-}_{L^2(\RRR^3)}\les \norm{\nabla^2\phi_0}_{L^2(\Si_0)}+\norm{\nabla \phi_1}_{L^2(\Si_0)},$$
\bea\lab{par1:bis}
\norm{\square\Psi_{om}[\phi_0,\phi_1]}_{L^2(\MM)}\les M\ep(\norm{\nabla\phi_0}_{L^2(\Si_0)}+\norm{\phi_1}_{L^2(\Si_0)}),
\eea
and:
\bea\lab{par1}
\norm{\pr\square\Psi_{om}[\phi_0,\phi_1]}_{L^2(\MM)}\les M\ep(\norm{\nabla^2\phi_0}_{L^2(\Si_0)}+\norm{\nabla \phi_1}_{L^2(\Si_0)}).
\eea

Next, let $\uoms_\pm$ two families, indexed by $\om\in\SSS^2$ and $s\in \RRR$\footnote{In fact, we need to restrict $s$ to an open interval of $[0,1]$ for which the bootstrap assumptions \eqref{bootR} \eqref{bootcurvatureflux} on $\R$ hold}, of scalar functions on the space-time $\MM$ satisfying the eikonal equation for each $\om\in\SSS^2$ and $s\in\RRR$. We have the freedom of choosing $\uoms_\pm$ on the slice $\Si_s$, and in order for the results in \cite{param2} \cite{param4} to apply, we need to initialize $\uoms_\pm$ on $\Si_s$ as in \cite{param1}.  
Note that the families $\uom_\pm$ correspond to $\uoms$ with the choice $s=0$. For any pair of functions $f_\pm$ on $\RRR^3$, and for any $s\in\RRR$, we define the following scalar function on $\MM$:
$$\psi_s[f_+, f_-](t,x,s)=\int_{\SSS^2} \int_0^\infty  e^{i \la \uoms_+(t,x)}f_+(\la\om)\la^2d\la d\om
 +\int_{\SSS^2} \int_0^\infty e^{i\la \uoms_-(t,x)}f_-(\la\om)\la^2d\la d\om.$$
We have the following straightforward corollary of Theorem \ref{prop:estparam}:
\begin{corollary}\lab{cor:estparam}
Let $s\in\RRR$. Let $\phi_0$ and $\phi_1$ two scalar functions on $\Si_s$. Then, there is a unique pair of functions $(f_+, f_-)$ such that: 
$$\psi_s[f_+, f_-]|_{\Si_s}=\phi_0\textrm{ and }\pr_0(\psi_s[f_+, f_-])|_{\Si_s}=\phi_1.$$
Furthermore, $f_\pm$ satisfy the following estimates:
$$\norm{\la f_+}_{L^2(\RRR^3)}+\norm{\la f_-}_{L^2(\RRR^3)}\les \norm{\nabla\phi_0}_{L^2(\Si_s)}+\norm{\phi_1}_{L^2(\Si_s)},$$
and:
$$\norm{\la^2f_+}_{L^2(\RRR^3)}+\norm{\la^2f_-}_{L^2(\RRR^3)}\les \norm{\nabla^2\phi_0}_{L^2(\Si_s)}+\norm{\nabla\phi_1}_{L^2(\Si_s)}.$$
Finally, $\square\psi_s[f_+, f_-]$ satisfies the following estimates:
$$\norm{\square\psi_s[f_+, f_-]}_{L^2(\MM)}\les M\ep(\norm{\nabla\phi_0}_{L^2(\Si_s)}+\norm{\phi_1}_{L^2(\Si_s)}),$$
and:
$$\norm{\pr\square\psi_s[f_+, f_-]}_{L^2(\MM)}\les M\ep(\norm{\nabla^2\phi_0}_{L^2(\Si_s)}+\norm{\nabla\phi_1}_{L^2(\Si_s)}).$$
\end{corollary}

Next, for any $s\in \RRR$, we associate to any function $F$ on $\Si_s$ the function $\Psi(t,s)F$ defined for $(t,x)\in\MM$ as:
$$\Psi(t,s)F=\psi_s[f_+, f_-](t)$$
where $(f_+, f_-)$ is defined in view of Corollary \ref{cor:estparam} as the unique pair of functions associated to the choice $(\phi_0, \phi_1)=(0,-nF)$. In particular, we obtain in view of Corollary \ref{cor:estparam} and the control of the lapse $n$ given by \eqref{bootn}:
$$\norm{\la f_+}_{L^2(\RRR^3)}+\norm{\la f_-}_{L^2(\RRR^3)}\les \norm{F}_{L^2(\Si_s)},$$
$$\norm{\la^2f_+}_{L^2(\RRR^3)}+\norm{\la^2f_-}_{L^2(\RRR^3)}\les \norm{\nabla F}_{L^2(\Si_s)},$$
\bea\lab{par2:bis}
\norm{\square\Psi(t,s)F}_{L^2(\MM)}\les M\ep\norm{F}_{L^2(\Si_s)},
\eea
and:
\bea\lab{par2}
\norm{\pr\square\Psi(t,s)F}_{L^2(\MM)}\les M\ep\norm{\nabla F}_{L^2(\Si_s)}.
\eea

\begin{remark}\lab{rem:duhamel}
Note that we have 
$$\square\left(\int_0^t\Psi(t,s)F(s)ds\right)=F(t)+\int_0^t\square\Psi(t,s)F(s)ds.$$
\end{remark}

Now, we are in position to construct a parametrix for the wave equation \eqref{eq:wave}. 

\begin{theorem}[Representation formula]\lab{lemma:parametrixconstruction}
Let $F$ a scalar function on $\MM$, and let $\phi_0$ and $\phi_1$ two scalar functions on $\Si_0$. Let $\phi$ the solution of the wave equation \eqref{eq:wave} on $\MM$. Then, there is a sequence of scalar functions $(\phi^{(j)}, F^{(j)})$, $j\ge 0$ on $\MM$, defined according to:
\beaa
\phi^{(0)}=\Psi_{om}[\phi_0,\phi_1]+\int_0^t\Psi(t,s)F^{(0)}(s,.)ds,\qquad F^{(0)}=F
\eeaa
and for all $j\geq 1$:
$$\phi^{(j)}=\int_0^t\Psi(t,s)F^{(j)}(s,.)ds,\qquad F^{(j)}=-\square\phi^{(j-1)}+F^{(j-1)}$$
such that,
$$\phi=\sum_{j=0}^{+\infty}\phi^{(j)},$$ 
 and    such that  $\phi^{(j)}$ and $F^{(j)}$ satisfy the following estimates:
$$\norm{\prb\phi^{(j)}}_{\lsit{2}}+\norm{F^{(j)}}_{L^2(\MM)}\les (M\ep)^j(\norm{\nabla\phi_0}_{L^2(\Si_0)}+\norm{\phi_1}_{L^2(\Si_0)}+\norm{F}_{L^2(\MM)}),$$
and:
$$\norm{\pr\prb\phi^{(j)}}_{\lsit{2}}+\norm{\pr F^{(j)}}_{L^2(\MM)}\les (M\ep)^j(\norm{\nabla^2\phi_0}_{L^2(\Si_0)}+\norm{\nabla\phi_1}_{L^2(\Si_0)}+\norm{\pr F}_{L^2(\MM)}).$$
\end{theorem}

\begin{proof}
Let us define:
$$F^{(0)}=F\textrm{ and }\phi^{(0)}=\Psi_{om}[\phi_0,\phi_1]+\int_0^t\Psi(t,s)F^{(0)}(s,.)ds.$$
Then, we define iteratively for $j\geq 1$:
$$F^{(j)}=-\square\phi^{(j-1)}+F^{(j-1)}\textrm{ and }\phi^{(j)}=\int_0^t\Psi(t,s)F^{(j)}(s,.)ds.$$
In view of Remark \ref{rem:duhamel}, note that for $j\geq 1$:
$$\square\phi^{(j)}=\square\int_0^t\Psi(t,s)F^{(j)}(s,.)ds=F^{(j)}+\int_0^t\square\Psi(t,s)F^{(j)}(s,.)ds,$$
which yields:
$$F^{(j+1)}=-\int_0^t\square\Psi(t,s)F^{(j)}(s,.)ds.$$
Thus, we obtain in view of \eqref{par2:bis} and \eqref{par2}:
$$\norm{F^{(j+1)}}_{L^2(\MM)}\les M\ep\norm{F^{(j)}}_{L^2(\MM)},$$
and:
$$\norm{\pr F^{(j+1)}}_{L^2(\MM)}\les M\ep\norm{\pr F^{(j)}}_{L^2(\MM)}.$$ 
Therefore, we obtain for all $j\geq 2$:
\bea\lab{par3:bis}
\norm{F^{(j)}}_{L^2(\MM)}\les (M\ep)^{j-1}\norm{F^{(1)}}_{L^2(\MM)},
\eea
and:
\bea\lab{par3}
\norm{\pr F^{(j)}}_{L^2(\MM)}\les (M\ep)^{j-1}\norm{\pr F^{(1)}}_{L^2(\MM)}.
\eea
Also, we have:
$$\square\phi^{(0)}=F^{(0)}+\square\Psi_{om}[\phi_0,\phi_1]+\int_0^t\square\Psi(t,s)F(s,.)ds,$$
This yields:
$$F^{(1)}=-\square\Psi_{om}[\phi_0,\phi_1]-\int_0^t\square\Psi(t,s)F(s,.)ds$$
which together with \eqref{par1:bis}, \eqref{par1}, \eqref{par2:bis} and \eqref{par2} implies:
$$\norm{F^{(1)}}_{L^2(\MM)}\les M\ep(\norm{\nabla\phi_0}_{L^2(\Si_0)}+\norm{\phi_1}_{L^2(\Si_0)}+\norm{F}_{L^2(\MM)}),$$
and:
$$\norm{\pr F^{(1)}}_{L^2(\MM)}\les M\ep(\norm{\nabla^2\phi_0}_{L^2(\Si_0)}+\norm{\nabla\phi_1}_{L^2(\Si_0)}+\norm{\pr F}_{L^2(\MM)}).$$
Together with \eqref{par3:bis} and \eqref{par3}, we obtain for any $j\geq 1$:
\bea\lab{par4:bis}
\norm{F^{(j)}}_{L^2(\MM)}\les (M\ep)^j(\norm{\nabla\phi_0}_{L^2(\Si_0)}+\norm{\phi_1}_{L^2(\Si_0)}+\norm{F}_{L^2(\MM)}),
\eea
and:
\bea\lab{par4}
\norm{\pr F^{(j)}}_{L^2(\MM)}\les (M\ep)^j(\norm{\nabla^2\phi_0}_{L^2(\Si_0)}+\norm{\nabla\phi_1}_{L^2(\Si_0)}+\norm{\pr F}_{L^2(\MM)}).
\eea

We now estimate $\phi^{(j)}, j\geq 1$. For $j\geq 1$, $\phi^{(j)}$ satisfies the following wave equation:
\begin{displaymath}
\left\{\begin{array}{l}
\square\phi^{(j)}=F^{(j)}-F^{(j+1)},\\
\phi^{(j)}|_{\Si_0}=0,\, \pr_0(\phi^{(j})|_{\Si_0}=0. 
\end{array}\right.
\end{displaymath}
which together with Lemmas \ref{lemma:energyestimate} and \ref{lemma:energyestimatebis}, \eqref{par4:bis} and \eqref{par4}  yields:
\bea\lab{par5:bis}
\norm{\prb\phi^{(j)}}_{\lsit{2}}\les (M\ep)^j(\norm{\nabla\phi_0}_{L^2(\Si_0)}+\norm{\phi_1}_{L^2(\Si_0)}+\norm{F}_{L^2(\MM)}),
\eea
and:
\bea\lab{par5}
\norm{\pr(\prb\phi^{(j)})}_{\lsit{2}}\les (M\ep)^j(\norm{\nabla^2\phi_0}_{L^2(\Si_0)}+\norm{\nabla\phi_1}_{L^2(\Si_0)}+\norm{\pr F}_{L^2(\MM)}).
\eea

Now, we have:
$$\square\left(\sum_{j=0}^J\phi^{(j)}\right)=\sum_{j=0}^J(F^{(j)}-F^{(j+1)})=F-F^{(J+1)},$$
which together with \eqref{par4} and \eqref{par5} yields in the limit $j\rightarrow +\infty$:
$$\square\left(\sum_{j=0}^{+\infty}\phi^{(j)}\right)=F.$$
Note also that 
$$\phi^{(0)}|_{\Si_0}=\phi_0\textrm{ and }\pr_0\phi^{(0)}|_{\Si_0}=\phi_1,$$
while for all $j\geq 1$, we have:
$$\phi^{(j)}|_{\Si_0}=0\textrm{ and }\pr_0\phi^{(j)}|_{\Si_0}=0.$$
Thus, $\sum_{j=0}^{+\infty}\phi^{(j)}$ satisfies the wave equation \eqref{eq:wave}, and by uniqueness, we have:
$$\phi=\sum_{j=0}^{+\infty}\phi^{(j)}.$$
This concludes the proof of the theorem.
\end{proof}

\section{Proof of Proposition \ref{prop:improve2} (part 1)}\lab{sec:improve2}

The goal of this and next sections   is to prove Proposition \ref{prop:improve2}. This requires  the use of the representation formula of 
Theorem \ref{lemma:parametrixconstruction}. In this section  we derive the improved bilinear estimate \eqref{bil1imp}, \eqref{bil5imp}, \eqref{bil6imp}, \eqref{bil8imp} and \eqref{bil8imp} of Proposition \ref{prop:improve2}. We also  derive  the improved trilinear estimate  \eqref{trilinearbootimp}.

\subsection{Improvement of the bilinear bootstrap assumptions I}\lab{sec:proofbil1}  We prove the bilinear  estimates
   \eqref{bil1imp}, \eqref{bil5imp}, \eqref{bil6imp}, \eqref{bil8imp}, \eqref{bil8imp}. These bilinear estimates all involve the $L^2(\MM)$ norm of quantities of the type:
$$\CC(U,\pr\phi),$$
where $\CC(U,\pr\phi)$ denotes a contraction with respect to one index between a tensor $U$ and $\pr\phi$, for $\phi$ a solution of the scalar wave equation \eqref{eq:wave} with $F, \phi_0$ and $\phi_1$ satisfying the estimate:
$$\norm{\nabla^2\phi_0}_{L^2(\Si_0)}+\norm{\nabla\phi_1}_{L^2(\Si_0)}+\norm{\pr F}_{L^2(\MM)}\les M\ep.$$
In particular, we may use the parametrix constructed in Lemma \ref{lemma:parametrixconstruction} for $\phi$:
$$\phi=\sum_{j=0}^{+\infty}\phi^{(j)},$$
with:
$$\phi^{(0)}=\Psi_{om}[\phi_0,\phi_1]+\int_0^t\Psi(t,s)F(s,.)ds,$$
and for all $j\geq 1$:
$$\phi^{(j)}=\int_0^t\Psi(t,s)F^{(j)}(s,.)ds.$$
Thus, we need to estimate the norm in $L^2(\MM)$ of contractions of quantities of the type:
$$\CC(U, \pr(\Psi_{om}[\phi_0,\phi_1]))+\sum_{j=0}^{+\infty}\int_0^t\CC(U, \pr(\Psi(t,s)F^{(j)}(s,.)))ds.$$
After using the definition of $\Psi_{om}$ and $\Psi(t,s)$, and the estimates for $F^{(j)}$ provided by Lemma \ref{lemma:parametrixconstruction}, this reduces to estimating:
$$\int_{\SSS^2} \int_0^\infty  \CC(U, \pr(e^{i \la \uom_+(t,x)}))f_+(\la\om)\la^2d\la d\om
 +\int_{\SSS^2} \int_0^\infty \CC(U,\pr(e^{i\la \uom_-(t,x)}))f_-(\la\om)\la^2d\la d\om,$$
where $f_\pm$ in view of Theorem \ref{prop:estparam} and the estimates for $F, \phi_0$ and $\phi_1$ 
satisfies:
$$\norm{\la^2f_\pm}_{L^2(\RRR^3)}\les M\ep.$$
Since both half-wave parametrices are estimated in the same way, the bilinear estimates \eqref{bil1},  \eqref{bil5}, \eqref{bil6}, \eqref{bil8} and \eqref{bil8} all estimate the norm in $L^2(\MM)$ of  contractions of quantities of the type:
$$\int_{\SSS^2} \int_0^\infty \CC(U,\pr(e^{i \la \uom(t,x)}))f(\la\om)\la^2d\la d\om,$$
where $f$ satisfies:
\bea\lab{shouf}
\norm{\la^2f}_{L^2(\RRR^3)}\les M\ep.
\eea
Now, we have:
$$\pr_j(e^{i \la \uom})=i\la e^{i \la \uom}\pr_j(\uom),$$
and the gradient of $\uom$ on $\Si_t$ is given by:
$$\nabla(\uom)=\bom^{-1}\Nom,$$
where $\bom=|\nabla(\uom)|^{-1}$ is the null lapse, and 
$$\Nom=\frac{\nabla\uom}{|\nabla\uom|}$$ 
is the unit normal to $\HH_{\uom}\cap\Si_t$ along $\Si_t$.  Thus, the bilinear estimates \eqref{bil1},  \eqref{bil5}, \eqref{bil6}, \eqref{bil8} and \eqref{bil8} all  reduce to  $L^2(\MM)$-estimates of   expressions of the form:

\bea
\frak{C}[U,f]:=\int_{\SSS^2} \int_0^\infty e^{i \la \uom(t,x)}\bom^{-1}\CC(U,\Nom)f(\la\om)\la^3d\la d\om,
\eea
where $f$ satisfies \eqref{shouf}. 

To estimate $\frak{C}[U,f]$  we  follow the strategy of \cite{Kl-R3}. 
\bea\lab{dorbia}
\left\|\frak{C}[U,f]\right\|_{L^2(\MM)}&\les&
 \int_{\SSS^2} \left\|  \bom^{-1}\CC(U, \Nom)\left(\int_0^{+\infty}e^{i \la \uom(t,x)}f(\la\om)\la^3d\la\right)\right\|_{L^2(\MM)} d\om\\
&\les & \int_{\SSS^2} \norm{\bom^{-1}}_{L^\infty(\MM)}\norm{\CC(U, \Nom)}_{\luom{2}}\left\|\int_0^{+\infty}e^{i \la \uom(t,x)}f(\la\om)\la^3d\la\right\|_{L^2_{\uom}} d\om          \nn\\
&\les & \left(\sup_{\om\in\SSS^2}\norm{\bom^{-1}}_{L^\infty(\MM)}\right)\left(\sup_{\om\in\SSS^2}\norm{\CC(U, \Nom)}_{\luom{2}}\right)\left(\int_{\SSS^2}\norm{\la^3f(\la\om)}_{L^2_\la}d\om\right)\nn\\
&\les & \left(\sup_{\om\in\SSS^2}\norm{\bom^{-1}}_{L^\infty(\MM)}\right)\left(\sup_{\om\in\SSS^2}\norm{\CC(U, \Nom)}_{\luom{2}}\right)\norm{\la^2f}_{L^2(\RRR^3)},\nn
\eea
where we used Plancherel in $\la$ and Cauchy Schwarz in $\om$. Now, since $\uom$ has been initialized on $\Si_0$ as in \cite{param1}, and satisfies the eikonal equation on $\MM$, the results in \cite{param3} (see section 4.8 in that paper) under the assumption of Theorem \ref{th:mainter} imply:
$$\sup_{\om\in\SSS^2}\norm{\bom^{-1}}_{L^\infty(\MM)}\les 1.$$
Together with the fact that $f$ satisfies \eqref{shouf}, and with \eqref{dorbia}, we finally obtain:
\bea\lab{bilproof} 
&&\left\|\int_{\SSS^2} \int_0^\infty  e^{i \la \uom(t,x)}\bom^{-1}\CC(U, \Nom)f(\la\om)\la^3d\la d\om\right\|_{L^2(\MM)}\\
\nn&\les & M\ep\left(\sup_{\om\in\SSS^2}\norm{\CC(U, \Nom)}_{\luom{2}}\right).
\eea

It remains to estimate the right-hand side of \eqref{bilproof} for the contractions appearing in the bilinear estimates \eqref{bil1imp}, \eqref{bil5imp}, \eqref{bil6imp}, \eqref{bil8imp} and \eqref{bil8imp}. Since all the estimates in the proof are uniform in $\om$, we drop the index $\om$ to ease the notations.

\begin{remark}
In the    proof of  bilinear estimates \eqref{bil1imp}, \eqref{bil5imp}, \eqref{bil6imp}, \eqref{bil8imp} and \eqref{bil8imp}, 
the tensor $U$ appearing  in the expression $\CC(U, N)$ is either $\R $ or    derivatives of solutions $\phi$ of a 
a scalar wave equation.    In view of the bootstrap assumption \eqref{bootcurvatureflux} for the curvature flux,  as well  as the energy estimate for the wave equation in Lemma \ref{lemma:energyestimate}, we can control 
$ \| \CC(U, N)\|_{\lu{2}}$  as long as we can show that $\CC(U, N)$ can be expressed  in terms of,
$$\R\c L,\, \nabb\phi\textrm{ and }L(\phi).$$
In other words, our goal is to check that the term $\CC(U, N)$  does not involve the dangerous terms of the type:
$$\underline{\a}\textrm{ and }\Lb\phi$$
where $\Lb$ is the vectorfield defined as $\Lb=2T-L$, and $\underline{\a}$ is the two tensor on $\Si_t\cap \HH_u$ defined as:
$$\underline{\a}_{AB}=\R_{\Lb\,A\,\Lb\,B}.$$
\end{remark}

 \subsubsection{Proof of   \eqref{bil1imp}}
 Since $A=\curl(B)+E$ in view of Lemma \ref{recoverA}, we have:
\bea\lab{sodeska}
\norm{A^j\pr_j(A)}_{L^2(\MM)}&\les& \norm{(\curl(B))^j\pr_j(A)}_{L^2(\MM)}+\norm{E}_{\lsitt{2}{\infty}}\norm{\pr A}_{\lsit{2}}\\
\nn&\les& \norm{(\curl(B))^j\pr_j(A)}_{L^2(\MM)}+M^3\ep^3,
\eea
where we used in the last inequality Lemma \ref{recoverA} for $E$, and the bootstrap assumption \eqref{bootA} for $A$. Next, we estimate $\norm{(\curl(B))^j\pr_j(A)}_{L^2(\MM)}$. Recall that we have:
$$(\curl(B))^j\pr_j(A)=\in_{jmn}\pr_m(B_n)\pr_j(A).$$

We are now ready to apply  the representation theorem \ref{lemma:parametrixconstruction}  to $B$. 
Indeed, according to  Proposition  \ref{prop:waveeqB},  and Proposition   \ref{lemma:estB},    we have
\bea
\square B=F, \qquad 
\|\pr F\|_{L^2(\MM)}&\les& M^2\ep^2\label{squareB.}\\
\norm{\prb B(0)}_{L^2(\Si_0)}+\norm{\pr^2 B(0)}_{L^2(\Si_0)}+\norm{\pr(\pr_0 B(0))}_{L^2(\Si_0)}&\les& M\ep.\nn
\eea

We are thus in a position to apply  the reduction discussed in the  subsection above and  reduce our desired bilinear estimate to an estimate for,
\beaa
\CC(U, N)&=&\in_{jm\c}N_m\pr_j(A)
\eeaa
Now, we decompose $\pr_j$ on the orthonormal frame $N, f_A, A=1, 2$ of $\Si_t$, where we recall that $f_A, A=1, 2$ denotes an orthonormal basis of $\HH_u\cap \Si_t$. We have schematically:
\bea\lab{decompositionpr}
\pr_j=N_jN+\nabb,
\eea
where $\nabb$ denotes derivatives which are tangent to $\HH_u\cap \Si_t$. Thus, we have:
$$\CC(U,N)=\in_{jm\c}N_mN_j\pr_N(A)+\nabb(A)=\nabb(A),$$
where we have  used the antisymmetry of $\in_{jm\c}$ in the last equality. Therefore, we obtain in this case:
$$\norm{\CC(U, N)}_{\lu{2}}\les \norm{\nabb(A)}_{\lu{2}}.$$
It remains  to  estimate $\norm{\nabb(A)}_{\lu{2}}$. Since we have $A=\curl(B)+E$ in view of Lemma \ref{recoverA}, we obtain:
\beaa
\norm{\nabb(A)}_{\lu{2}}&\les&  \norm{\nabb(\pr B)}_{\lu{2}}+\norm{\nabb(E)}_{\lu{2}}\\
&\les&  \norm{\nabb(\pr B)}_{\lu{2}}+\norm{\pr E}_{\lsit{2}}+\norm{\pr^2E}_{\lsit{\frac{3}{2}}}\\
&\les & \norm{\nabb(\pr B)}_{\lu{2}}+M\ep,\\
&\les& M\ep
\eeaa
where we used the embedding \eqref{HversusSit} and the estimates for $E$ given by Lemma \ref{recoverA}. Furthermore, we have in view of Proposition \ref{prop:waveeqB} and Lemma \ref{lemma:energyestimatebis} the following estimate for $B$:
$$\norm{\nabb(\pr B)}_{\lu{2}}\les M\ep.$$
We finally obtain:
$$\norm{\nabb(A)}_{\lu{2}}\les M\ep.$$
This improves the bilinear estimate \eqref{bil1}.

\subsubsection{Proof of \eqref{bil5imp}}   In view of Lemma \ref{recoverA}  $A=\curl(B)+E$. Arguing as in \eqref{sodeska}, we reduce the proof to the estimate of: 
$$\norm{(\curl B)^j\pr_j(\prb B)}_{L^2(\MM)}.$$
Since $B$ satisfies  the wave equation \eqref{squareB.}, the quantity $\CC(U, N)$ is in this case, 
$$\CC(U,N)=\in_{jm\c}N_m\pr_j(\prb B).$$
Using the decomposition of $\pr_j$ \eqref{decompositionpr} and the antisymmetry of $\in_{jm\c}$, we have schematically:
\bea\lab{eq:usefullater}
\in_{jm\c}N_m\pr_j(\prb B)&=&\in_{jm\c}N_mN_j\pr_N(\prb B)+\nabb(\prb B)\\
\nn&=&\nabb(\prb B)\\
\nn&=&\nabb(\pr B)+\nabb(\pr_0 B)\\
\nn&=&\nabb(\pr B)+L(\pr B)+(\D L+\D N+A)\pr(B),
\eea
where in the last equality  we used the decomposition \eqref{decompositionpr0} for $\nabb(\pr_0(B))$.  Together with the assumptions \eqref{assumptionH1} and \eqref{assumptionH2} on $\HH_u$, and the Sobolev embedding \eqref{HversusSitbis} on $\HH_u$, we obtain\footnote{We    ignore  in the estimates   below      the lower order term   $\|\pr B \|_{L^2(\HH_u)}$  which is easier to handle.           }:
\beaa
&&\norm{\in_{jm\c}N_m\pr_j(\prb B)}_{L^2(\HH_u)}\\
&\les& \norm{\nabb(\pr B)}_{L^2(\HH_u)}+\norm{L(\pr B)}_{L^2(\HH_u)}\\
&&+(\norm{\D N}_{L^3(\HH_u)}+\norm{\D L}_{L^3(\HH_u)}+\norm{A}_{L^3(\HH_u)})\norm{\pr B}_{L^6(\HH_u)}\\
&\les& (1+\norm{A}_{\lsit{2}}+\norm{\pr A}_{\lsit{2}})(\norm{\nabb(\pr B)}_{L^2(\HH_u)}+\norm{L(\pr B)}_{L^2(\HH_u)})\\
&\les& \norm{\nabb(\pr B)}_{L^2(\HH_u)}+\norm{L(\pr B)}_{L^2(\HH_u)},
\eeaa
where we used the bootstrap assumption \eqref{bootA} for $A$ in the last inequality. Now, we have in view of Proposition  \ref{prop:waveeqB} and Lemma \ref{lemma:energyestimatebis} the following estimate for $B$:
$$\norm{\nabb(\pr B)}_{\lu{2}}+\norm{L(\pr B)}_{\lu{2}}\les M\ep,$$
which improves the bilinear estimate \eqref{bil5}.

\subsubsection{Proof of  \eqref{bil6imp}} Since $B$ satisfies a wave equation in view of Lemma \ref{prop:waveeqB}, the quantity $\CC(U, N)$ is in this case:
$$N_j\R_{0\,j\,\c\,\c}=\R_{0\,N\,\c\,\c}.$$
Thus, using the fact that $L=T+N$, $\Lb=T-N$ and the symmetries of $\R$, we deduce:
$$N_j\R_{0\,j\,\c\,\c}=\frac{1}{2}\R_{L\,\Lb\,\c\,\c}$$
which together with the bootstrap assumption for the curvature flux \eqref{bootcurvatureflux} improves the bilinear estimate \eqref{bil6}.


\subsubsection{Proof of \eqref{bil8imp}} We have $A=\curl(B)+E$ in view of Lemma \ref{recoverA}. Arguing as in \eqref{sodeska}, we reduce the proof to the estimate of: 
$$\norm{(\curl B)^j\pr_j\phi}_{L^2(\MM)}.$$
Since $B$ satisfies a wave equation in view of Lemma \ref{prop:waveeqB}, the quantity $\CC(U, N)$ is in this case:
$$\in_{jm\c}N_m\pr_j\phi=\nabb\phi.$$
Using again the decomposition \eqref{decompositionpr} for $\pr_j$ and the antisymmetry of $\in_{jm\c}$, we obtain schematically:
$$\in_{jm\c}N_m\pr_j\phi=\in_{jm\c}N_mN_j\pr_N\phi+\nabb\phi=\nabb\phi,$$
which improves the bilinear estimate \eqref{bil8}.

\subsection{Improvement of the trilinear estimate}

In this section, we shall derive the improved trilinear estimate \eqref{trilinearbootimp}. Let $Q_{\a\b\ga\de}$ the Bell-Robinson tensor of $\R$:
 \bea
 Q_{\a\b\ga\de}=\R_\a\,^\la\,\ga\,^\si \R_{\b\,\la\,\de\,\si}+\dual \R_\a\,^\la\,\ga\,^\si\dual \R_{\b\,\la\,\de\,\si}
 \eea
We need an trilinear estimate for the following quantity
\beaa
\left|\int_{\MM}Q_{ij\ga\de}k^{ij}e_0^\ga e_0^\de\right|.
\eeaa
We have $A=\curl(B)+E$ by Lemma \ref{recoverA}. Arguing as in \eqref{sodeska}, we reduce the proof to the estimate of: 
$$\left|\int_{\MM}Q_{\c\,j\ga\de}(\curl(B))_je_0^\ga e_0^\de\right|.$$
Making use of the wave equation \eqref{squareB.} for $B$
 we argue as in the beginning of section \ref{sec:proofbil1}  to  reduce the proof to an  estimate of the following:
$$\left|\int_{\MM}\int_{\SSS^2} \int_0^\infty e^{i \la \uom(t,x)}\bom^{-1}\left(\in_{jm\c}\Nom_m Q_{j\,\c\,\c\,\c}\right)f(\la\om)\la^3d\la d\om d\MM\right|$$
where $f$ satisfies:
$$\norm{\la^2f}_{L^2(\RRR^3)}\les M\ep.$$
Arguing as in \eqref{dorbia} \eqref{bilproof}, we obtain:  
\beaa
&&\left|\int_{\MM}\int_{\SSS^2} \int_0^\infty e^{i \la \uom(t,x)}\bom^{-1}\left(\in_{jm\c}\Nom_m Q_{j\,\c\,\c\,\c}\right)f(\la\om)\la^3d\la d\om d\MM\right|\\
\nn&\les & \int_{\SSS^2} \left\|  \bom^{-1}\left(\in_{jm\c}\Nom_m Q_{j\,\c\,\c\,\c}\right)\left(\int_0^{+\infty}e^{i \la \uom(t,x)}f(\la\om)\la^3d\la\right)\right\|_{L^1(\MM)} d\om\\
\nn&\les & \int_{\SSS^2} \norm{\bom^{-1}}_{L^\infty(\MM)}\norm{\in_{jm\c}\Nom_m Q_{j\,\c\,\c\,\c}}_{L^2_{\uom}L^1(\HH_{\uom})}\left\|\int_0^{+\infty}e^{i \la \uom(t,x)}f(\la\om)\la^3d\la\right\|_{L^2_{\uom}} d\om\\
\nn&\les & \left(\sup_{\om\in\SSS^2}\norm{\bom^{-1}}_{L^\infty(\MM)}\right)\left(\sup_{\om\in\SSS^2}\norm{\in_{jm\c}\Nom_m Q_{j\,\c\,\c\,\c}}_{L^2_{\uom}L^1(\HH_{\uom})}\right)\left(\int_{\SSS^2}\norm{\la^3f(\la\om)}_{L^2_\la}d\om\right)\\
\nn&\les & \sup_{\om\in\SSS^2}\norm{\in_{jm\c}\Nom_m Q_{j\,\c\,\c\,\c}}_{L^2_{\uom}L^1(\HH_{\uom})}M\ep,
\eeaa
where we used Plancherel in $\la$ and Cauchy Schwarz in $\om$. Thus, we finally obtain:
\bea\lab{kabul}
\left|\int_{\MM}Q_{ij\ga\de}k^{ij}e_0^\ga e_0^\de\right|\les\sup_{\om\in\SSS^2}\norm{\in_{jm\c}N_m Q_{j\,\c\,\c\,\c}}_{L^2_{\uom}L^1(\HH_{\uom})}M\ep+M^3 \ep^3.
\eea

Next, we estimate the right-hand side of \eqref{kabul}. Since all the estimates in the proof will be uniform in $\om$, we drop the index $\om$ to ease the notations.  The formula for the Bell-Robinson tensor $Q$ yields:
\beaa
Q_{j\c\c\c}&=&\R_j\,^\la\,\c\c \R_{\c\,\la\,\c\,\c}+dual\\
&=& -\frac{1}{2}\R_{j\,L\,\c\c} \R_{\c\,\Lb\,\c\,\c}-\frac{1}{2}\R_{j\,\Lb\,\c\c} \R_{\c\,L\,\c\,\c}+\R_{j\,A\,\c\c} \R_{\c\,A\,\c\,\c}+dual,
\eeaa
where we used the frame $L, \Lb, f_A, A=1, 2$ in the last equality. Thus, we have schematically:
$$\in_{jm\c}N_m Q_{j\,\c\,\c\,\c}=\R(\R\c L+\in_{jm\c}N_m\R_{j\,A\,\c\c})$$

 Decomposing $e_j$  with respect  to the   orthonormal  frame $N, f_B, B=1, 2$,  we note that:
$$\in_{jm\c}N_m \R_{jA\c\c}=\in_{jm\c}N_jN_m \R_{NA\c\c}+\in_{jm\c}(f_B)_jN_m \R_{BA\c\c}=\R_{BA\c\c}.$$
On the other hand, decomposing  $\R_{BA\c\c}$ further  and using the symmetries of $\R$, one easily checks that $\R_{BA\c\c}$ must contain at least one $L$ so that it is of the type $\R\c\L$.
  Thus, we have schematically:
  \bea
  \in_{jm\c}N_m Q_{j\,\c\,\c\,\c}=\R( \R\c L).
  \eea
Thus,  in view of   \eqref{kabul},   making use of  the bootstrap assumptions \eqref{bootR} on $R$ and \eqref{bootcurvatureflux} on the curvature flux, we deduce,
\beaa
\left|\int_{\MM}Q_{ij\ga\de}k^{ij}e_0^\ga e_0^\de\right|
&\les& (M\ep)^3+M\ep\norm{\R  \R_L}_{L^2_uL^1(\HH_u)}\\
&\les& (M\ep)^3+M\ep\norm{\R}_{L^2(\MM)}\norm{\R_L}_{\lu{2}}\\
&\les& M^3 \ep^3
\eeaa
In other words,
\bea\lab{belrob2}
\left|\int_{\MM}Q_{ij\ga\de}k^{ij}e_0^\ga e_0^\de\right|\les (M\ep)^3.
\eea
 which yields  the desired  improvement of the trilinear estimate \eqref{trilinearboot}.
 
\section{Proof of Proposition \ref{prop:improve2}, (part 2)}\lab{sec:improve2B}

In this section we prove the bilinear estimates  II. We start with a discussion of the sharp $L^4$ Strichartz estimate.

\subsection{The sharp $L^4(\MM)$ Strichartz estimate}

To a function $f$ on $\RRR^3$ and a family $\uom$ indexed by $\om\in\SSS^2$ of scalar functions on the space-time $\MM$ satisfying the eikonal equation for each $\om\in\SSS^2$, we associate a half-wave parametrix:
$$\int_{\SSS^2} \int_0^\infty  e^{i \la \uom(t,x)}f(\la\om)\la^2d\la d\om.$$
Let $p$ be an integer  and $\psi$ be a smooth cut-off function on supported on the interval $[1/2,2]$. We call a half-wave wave parametrix localized at frequencies of size $\la\sim 2^p$ the following Fourier integral operator:
$$\int_{\SSS^2} \int_0^\infty  e^{i \la \uom(t,x)}\psi(2^{-p}\la)f(\la\om)\la^2d\la d\om.$$
We have the following $L^4(\MM)$ Strichartz estimates localized in frequency for a half wave parametrix which are proved in \cite{bil2}:
\begin{proposition}[Corollary 2.8 in \cite{bil2}]\lab{prop:L4strichartz}
Let $f$ be a function on $\RRR^3$, let $p\in\mathbb{N}$, and let $\psi$ be as defined above. Let $\uom$ be a family of scalar functions on the space-time $\MM$ satisfying the eikonal equation for each $\om\in\SSS^2$ and initialized on the initial slice $\Si_0$ as in \cite{param1}. Define a scalar function $\phi_p$ on $\MM$ as 
the following oscillatory integral:
$$\phi_p(t,x)=\int_{\SSS^2} \int_0^\infty  e^{i \la \uom(t,x)}\psi(2^{-p}\la)f(\la\om)\la^2d\la d\om.$$
Then, we have the following $L^4(\MM)$ Strichartz estimates for $\phi_p$:
\bea
\norm{\phi_p}_{L^4(\MM)}&\les& 2^{\frac{p}{2}}\norm{\psi(2^{-p}\la)f}_{L^2(\RRR^3)},\\
\norm{\pr\phi_p}_{L^4(\MM)}&\les& 2^{\frac{3p}{2}}\norm{\psi(2^{-p}\la)f}_{L^2(\RRR^3)},\\
\norm{\pr^2\phi_p}_{L^4(\MM)}&\les& 2^{\frac{5p}{2}}\norm{\psi(2^{-p}\la)f}_{L^2(\RRR^3)}.
\eea
\end{proposition}

Note that these Strichartz estimates are  sharp. 

\subsection{Improvement of the non sharp Strichartz estimates}\lab{sec:improvnonsharpstrich}

In this section, we derive the improved non sharp Strichartz estimates \eqref{bootstrichimp} and  \eqref{bootstrichBimp}. In order to do this, we first estimate the $\lsitt{2}{7}$ norm of $\partial B$ using the  $L^4(\MM)$ Strichartz estimate together with Sobolev embeddings on $\Si_t$.\begin{corollary}\lab{cor:strichartzB}
$B$ satisfies the following Strichartz estimate:
$$\norm{\pr B}_{\lsitt{2}{7}}\lesssim M\ep.$$
\end{corollary}

\begin{proof}
Decompose $B$ as before,  with the help  of  Theorem  \ref{lemma:parametrixconstruction},

\bea\lab{kyoto}
\norm{\pr B}_{\lsitt{2}{7}}\leq \sum_{j=0}^{+\infty}\norm{\pr\phi^{(j)}}_{\lsitt{2}{7}}.
\eea
Thus is suffices to prove for all $j\geq 0$:
\bea\lab{kyoto1}
\norm{\pr\phi^{(j)}}_{\lsitt{2}{7}}\lesssim (M\ep)^{j+1}.
\eea
The estimates in \eqref{kyoto1} are analogous for all $j$, so it suffices to prove \eqref{kyoto1} in the case $j=0$. In view of the definition of $\phi^{(0)}$, the estimates for $B$ on the initial slice $\Si_0$ obtained in Lemma \ref{lemma:initialslice}, the estimate \eqref{spacetimeB} for $\pr \square  B$, and the definition of $\Psi_{om}$ and $\Psi(t,s)$, \eqref{kyoto1} reduces to the following estimate for a half wave parametrix:
\bea\lab{kyoto1bis}
\left\|   \pr \left( \int_{\SSS^2} \int_0^\infty  e^{i \la \uom(t,x)}  f(\la\om)\la^2d\la d\om\right)\right\|_{\lsitt{2}{7}}\les \norm{\la^2 f}_{L^2(\RRR^3)}.
\eea

Next, we introduce $\varphi$  and $\psi$ two smooth compactly supported functions on $\RRR^+$ such that $\psi$ is supported away from 0 and: 
\begin{equation}\label{kyoto2}
\varphi(\lambda)+\sum_{p\geq 0}\psi(2^{-p}\lambda)=1\textrm{ for all }\lambda\in\RRR.
\end{equation}
We define the family of scalar functions $\phi_p$ for $p\geq -1$ on $\MM$ as:
\bea\lab{kyoto3}
\phi_{-1}(t,x)=\int_{\SSS^2} \int_0^\infty  e^{i \la \uom(t,x)}\varphi(\la)f(\la\om)\la^2d\la d\om,
\eea
and for all $p\geq 0$:
\bea\lab{kyoto4}
\phi_p(t,x)=\int_{\SSS^2} \int_0^\infty  e^{i \la \uom(t,x)}\psi(2^{-p}\la)f(\la\om)\la^2d\la d\om.
\eea
In view of \eqref{kyoto2}, we have:
$$\pr\left(\int_{\SSS^2} \int_0^\infty  e^{i \la \uom(t,x)}f(\la\om)\la^2d\la d\om\right)=\sum_{p\geq -1}\pr \phi_p(t,x),$$
which yields:
\bea\lab{kyoto5}
\left\|  \pr\left(     \int_{\SSS^2} \int_0^\infty  e^{i \la \uom(t,x)}f(\la\om)\la^2d\la d\om\right)\right\|_{\lsitt{2}{7}}\les\sum_{p\geq -1}\norm{\pr \phi_p}_{\lsitt{2}{7}}.
\eea

The estimate for $\phi_{-1}$ is easier, so we focus on $\phi_p$ for $p\geq 0$. Using the Sobolev embedding \eqref{sobinftysit} on $\Si_t$, the $L^4(\MM)$ Strichartz localized in frequency of Proposition \ref{prop:L4strichartz}, and the fact that $\psi$ is supported in $(0,+\infty)$, we have,  
\beaa
\norm{\pr \phi_p}_{\lsitt{2}{7}}&\les &\norm{\pr \phi_p}^{\frac{4}{7}}_{L^4(\MM)}\norm{\pr \phi_p}^{\frac{3}{7}}_{\lsitt{4}{\infty}}\\
\nn&\les & \norm{\pr \phi_p}^{\frac{4}{7}}_{L^4(\MM)}\norm{\pr^2\phi_p}^{\frac{3}{7}}_{L^4(\MM)} +\norm{\pr\phi_p}_{L^4(\MM)}  \\
\nn&\les & \left(2^{\frac{3p}{2}}\norm{\psi(2^{-p}\la)f}_{L^2(\RRR^3)}\right)^{\frac{4}{7}}\left(2^{\frac{5p}{2}}\norm{\psi(2^{-p}\la)f}_{L^2(\RRR^3)}\right)^{\frac{3}{7}}\\
\nn&\les & 2^{-\frac{p}{14}}\norm{\la^2\psi(2^{-p}\la)f}_{L^2(\RRR^3)}\\
\nn&\les & 2^{-\frac{p}{14}}\norm{\la^2 f}_{L^2(\RRR^3)}.
\eeaa
Note that we have ignored  the lower order term $\norm{\pr\phi_p}_{L^4(\MM)}$   which    is easier  and   can be treated   in the same way.
Together with \eqref{kyoto5}, we obtain:
$$\left\|\pr\left(  \int_{\SSS^2} \int_0^\infty  e^{i \la \uom(t,x)}f(\la\om)\la^2d\la d\om\right)\right\|_{\lsitt{2}{7}}\les\left(\sum_{p\geq -1}2^{-\frac{p}{14}}\right)\norm{\la^2 f}_{L^2(\RRR^3)}\les \norm{\la^2 f}_{L^2(\RRR^3)},$$
which is \eqref{kyoto1bis}. This concludes the proof of the Corollary.
\end{proof}

Lemma \ref{recoverA} and Corollary \ref{cor:strichartzB} yield:
\bea\lab{saionara1}
\norm{A}_{\lsitt{2}{7}}\les \norm{\pr B}_{\lsitt{2}{7}}+\norm{E}_{\lsitt{2}{7}}\les M\ep+M^2\ep^2,
\eea
which is an improvement on the bootstrap assumption \eqref{bootstrich}.

\subsection{Improvement of the bilinear bootstrap assumptions II}\lab{sec:proofbil2} 
In this section, we derive the improved bilinear estimate \eqref{bil3imp} and \eqref{bil4imp} of Proposition \ref{prop:improve2}. Recall the decomposition $A=\curl(B)+E$ of Lemma \ref{recoverA}. Using the bootstrap assumption \ref{bootA} for $A$, the estimates for $E$ given by Lemma \ref{recoverA} and the Sobolev embedding on $\Si_t$ \eqref{sobineqsit}, we have:
\beaa
&&\norm{(-\Delta)^{-\frac{1}{2}}(\pr A\pr E)}_{L^2(\MM)}+\norm{(-\Delta)^{-\frac{1}{2}}(\pr E\pr E)}_{L^2(\MM)}\\
&\les& \norm{\pr A\pr E}_{\lsit{\frac{6}{5}}}+\norm{\pr E\pr E}_{\lsit{\frac{6}{5}}}\\
&\les& \norm{\pr A}_{\lsit{2}}\norm{\pr E}_{\lsit{3}}+\norm{\pr E}_{\lsit{3}}\norm{\pr E}_{\lsit{2}}\\
&\les& M^2\ep^2.
\eeaa
Together with the decomposition $A=\curl(B)+E$ of Lemma \ref{recoverA}, this implies that the proof of 
the bilinear estimates \eqref{bil3imp} and \eqref{bil4imp} reduces to:
$$\norm{(-\Delta)^{-\frac12}(Q_{ij}(\curl(B),\curl(B)))}_{L^2(\MM)}+\norm{(-\Delta)^{-\frac12}(\pr(\curl(B)^l)\pr_l(\curl(B)))}_{L^2(\MM)}\lesssim M^2\ep^2,$$
where the bilinear form $Q_{ij}$ is given by $Q_{ij}(\phi,\psi)=\pr_i\phi\pr_j\psi-\pr_j\phi\pr_i\psi$. Also, note that:
$$\pr(\curl(B)^l)\pr_l(\curl(B))=Q_{ij}(\pr B,\pr B)+A\,\pr B\, \pr^2 B.$$
Together with the Sobolev embedding \eqref{ikea3} on $\Si_t$, the bootstrap assumptions \eqref{bootA} on $A$, and the estimates oProposition \ref{lemma:estB} for $B$, we obtain:
\beaa
&&\norm{(-\Delta)^{-\frac12}(\pr(\curl(B)^l)\pr_l(\curl(B)))}_{L^2(\MM)}\\
&\les& \norm{(-\Delta)^{-\frac12}(Q_{ij}(\pr B,\pr B))}_{L^2(\MM)}+\norm{(-\Delta)^{-\frac12}(A\,\pr B\, \pr^2 B)}_{L^2(\MM)}\\
&\les& \norm{(-\Delta)^{-\frac12}(Q_{ij}(\pr B,\pr B))}_{L^2(\MM)}+\norm{A\,\pr B\, \pr^2 B}_{\lsit{\frac{6}{5}}}\\
&\les& \norm{(-\Delta)^{-\frac12}(Q_{ij}(\pr B,\pr B))}_{L^2(\MM)}+\norm{A}_{\lsit{6}}\norm{\pr B}_{\lsit{6}}\norm{\pr^2 B}_{\lsit{2}}\\
&\les& \norm{(-\Delta)^{-\frac12}(Q_{ij}(\pr B,\pr B))}_{L^2(\MM)}+M^3\ep^2.
\eeaa
Finally, the proof of the bilinear estimates \eqref{bil3} and \eqref{bil4} reduces to:
\bea\lab{domo}
\norm{(-\Delta)^{-\frac12}(Q_{ij}(\pr B,\pr B))}_{L^2(\MM)}\lesssim M^2\ep^3.
\eea

Next, we focus on proving \eqref{domo}.  Decomposing $B$ according to   Theorem  \ref{lemma:parametrixconstruction}, we have:
\bea\lab{domo0}
\norm{(-\Delta)^{-\frac12}(Q_{ij}(\pr B,\pr B))}_{L^2(\MM)}\leq \sum_{m, n=0}^{+\infty}\norm{(-\Delta)^{-\frac12}(Q_{ij}(\pr \phi^{(m)},\pr \phi^{(n)}))}_{L^2(\MM)}.
\eea
Thus it suffices to prove for all $m, n\geq 0$:
\bea\lab{domo1}
\norm{(-\Delta)^{-\frac12}(Q_{ij}(\pr \phi^{(m)}, \pr \phi^{(n)}))}_{L^2(\MM)}\lesssim (M\ep)^{m+1}(M\ep)^{n+1}.
\eea
The estimates in \eqref{domo1} are analogous for all $m, n$, so it suffices to prove \eqref{domo1} in the case $(m,n)=(0,0)$. In view of the definition of $\phi^{(0)}$, the estimates for $B$ on the initial slice $\Si_0$ obtained in Lemma \ref{lemma:initialslice},    estimate \eqref{spacetimeB} for $\pr \square B$, \eqref{domo1} reduces to the following bilinear estimate for half-wave parametrices:

\bea\lab{domo1bis}
\left\|(-\Delta)^{-\frac12}Q_{ij}(\pr\phi^{(1)}, \pr\phi^{(2)})\right\|_{L^2(\MM)}    \les  \norm{\la^2 f_1}_{L^2(\RRR^3)}\norm{\la^2 f_2}_{L^2(\RRR^3)}.
\eea
with,
\bea
\phi^{(k)}=\int_{\SSS^2} \int_0^\infty  e^{i \la \uom(t,x)}f_k(\la\om)\la^2d\la d\om, \qquad k=1,2.
\eea
Recall the smooth cut off functions $\varphi$  and $\psi$ introduced in the proof of Corollary \ref{cor:strichartzB}. We define two families of scalar functions $\phi^j_p, j=1, 2$ for $p\geq -1$ on $\MM$ as:
\bea\lab{domo3}
\phi^{(j)}_{-1}(t,x)=\int_{\SSS^2} \int_0^\infty  e^{i \la \uom(t,x)}\varphi(\la)f_j(\la\om)\la^2d\la d\om,
\eea
and for all $p\geq 0$:
\bea\lab{domo4}
\phi^{(j)}_p(t,x)=\int_{\SSS^2} \int_0^\infty  e^{i \la \uom(t,x)}\psi(2^{-p}\la)f_j(\la\om)\la^2d\la d\om.
\eea
In view of \eqref{kyoto2}, we have:
$$\pr\phi^{(k)} =\pr\left( \int_{\SSS^2} \int_0^\infty  e^{i \la \uom(t,x)}f_k(\la\om)\la^2d\la d\om\right)=\sum_{p\geq -1} \pr \phi^{(k)}_p(t,x),$$
which yields:
\bea\lab{domo5}
\left\|(-\Delta)^{-\frac12}Q_{ij}(\pr\phi^{(1)}, \pr\phi^{(2)})\right\|_{L^2}     &\les & \sum_{p, q\geq -1}\norm{(-\Delta)^{-\frac12}\big(Q_{ij}\pr \phi^{(1)}_p, \pr\phi^{(2)}_q)\big)}_{L^2}.
\eea

The estimates involving $\phi^{(k)}_{-1}$ are easier, so we focus on $\phi^{(1)}_p, \phi^{(2)}_q$ for $p, q\geq 0$. 
We may assume $q\geq p$. Note that the structure of $Q_{ij}$ implies:
$$Q_{ij}(\pr \phi^1_p,\pr  \phi^2_q)=\pr(\pr^2\phi^{(1)}_p\c \pr\phi_q^{(2)})+A\c \pr^2  \phi^{(1)}_p \c   \pr \phi^{(2)}_q$$
which yields:
\beaa
&&\norm{(-\Delta)^{-\frac12}(Q_{ij}(\pr \phi^{(1)}_p, \pr \phi^{(2)}_q))}_{L^2(\MM)}\\
&\les& \norm{(-\Delta)^{-\frac12}\pr\big(\pr^2\phi^{(1)}_p\c \pr \phi^{(2)}_q\big)}_{L^2(\MM)}+\norm{(-\Delta)^{-\frac12}\big (A\c \pr^2\phi^{(1)}_p\c\pr \phi^{(2)}_q\big)}_{L^2(\MM)}\\
&\les& \norm{\pr^2\phi^{(1)}_p\c  \pr \phi^{(2)}_q}_{L^2(\MM)}+\norm{A\c\pr^2\phi^1_p\c \pr \phi^2_q}_{\lsitt{2}{\frac{6}{5}}}\\
&\les& \norm{\pr^2\phi^{(1)}_p}_{L^4(\MM)}\norm{\pr \phi^{(2)}_q}_{L^4(\MM)}+\norm{A}_{\lsit{3}}\norm{\pr^2\phi^{(1)}_p}_{L^4(\MM)}\norm{\pr \phi^{(2)}_q}_{L^4(\MM)}\\
&\les & \norm{\pr^2\phi^{(1)}_p}_{L^4(\MM)}\norm{\pr\phi^{(2)}_q}_{L^4(\MM)}
\eeaa
where we used the bootstrap assumption \eqref{bootA} for $A$ in the last inequality. Together with the $L^4(\MM)$ frequency localized  Strichartz estimate of Proposition \ref{prop:L4strichartz}, and the fact that $\psi$ is supported in $(0,+\infty)$, we obtain:
\beaa
\norm{(-\Delta)^{-\frac12}(Q_{ij}(\pr \phi^{(1)}_p,\pr  \phi^{(2)}_q))}_{L^2(\MM)}&\les & 2^{\frac{5p}{2}+\frac{3q}{2}}\norm{\psi(2^{-p}\la)f_1}_{L^2(\RRR^3)}\norm{\psi(2^{-q}\la)f_2}_{L^2(\RRR^3)}\\
\nn&\les & 2^{\frac{p}{2}-\frac{q}{2}}\norm{\la^2\psi(2^{-p}\la)f_1}_{L^2(\RRR^3)}\norm{\la^2\psi(2^{-q}\la)f_2}_{L^2(\RRR^3)}.
\eeaa
Since we assume $q\geq p$, this yields:
\bea\lab{domo6}
&&\sum_{p, q\geq -1}\norm{(-\Delta)^{-\frac12}(Q_{ij}(\pr \phi^{(1)}_p, \pr \phi^{(2)}_q))}_{L^2(\MM)}\\
\nn&\les& \sum_{p, q\geq -1}2^{-\frac{|p-q|}{2}}\norm{\la^2\psi(2^{-p}\la)f_1}_{L^2(\RRR^3)}\norm{\la^2\psi(2^{-q}\la)f_2}_{L^2(\RRR^3)}\\
\nn&\les& \left(\sum_{p\geq -1}\norm{\la^2\psi(2^{-p}\la)f_1}^2_{L^2(\RRR^3)}\right)^{\frac{1}{2}}\left(\sum_{q\geq -1}\norm{\la^2\psi(2^{-q}\la)f_2}^2_{L^2(\RRR^3)}\right)^{\frac{1}{2}}\\
\nn&\les &  \norm{\la^2 f_1}_{L^2(\RRR^3)}\norm{\la^2 f_2}_{L^2(\RRR^3)}.
\eea
Finally, \eqref{domo5} and \eqref{domo6} imply \eqref{domo1bis}. This concludes the proof of the improved bilinear  estimates \eqref{bil3imp} and \eqref{bil4imp}.

Finally, \eqref{saionara1}, the results in section \ref{sec:proofbil1} and section \ref{sec:proofbil2}, and \eqref{belrob2} yield the improved estimates \eqref{bil1imp}, \eqref{bil5imp}, \eqref{bil6imp}, \eqref{bil8imp}, \eqref{bil8imp}, \eqref{bil3imp}, \eqref{bil4imp}, \eqref{bootstrichimp}, and \eqref{trilinearbootimp}. This concludes the proof of Proposition \ref{prop:improve2}.

\section{Propagation of regularity}\label{sect:propag}

The goal of this section is to prove Proposition \ref{prop:propagreg}. Recall from the statement of Proposition \ref{prop:propagreg} that we assume that the estimates corresponding to all bootstrap assumptions of the section \ref{sec:mainbootassss} hold for $0\leq t\leq T^*$ with a universal constant $M$. For the convenience of the reader, we recall some of these estimates below
\bea\label{pr:bootR}
\norm{\R}_{\lsit{2}}+\norm{\R\c L}_{L^2(\HH)}\les \ep,
\eea
\bea\label{pr:bootA}
\norm{A}_{\lsit{2}}+\norm{\prb A}_{\lsit{2}}\les \ep,
\eea
and:
\bea\label{pr:bootA0}
\nn\norm{A_0}_{\lsit{4}}+\norm{\prb A_0}_{\lsit{2}}+\norm{A_0}_{\lsitt{2}{\infty}}+\norm{\pr A_0}_{\lsit{3}}&&\\
+\norm{\pr\pr A_0}_{\lsit{\frac{3}{2}}}+\norm{\pr_0A_0}_{\lsitt{2}{\frac{42}{13}}}+\norm{\pr\pr_0A_0}_{\lsitt{2}{\frac{14}{9}}} &\les& \ep.
\eea
Note also that we have \eqref{bootn} with a universal constant $M$, i.e.
\bea\lab{pr:bootn}
\norm{n-1}_{L^\infty(\MM)}+\norm{\nabla n}_{L^\infty(\MM)}\les \ep.
\eea

Finally,  for any weakly regular  null hypersurface $\HH$ and  any smooth  scalar function $\phi$ on $\MM$,
\bea\label{pr:bil8}
\norm{A^j\pr_j\phi}_{L^2(\MM)}\les \ep\sup_{\HH}\norm{\nabb\phi}_{L^2(\HH)},
\eea
where the supremum is taken over all null hypersurfaces $\HH$. \\


We now turn to the proof of Proposition \ref{prop:propagreg}. The estimates will be similar to the one derived in Proposition \ref{prop:improve1} and \ref{prop:improve2}. Thus, we will simply sketch the arguments without going into details. Note that the crucial  point is to check that the null structure - which is at the core of the proof of Proposition \ref{prop:improve1} and \ref{prop:improve2} - is preserved after differentiation. This will be done in section \ref{sec:preservenullstruct}. 

\subsection{Elliptic estimates}

Recall that we need to control $\norm{\D\R}_{L^\infty_tL^2(\Sigma_t)}$. Note that  we may restrict our attention to the control of $\D_l\R_{ij\a\b}$ and $\D_0\R_{ij\a\b}$. Indeed, all the other components are recovered from the Bianchi identities and the standard symmetries of $\R$. Furthermore, since the estimates are for $\D_l\R_{ij\a\b}$ and $\D_0\R_{ij\a\b}$ similar, we will restrict only to the control of $\D_l\R_{ij\a\b}$. These components, according to the Cartan formalism, can be expressed in the form
\begin{equation}\label{eq:RA}
\R_{ij\a\b} = \left (\pa_i A_j-\pa_j A_i+[A_i,A_j]\right)_{\a\b}. 
\end{equation}
Since
$$
\D_l\R = \pa_l \R+ \A \R
$$
we can estimate 
$$
\|\D_l\R_{ij\a\b}\|_{L^2(\Sigma_t)} \les \|\pa_l\R_{ij\a\b}\|_{L^2(\Sigma_t)} + \|\R\|_{L^6(\Sigma_t)} \|\A\|_{L^3(\Sigma_t)}. 
$$
Furthermore, from \eqref{eq:RA}
$$
\|\pa \R_{ij\a\b}\|_{L^2(\Sigma_t)} \les \|\pa^2 A\|_{L^2(\Sigma_t)} +   \|\pa A\|_{L^6(\Sigma_t)} \|\A\|_{L^3(\Sigma_t)} \les 
\|\pa^2 A\|_{L^2(\Sigma_t)}. 
$$
Proceeding the same way with all other components of $\D\R$, we derive
\bea\lab{eq:da2dr}
\|\pa \R\|_{L^2(\Sigma_t)}  \les \|\pa\prb \A\|_{L^2(\Sigma_t)}+\ep^2. 
\eea

In view of \eqref{eq:da2dr}, it suffices to derive $L^2$ bounds for $\pa\prb\A$. Furthermore, since $A_0$ satisfies an elliptic equation, and hence better estimates, we focus on the estimates for $\pr\prb A$. We sketch below the estimates just for $\pr^2 A$ since the estimates $\pr\pr_0A$ are similar. To establish bounds on $\|\pa^2 A\|_{L^2(\Sigma_t)}$ we use the following analog of Lemma \ref{recoverA}:
\begin{lemma}\lab{recoverA'}
The following decomposition
$$\pa A=\curl(\pa B)+E'$$
holds with $E'$ satisfying:
$$\norm{\pr E'}_{\lsit{3}}+\norm{\pr^2E'}_{\lsit{\frac{3}{2}}}+\norm{E'}_{\lsitt{2}{\infty}}\les \eps \left (\|\pa^3 B\|_{L^\infty_t L^2(\Sigma_t)}+
\|\pa^3 B\|_{L^2_t L^7(\Sigma_t)}\right)+\ep^2.$$
Furthermore, $A$ satisfies
$$\|\pa^2 A\|_{L^\infty_t L^2(\Sigma_t)}\les \|\pa^3 B\|_{L^\infty_t L^2(\Sigma_t)}+\ep^2.$$
\end{lemma}

\begin{proof}
Recall from Lemma \ref{recoverA} that we have the following decomposition
$$
A=\curl B +E, 
$$
with 
\bea\lab{eq:eel1}
E=-(-\Delta)^{-1}(\R\pr B+\pr A\pr B+A\pr^2B+A^2\pr B)+(-\Delta)^{-1} (A \pr  A+A^3)
\eea
given in \eqref{eel1}. We now introduce the new variables $\pa A$ and $\pa B$, linked by the equation
$$
\pa A=\curl (\pa B) +E', 
$$
where 
$$
E'=\pa E+[\pa,\curl] B,
$$
so that 
$$
E'=-\pa (-\Delta)^{-1}\left [  (    \R\pr B+\pr A\pr B+A\pr^2B+A^2\pr B)+A \pr  A+A^3\right ]+A\pa B.
$$
It then follows that 
\bea\lab{eq:what}
\|\pa^2 A\|_{L^\infty_t L^2(\Sigma_t)}\les \|\pa^3 B\|_{L^\infty_t L^2(\Sigma_t)}+\|\pa E'\|_{L^\infty_t L^2(\Sigma_t)}
\eea
and
\begin{align*}
\|\pa E'\|_{L^\infty_t L^2(\Sigma_t)} &\les \left(\|R\|_{L^\infty_t L^3(\Sigma_t)} + \|\pa A\|_{L^\infty_t L^3(\Sigma_t)}+ \|A\|^2_{L^\infty_t L^{6}(\Sigma_t)}\right) 
\|\pa B\|_{L^\infty_t L^6(\Sigma_t)} \\ &+ \|A\|_{L^\infty_t L^6(\Sigma_t)} \|\pa^2 B\|_{L^\infty_t L^3(\Sigma_t)}\\ &+ 
\|A\|_{L^\infty_t L^6(\Sigma_t)} \|\pa A\|_{L^\infty_t L^3(\Sigma_t)} + \|A\|^3_{L^\infty_t L^6(\Sigma_t)}\\ &\les  \epsilon\left (\|\pa R\|_{L^\infty_t L^2(\Sigma_t)}+
\|\pa^2 A\|_{L^\infty_t L^2(\Sigma_t)}+\|\pa^3 B\|_{L^\infty_t L^2(\Sigma_t)}\right)+\ep^2\\ &\les   \epsilon\left 
(\|\pa E'\|_{L^\infty_t L^2(\Sigma_t)}+\|\pa^3 B\|_{L^\infty_t L^2(\Sigma_t)}\right)+\ep^2,
\end{align*}
where in the second inequality we interpolated between 
$\|\pa A\|_{L^\infty_t L^2(\Sigma_t)}$, $\|R\|_{L^\infty_t L^2(\Sigma_t)}$, $\|\pa^2 B\|_{L^\infty_t L^2(\Sigma_t)}$ and 
$\|\pa^2 A\|_{L^\infty_t L^2(\Sigma_t)}$, $\|\pa R\|_{L^\infty_t L^2(\Sigma_t)}$, $\|\pa^3 B\|_{L^\infty_t L^2(\Sigma_t)}$, using 
the already established bounds on the former, and in the last inequality we used the bounds \eqref{eq:da2dr} and \eqref{eq:what}.
From this we obtain that 
$$\|\pa^2 A\|_{L^\infty_t L^2(\Sigma_t)}\les \|\pa^3 B\|_{L^\infty_t L^2(\Sigma_t)}+\ep^2.$$
Similarly,
\begin{align*}
\|\pa E'\|_{L^\infty_t L^3(\Sigma_t)} &\les \left(\|R\|_{L^\infty_t L^6(\Sigma_t)} + \|\pa A\|_{L^\infty_t L^6(\Sigma_t)}+ \|A\|^2_{L^\infty_t L^{12}(\Sigma_t)}\right) 
\|\pa B\|_{L^\infty_t L^6(\Sigma_t)} \\ &+ \|A\|_{L^\infty_t L^6(\Sigma_t)} \|\pa^2 B\|_{L^\infty_t L^6(\Sigma_t)}\\ &+ 
\|A\|_{L^\infty_t L^6(\Sigma_t)} \|\pa A\|_{L^\infty_t L^6(\Sigma_t)} + \|A\|^3_{L^\infty_t L^9(\Sigma_t)}\\ &\les  \epsilon\left (\|\pa R\|_{L^\infty_t L^2(\Sigma_t)}+
\|\pa^2 A\|_{L^\infty_t L^2(\Sigma_t)}+\|\pa^3 B\|_{L^\infty_t L^2(\Sigma_t)}\right)+\ep^2\\ &\les   \epsilon\left 
(\|\pa E'\|_{L^\infty_t L^2(\Sigma_t)}+\|\pa^3 B\|_{L^\infty_t L^2(\Sigma_t)}\right)+\ep^2
\end{align*}
and 
$$\|\pa E'\|_{L^\infty_t L^3(\Sigma_t)} \les   \epsilon\|\pa^3 B\|_{L^\infty_t L^2(\Sigma_t)}+\ep^2.$$
The other estimates are proved in the same fashion. This concludes the proof of the lemma.
\end{proof}

In view of Lemma \ref{recoverA'}, it remains to estimate $\pr^3B$. This will be done following the same circle of ideas used for $\pr^2B$. 
 
\subsection{The wave equation for $\pa B$}

Recall from \eqref{eq:fondwaveeqB} that we have, schematically: 
$$
\Box B=(-\Delta)^{-1} [\Box,\Delta]B + (-\Delta)^{-1} \Box (\curl A). 
$$
Using the commutation formula \eqref{commsquare2} we obtain
\begin{equation}\label{eq:eqB}
\Box \pr B=F
\end{equation}
where $F$ is given by
\begin{equation}\label{eq:eqBbis}
F=A^\ell \pa_\ell\prb B + A^0\prb^2 B+\A^2 \prb B +\pa_0 A^0 \prb B+ \pa (-\Delta)^{-1} [\Box,\Delta]B + \pa (-\Delta)^{-1} \Box (\curl A).
\end{equation}

Using the energy estimates for the wave equation derived in Lemma \ref{lemma:energyestimatebis} and the proof of the non sharp Strichartz estimate in section \ref{sec:improvnonsharpstrich}, we obtain the following proposition.
\begin{proposition}\label{prop:prB}
We have the following estimates
\begin{itemize}
\item Energy estimate
\begin{equation}\label{eq:enB}
\mathcal E:=\sup_{\mathcal H}(\|\nabb \pa^2 B\|_{L^2({\mathcal H})}+\|L \pa^2 B\|_{L^2({\mathcal H})}) + \|\prb \pa^2 B\|_{L^\infty_t L^2(\Sigma_t)} \les \|\prb \pa^2 B\|_{L^2(\Sigma_0)}+\|\pr F\|_{L^2(\mathcal M)}+\ep.
\end{equation}
\item Non sharp Strichartz estimate 
\begin{equation}\label{eq:strB}
\mathcal S:=\|\pa^2 B\|_{L^2_t L^7(\Sigma_t)} \les \|\prb \pa^2 B\|_{L^2(\Sigma_0)}+\|\pr F\|_{L^2(\mathcal M)}+\ep.
\end{equation}
\end{itemize}
Furthermore,  we have the following estimate for $\pr\pr_0\pr_0 B$
\bea\lab{prpr0pr0B}
\norm{\pr\pr_0\pr_0 B}_{L^2(\MM)}\les  \|\prb \pa^2 B\|_{L^2(\Sigma_0)}+\|\pr F\|_{L^2(\mathcal M)}+\ep.
\eea
\end{proposition}

In view of Proposition \ref{prop:prB}, we need to estimate $\pr F$. This is done in the following proposition:
\begin{proposition}\label{prop:F}
$F$ satisfies the following estimate
$$
\|\pr F\|_{L^2(\mathcal M)}\les \ep\left(\|\pr^3 B_0\|_{L^2(\Sigma_0)}+\|\pr^2 B_1\|_{L^2(\Sigma_0)} +
\mathcal E+\mathcal S+ \|\pr F\|_{L^2(\mathcal M)}\right)+\ep^2.
$$
\end{proposition} 
The proof of Proposition \ref{prop:F} is postponed to the next section. In view of Proposition \ref{prop:prB} and Proposition \ref{prop:F}, we obtain the control for $\|\prb \pa^2 B\|_{L^\infty_t L^2(\Sigma_t)}$ by its initial data which together with the elliptic estimates of the previous section yields the desired control for $\norm{\D\R}_{L^\infty_t L^2(\Sigma_t)}$. This concludes the proof of Proposition \ref{prop:propagreg}. \\

\subsection{Proof of Proposition \ref{prop:F}}\lab{sec:preservenullstruct}

From \eqref{eq:eqB}, we have
\beaa
\pa F &= &\pa A^\ell \pa_\ell\prb B+A^\ell \pa \pa_\ell\prb B + \pa(A^0\prb^2 B+\A^2 \prb B +\pa_0 A^0 \prb B)\\
&&+ \pa^2 
 (-\Delta)^{-1} [\Box,\Delta]B + \pa^2 (-\Delta)^{-1} \Box (\curl A).
\eeaa
This yields
\bea\lab{oooooo}
\norm{\pa F}_{L^2(\MM)} &\les &\norm{\pa A^\ell \pa_\ell\prb B}_{L^2(\MM)}+\norm{A^\ell \pa \pa_\ell\prb B}_{L^2(\MM)} + \norm{[\Box,\Delta]B}_{L^2(\MM)}\\
\nn&&+\norm{\Box (\curl A)}_{L^2(\MM)}+ \norm{\pa(A^0\prb^2 B+\A^2 \prb B +\pa_0 A^0 \prb B)}_{L^2(\MM)}\\
\nn&\les &\norm{\pa A^\ell \pa_\ell\prb B}_{L^2(\MM)}+\norm{A^\ell \pa \pa_\ell\prb B}_{L^2(\MM)} + \norm{[\Box,\Delta]B}_{L^2(\MM)}\\
\nn&&+\norm{\Box (\curl A)}_{L^2(\MM)}+ l.o.t.,
\eea
where we neglect the cubic terms and the terms involving $A^0$ since, as in the proof of Proposition \ref{prop:improve1} and \ref{prop:improve2}, they are significantly easier to treat.

Next, we isolate the terms $\Box (\curl A)$ and $[\Box,\Delta]B$ on the right-hand side of \eqref{oooooo}. 
From the proof of Proposition \ref{prop:commutsquarecurl}, we have
$$
\Box (\curl A)= \pa (A^\ell \pa_\ell A) + Q_{ij} (A^\ell,A_\ell) + {\text{terms involving}}\,\, A_0 +{\text{cubic terms}}.
$$
We deduce
\begin{equation}\lab{oooooo1}
\norm{\Box (\curl A)}_{L^2(\MM)}\les \norm{A^\ell \pa_\ell\pr A}_{L^2(\MM)}+\norm{\pa A^\ell \pa_\ell A}_{L^2(\MM)} + \norm{Q_{ij} (A^\ell,A_\ell)}_{L^2(\MM)} + l.o.t.
\end{equation}

Next, we deal with $[\square, \Delta]B$. According to \eqref{csd2}, we have schematically
\begin{align*}
[\square, \Delta]B
&=k^{ab}\nabla_a\nabla_b(\pr_0 B)+n^{-1}\nabla_bn\nabla_b(\pr_0(\pr_0 B))+\nabla_0k^{ab}\nabla_a\nabla_b B\\ &
+ {\text{terms involving}}\,\, A_0 +{\text{cubic terms}}.
\end{align*}
We deduce 
\beaa
\norm{[\square, \Delta]B}_{L^2(\MM)}&\les& \norm{k^{ab}\nabla_a\nabla_b(\pr_0 B)}_{L^2(\MM)}+\norm{\nabla_0k^{ab}\nabla_a\nabla_b B}_{L^2(\MM)}\\
&&+\norm{n^{-1}\nabla_bn\nabla_b(\pr_0(\pr_0 B))}_{L^2(\MM)}+l.o.t.
\eeaa
Together with \eqref{oooooo} and \eqref{oooooo1}, we obtain
\bea\lab{oooooo2}
\norm{\pa F}_{L^2(\MM)} &\les &\norm{\pa A^\ell \pa_\ell\prb B}_{L^2(\MM)}+\norm{A^\ell \pa \pa_\ell\prb B}_{L^2(\MM)}+\norm{\pr A^\ell \pa_\ell  A}_{L^2(\MM)} \\
\nn&&+\norm{A^\ell \pa_\ell\pa  A)}_{L^2(\MM)} + \norm{Q_{ij} (A^\ell,A_\ell)}_{L^2(\MM)}+\norm{k^{ab}\nabla_a\nabla_b(\pr_0 B)}_{L^2(\MM)}\\
\nn&&+\norm{\nabla_0k^{ab}\nabla_a\nabla_b B}_{L^2(\MM)}+\norm{n^{-1}\nabla_bn\nabla_b(\pr_0(\pr_0 B))}_{L^2(\MM)}+l.o.t.
\eea

We will use the following bilinear estimates.
\begin{lemma}\lab{pr:lemmabil1}
We have
\beaa
&&\norm{A^\ell \pa \pa_\ell\prb B}_{L^2(\MM)}+\norm{A^\ell \pa_\ell\pa  A}_{L^2(\MM)}+\norm{k^{ab}\nabla_a\nabla_b(\pr_0 B)}_{L^2(\MM)}\\
&\les& \ep(\|\prb \pa^2 B\|_{L^2(\Sigma_0)}+\EE+\SS)+\ep^2.
\eeaa
\end{lemma}

\begin{lemma}\lab{pr:lemmabil2}
We have
\beaa
&&\norm{\pa A^\ell \pa_\ell\prb B}_{L^2(\MM)}+\norm{\pr A^\ell \pa_\ell  A)}_{L^2(\MM)}+\norm{Q_{ij} (A^\ell,A_\ell)}_{L^2(\MM)}+\norm{\nabla_0k^{ab}\nabla_a\nabla_b B}_{L^2(\MM)}\\
&\les& \ep(\|\prb \pa^2 B\|_{L^2(\Sigma_0)}+\EE+\SS+\|\pr F\|_{L^2(\mathcal M)})+\ep^2.
\eeaa
\end{lemma}

The proof of Lemma \ref{pr:lemmabil1} is postponed to section \ref{sec:pr:lemmabil1} and the proof of Lemma \ref{pr:lemmabil2} is postponed to section \ref{sec:pr:lemmabil2}. We now conclude the proof of Proposition \ref{prop:F}. In view of \eqref{oooooo2}, Lemma \ref{pr:lemmabil1} and Lemma \ref{pr:lemmabil2}, we have
\beaa
\norm{\pa F}_{L^2(\MM)} \les \norm{n^{-1}\nabla_bn\nabla_b(\pr_0(\pr_0 B))}_{L^2(\MM)}+\ep(\|\prb \pa^2 B\|_{L^2(\Sigma_0)}+\|\pr F\|_{L^2(\mathcal M)})+\ep^2.
\eeaa
Using the estimates \eqref{pr:bootn} and \eqref{prpr0pr0B}, we have
\beaa
\norm{n^{-1}\nabla_bn\nabla_b(\pr_0(\pr_0 B))}_{L^2(\MM)}&\les& \norm{n^{-1}\nabla n}_{L^\infty(\MM)}\norm{\pr\pr_0\pr_0 B}_{L^2(\MM)}\\
&\les& \ep(\|\prb \pa^2 B\|_{L^2(\Sigma_0)}+\|\pr F\|_{L^2(\mathcal M)})+\ep^2
\eeaa
and we deduce
\beaa
\norm{\pa F}_{L^2(\MM)} \les \ep(\|\prb \pa^2 B\|_{L^2(\Sigma_0)}+\EE+\SS+\|\pr F\|_{L^2(\mathcal M)})+\ep^2
\eeaa
which is the desired estimate. This concludes the proof of Proposition \ref{prop:F}.

\subsubsection{Proof of Lemma \ref{pr:lemmabil1}}\lab{sec:pr:lemmabil1}

In view of the identity $\pa A=\curl(\pa B)+E'$ of Lemma \ref{recoverA'}, we have
\beaa
&&\norm{A^\ell \pa \pa_\ell\prb B}_{L^2(\MM)}+\norm{A^\ell \pa_\ell\pa  A}_{L^2(\MM)}+\norm{k^{ab}\nabla_a\nabla_b(\pr_0 B)}_{L^2(\MM)}\\
&\les& \norm{A^\ell \pa \pa_\ell\prb B}_{L^2(\MM)}+\norm{k^{ab}\nabla_a\nabla_b(\pr_0 B)}_{L^2(\MM)}+l.o.t.
\eeaa
Together with the bilinear estimates \eqref{pr:bil8} and \eqref{pr:bil8}, we deduce
\beaa
&&\norm{A^\ell \pa \pa_\ell\prb B}_{L^2(\MM)}+\norm{A^\ell \pa_\ell\pa  A}_{L^2(\MM)}+\norm{k^{ab}\nabla_a\nabla_b(\pr_0 B)}_{L^2(\MM)}\\
&\les& \eps \sup_{\mathcal H}\|\nabb\pa \prb B\|_{L^2(\mathcal H)}+l.o.t.
\eeaa
Arguing as in \eqref{eq:usefullater}, we finally obtain
\beaa
&&\norm{A^\ell \pa \pa_\ell\prb B}_{L^2(\MM)}+\norm{A^\ell \pa_\ell\pa  A}_{L^2(\MM)}+\norm{k^{ab}\nabla_a\nabla_b(\pr_0 B)}_{L^2(\MM)}\\
&\les& \eps \sup_{\mathcal H}(\|\nabb\pa^2B\|_{L^2(\mathcal H)}+\|L\pa^2B\|_{L^2(\mathcal H)})+l.o.t.\\
&\les & \ep(\|\prb \pa^2 B\|_{L^2(\Sigma_0)}+\EE+\SS)+\ep^2
\eeaa
which is the desired estimate. This concludes the proof of Lemma \ref{pr:lemmabil1}.

\subsubsection{Proof of Lemma \ref{pr:lemmabil2}}\lab{sec:pr:lemmabil2}

In view of the wave equation \eqref{eq:eqB} satisfied by $\pa B$, we may use for $\pr B$ the parametrix constructed in Lemma \ref{lemma:parametrixconstruction}:
$$\pr B=\sum_{j=0}^{+\infty}\phi^{(j)},$$
with:
$$\phi^{(0)}=\Psi_{om}[\phi_0,\phi_1]+\int_0^t\Psi(t,s)F(s,.)ds,$$
and for all $j\geq 1$:
$$\phi^{(j)}=\int_0^t\Psi(t,s)F^{(j)}(s,.)ds.$$
Furthermore  $\phi^{(j)}$ and $F^{(j)}$ satisfy the following estimate:
$$\norm{\pr\prb\phi^{(j)}}_{\lsit{2}}+\norm{\pr F^{(j)}}_{L^2(\MM)}\les \ep^j(\norm{\pr^3B_0}_{L^2(\Si_0)}+\norm{\pr^2B_1}_{L^2(\Si_0)}+\norm{\pr F}_{L^2(\MM)}).$$

We will show that the proof of the bilinear estimates of Lemma \ref{pr:lemmabil2} all  involve the $L^2(\MM)$ norm of quantities of the type:
$$\CC(U,\pr (\pr B)),$$
where $\CC(U,\pr(\pr B))$ denotes a contraction with respect to one index between a tensor $U$ and $\pr (\pr B)$. Using the parametrix for $\pr B$ discussed above, and arguing as in section \ref{sec:proofbil1}, we obtain the analog of  \eqref{bilproof}:  
\bea\lab{pr:bilproof} 
\norm{\CC(U,\pr (\pr B))}_{L^2(\MM)}&\les& (\norm{\pr^3B_0}_{L^2(\Si_0)}+\norm{\pr^2B_1}_{L^2(\Si_0)}+\norm{\pr F}_{L^2(\MM)})\\
\nn&&\times\left(\sup_{\HH}\norm{\CC(U, N)}_{L^2(\HH)}\right)
\eea
where the supremum is taken over all weakly regular null hypersurfaces, and where $N$ is the unit normal to $\HH\cap \Sigma_t$ inside $\Sigma_t$.  

We are now ready to prove the bilinear estimates of Lemma \ref{pr:lemmabil2}. Using the decomposition of Lemma \ref{recoverA'}, we have
\bea\lab{pr:aller}
&&\norm{\pa A^\ell \pa_\ell\prb B}_{L^2(\MM)}+\norm{\pr A^\ell \pa_\ell  A}_{L^2(\MM)}+\norm{Q_{ij} (A^\ell,A_\ell)}_{L^2(\MM)}\\
\nn&\les& \norm{Q_{ij} (\pr B, \prb B)}_{L^2(\MM)}+l.o.t.
\eea
which is of the type $\CC(U,\pr (\pr B))$ with $U=\pr\prb B$. Now, arguing as in section \ref{sec:proofbil1}, we have in this case 
$$\CC(U, N)=\in_{ij}\pr_i\prb BN_j=\nabb\prb B$$ and we deduce from \eqref{pr:bilproof} and \eqref{pr:aller}
\bea\lab{pr:aller1}
&&\norm{\pa A^\ell \pa_\ell\prb B}_{L^2(\MM)}+\norm{\pr A^\ell \pa_\ell  A)}_{L^2(\MM)}+\norm{Q_{ij} (A^\ell,A_\ell)}_{L^2(\MM)}\\
\nn&\les& \left(\sup_{\HH}\norm{\nabb\prb B}_{L^2(\HH)}\right)\left(\mathcal E+\mathcal S+\|F\|_{L^2(\mathcal M)}\right)+\ep^2\\
\nn&\les& \eps \left(\mathcal E+\mathcal S+\|F\|_{L^2(\mathcal M)}\right)+\ep^2.
\eea

Next, we consider the term $\nabla_0k^{ab}\nabla_a\nabla_b B$. Recall from \eqref{csd9} that we have
$$
\nabla_0k_{ab}= \R_{a0b0}+l.o.t.
$$
It follows that,
\bea\lab{pr:aller2}
\|\nabla_0k^{ab}\nabla_a\nabla_b B\|_{L^2(\mathcal M)} \les \|\R_{a0b0}\nabla_a\nabla_b B\|_{L^2(\mathcal M)} +l.o.t.
\eea
The right-hand side of \eqref{pr:aller2} is of the type $\CC(U,\pr (\pr B))$ with $U=\R_{a0b0}$. Now, arguing as in section \ref{sec:proofbil1}, we have in this case 
$$\CC(U, N)=\R_{N0\,\c\,\c}=\R\c L$$ 
and we deduce from \eqref{pr:bilproof} and \eqref{pr:aller}
\bea\lab{pr:aller3}
&&\norm{\pa A^\ell \pa_\ell\prb B}_{L^2(\MM)}+\norm{\pr A^\ell \pa_\ell  A)}_{L^2(\MM)}+\norm{Q_{ij} (A^\ell,A_\ell)}_{L^2(\MM)}\\
\nn&\les& \left(\sup_{\HH}\norm{\R\c L}_{L^2(\HH)}\right)\left(\mathcal E+\mathcal S+\|F\|_{L^2(\mathcal M)}\right)+\ep^2\\
\nn&\les& \ep\left(\mathcal E+\mathcal S+\|F\|_{L^2(\mathcal M)}\right)+\ep^2
\eea
where we used \eqref{pr:bootR} in the last inequality. Finally, \eqref{pr:aller1} and \eqref{pr:aller3} yield the desired estimate. This concludes the proof of Lemma \ref{pr:lemmabil2}.

\appendix

\section{Proof of Lemma \ref{lemma:divcurllemmainitialslice}}\lab{appendix:proofdivcurllemmainitialslice}

Let $X$ a vectorfield on $\Sigma_0$. We use the harmonic coordinate system on $\Si_t$ of Lemma \ref{lemma:estimatesit} with $t=0$. We have in a coordinate patch $U$:
\beaa
\div X &=& \frac{1}{\sqrt{|g|}}\dot{\pr}_i(\sqrt{|g|}g^{ij}X_j)\\
&=& \frac{1}{\sqrt{|g|}}\dot{\pr}_i(\sqrt{|g|}(g^{ij}-\delta^{ij})X_j+\delta^{ij}(\sqrt{|g|}-1)X_j)+\frac{1}{\sqrt{|g|}}\dot{\div}X 
\eeaa
where $\dot{\partial}$ and $\dot{\div}$ denote the derivatives and the flat divergence relative to the coordinate system defined above, 
as opposed to the frame derivatives $\pa$ and the divergence $\div$. This yields
\beaa
\dot{\div}X = \sqrt{|g|}\div X-\dot{\pr}_i\Big((\sqrt{|g|}(g^{ij}-\delta^{ij})+\delta^{ij}(\sqrt{|g|}-1))X_j\Big).  
\eeaa
Also, we have
$$\curl X=\dot{\curl}X.$$
Let $\varrho$ a smooth cut-off function localized in the coordinate patch $U$. Then, we have
\beaa
&&(\dot{\div}(\varrho X), \dot{\curl}(\varrho X))\\
 &=& \varrho\Big(\sqrt{|g|}\div X-\dot{\pr}_i\Big((\sqrt{|g|}(g^{ij}-\delta^{ij})+\delta^{ij}(\sqrt{|g|}-1))X_j\Big), \curl(X)\Big)\\
 &&+(\dot{\nabla}\varrho X, \dot{\nab}\varrho\dot{\wedge} X)\\
 &=& \Big(\varrho\sqrt{|g|}\div X-\dot{\pr}_i\Big((\sqrt{|g|}(g^{ij}-\delta^{ij})X_j+\delta^{ij}(\sqrt{|g|}-1))\varrho X_j\Big), \curl(X)\Big)\\
 &&+((\sqrt{|g|}(g^{ij}-\delta^{ij})X_j+\delta^{ij}(\sqrt{|g|}-1))\dot{\pr}_i\varrho\varrho X_j+\dot{\nabla}\varrho X, \dot{\nab}\varrho\dot{\wedge} X).
\eeaa
Let us denote by $\dot{\mathcal{D}}$ the div-curl system in coordinates, i.e.
$$\dot{\mathcal{D}}=(\dot{\div}(\varrho X), \dot{\curl}(\varrho X)).$$
We obtain
\beaa
\varrho X &=& \dot{\mathcal{D}}^{-1} \Big(\varrho\sqrt{|g|}\div X-\dot{\pr}_i\Big((\sqrt{|g|}(g^{ij}-\delta^{ij})X_j+\delta^{ij}(\sqrt{|g|}-1))\varrho X_j\Big), \curl(X)\Big)\\
 &&+\dot{\mathcal{D}}^{-1} ((\sqrt{|g|}(g^{ij}-\delta^{ij})X_j+\delta^{ij}(\sqrt{|g|}-1))\dot{\pr}_i\varrho\varrho X_j+\dot{\nabla}\varrho X, \dot{\nab}\varrho\dot{\wedge} X). 
\eeaa
Now, we use the following standard elliptic estimates on $\RRR^3$:
\beaa
\dot{\mathcal{D}}^{-1}\in \LL(L^{\frac{6}{5}}(\RRR^3), L^2(\RRR^3)),\,\, \dot{\mathcal{D}}^{-1}\dot{\partial}\in \LL(L^2(\RRR^3)),
\eeaa
where the notation $\mathcal{L}(X,Y)$ stands for the set of bounded linear operators from the space $X$ to the space $Y$. Together with our assumption on the harmonic coordinates \eqref{coorharmth1bis}, this yields in a coordinate Patch $U$
\beaa
&&\norm{\varrho X}_{L^2(U)} \\
&\les&  \norm{\div X}_{L^{\frac{6}{5}}(U)}+\norm{\curl X}_{L^{\frac{6}{5}}(U)}+\norm{(\sqrt{|g|}(g^{ij}-\delta^{ij})+\delta^{ij}(\sqrt{|g|}-1))\varrho X}_{L^2(U)}+\norm{\dot{\pr}\varrho X}_{L^{\frac{6}{5}}(U)}\\
&\les&  \norm{\div X}_{L^{\frac{6}{5}}(U)}+\norm{\curl X}_{L^{\frac{6}{5}}(U)}+\norm{g^{ij}-\delta^{ij}}_{L^\infty(U)}\norm{X}_{L^2(U)}+\norm{\dot{\pr}\varrho }_{L^{\frac{3}{2}}(U)}\norm{X}_{L^6(U)}\\
&\les&  \norm{\div X}_{L^{\frac{6}{5}}(U)}+\norm{\curl X}_{L^{\frac{6}{5}}(U)}+\delta\norm{X}_{L^2(U)}+C(\de)\norm{X}_{L^6(U)}\\
\eeaa
We then sum the contributions of the covering of $\Sigma_0$ by harmonic coordinate patches $U$ satisfying \eqref{coorharmth1bis}  together with a partition of unity $(\varrho_U)$ subordonate to the covering. Eventually increasing $C(\delta)$, we obtain
\beaa
\norm{X}_{L^2(\Si_0)} &\les&  \norm{\div X}_{L^{\frac{6}{5}}(\Si_0)}+\norm{\curl X}_{L^{\frac{6}{5}}(\Si_0)}+\delta\norm{X}_{L^2(\Si_0)}+C(\de)\norm{X}_{L^6(\Si_0)}.
\eeaa
Recall from Lemma \ref{lemma:estimatesit} that we have the freedom of choice for $\delta>0$. By choosing $\delta>0$ small enough, we finally obtain
\beaa
\norm{X}_{L^2(\Si_0)} &\les&  \norm{\div X}_{L^{\frac{6}{5}}(\Si_0)}+\norm{\curl X}_{L^{\frac{6}{5}}(\Si_0)}+\norm{X}_{L^6(\Si_0)}.
\eeaa
This concludes the proof of Lemma \ref{lemma:divcurllemmainitialslice}.

\section{Proof of \eqref{eel3}}

The goal of this appendix is to prove \eqref{eel3}. We first introduce 
 Littlewood-Paley projections on $\Si_t$.  
 These were constructed in \cite{param3} (see section 3.6 in that paper) using the heat flow on $\Si_t$. We recall below their main properties:
\begin{proposition}[Main properties of the LP $Q_j$ \cite{param3}]\label{prop:LP}
 Let $F$ a tensor on $\Sigma_t$. The LP-projections $Q_j$ on $\Si_t$ verify the following properties:

i) \quad {\sl Partition of unity}
\bea\lab{eq:partition}
\sum_jQ_j=I.
\eea

ii) \quad {\sl $L^p$-boundedness} \quad For any $1\le
p\le \infty$, and any interval $I\subset \Bbb Z$,
\bea\lab{eq:pdf1}
\|Q_IF\|_{L^p(\Si_t)}\lesssim \|F\|_{L^p(\Si_t)}
\eea

iii)\quad {\sl Finite band property}\quad For any $1\le p\le \infty$.
\begin{equation}
\begin{array}{lll}
\|\lap Q_j F\|_{L^p(\Si_t)}&\lesssim & 2^{2j} \|F\|_{L^p(\Si_t)}\\
\|Q_jF\|_{L^p(\Si_t)} &\lesssim & 2^{-2j} \|\lap F \|_{L^p(\Si_t)}.
\end{array}
\end{equation}

In addition, the $L^2$ estimates
\begin{equation}
\begin{array}{lll}
\|\nab Q_j F\|_{L^2(\Si_t)}&\lesssim & 2^{j} \|F\|_{L^2(\Si_t)}\\
\|Q_jF\|_{L^2(\Si_t)} &\lesssim & 2^{-j} \|\nab F  \|_{L^2(\Si_t)}
\end{array}
\end{equation}
hold together with the dual estimate
$$\| Q_j \nab F\|_{L^2(\Si_t)}\lesssim 2^j \|F\|_{L^2(\Si_t)}$$

iv) \quad{\sl Bernstein inequality}\quad For any $2\le p\leq +\infty$ and $j\in\mathbb{Z}$
$$\|Q_j F\|_{L^p(\Si_t)}\lesssim 2^{\frac{3}{2}(1-\frac 2p)j} \|F\|_{L^2(\Si_t)}$$
together with the dual estimates 
$$\|Q_j F\|_{L^2(\Si_t)}\lesssim 2^{\frac{3}{2}(1-\frac 2p)j} \|F\|_{L^{p'}(\Si_t)}$$
\end{proposition}

We now rely on Proposition \ref{prop:LP} to prove \eqref{eel3}. Using  Proposition \ref{prop:LP}, we have for any scalar function $v$ on $\Si_t$:

\beaa
\norm{(-\Delta)^{-1}v}_{\lt{\infty}}&\les& \sum_{j\in\mathbb{Z}}\norm{Q_j(-\Delta)^{-1}v}_{\lt{\infty}}\\
&\les& \sum_{j\in\mathbb{Z}}2^{\frac{3j}{2}}\norm{Q_j(-\Delta)^{-1}f}_{\lt{2}}\\
&\les& \sum_{j\in\mathbb{Z}}2^{-\frac{j}{2}}\norm{Q_jf}_{\lt{2}}\\
&\les& \left(\sum_{j\geq 0}2^{-\frac{j}{14}}\right)\norm{f}_{\lt{\frac{14}{9}}}+\left(\sum_{j< 0}2^{\frac{j}{13}}\right)\norm{f}_{\lt{\frac{13}{9}}}\\
&\les& \norm{f}_{\lt{\frac{14}{9}}}+\norm{f}_{\lt{\frac{13}{9}}}.
\eeaa
This concludes the proof of \eqref{eel3}.

\section{Proof of Lemma \ref{lemma:maru}}\lab{sec:prooflemmamaru}

Recall that $3<p<+\infty$ and $v$ is the solution of
$$\Delta v=f.$$
In view of Lemma \ref{lemma:moser}, we have
\bea\lab{aprioriboundLptostart}
\norm{v}_{L^p(\Si_t)}\les \norm{f}_{L^{\frac{3p}{2p+3}}(\Si_t)}.
\eea

Next, we need to derive an estimate for $\pr v$. We proceed as in the proof of Lemma \ref{lemma:divcurllemmainitialslice}. Using the harmonic coordinate system on $\Si_t$ of Lemma \ref{lemma:estimatesit}, we have
\beaa
\Delta v &=& \frac{1}{\sqrt{|g|}}\dot{\pr}_i(\sqrt{|g|}g^{ij}\dot{\pr}_jv)\\
&=& \dot{\De}v+\frac{1}{\sqrt{|g|}}\dot{\pr}_i\Big(\sqrt{|g|}(g^{ij}-\de^{ij})\dot{\pr}_jv+(\sqrt{|g|}-1)\de^{ij}\dot{\pr}_jv\Big)
\eeaa
This yields
\beaa
 \dot{\De}v &=& f -\frac{1}{\sqrt{|g|}}\dot{\pr}_i\Big(\sqrt{|g|}(g^{ij}-\de^{ij})\dot{\pr}_jv+(\sqrt{|g|}-1)\de^{ij}\dot{\pr}_jv\Big).
 \eeaa
Let $\varrho$ a smooth cut-off function localized in the coordinate patch $U$. Then, we have
\bea\lab{wed5}
 &&\dot{\De}(\varrho v)\\
 \nn&=& \varrho f -\frac{\varrho}{\sqrt{|g|}}\dot{\pr}_i\Big(\sqrt{|g|}(g^{ij}-\de^{ij})\dot{\pr}_jv+(\sqrt{|g|}-1)\de^{ij}\dot{\pr}_jv\Big) -2\dot{\nab}\varrho\c\dot{\nab}v-\dot{\De}\varrho v\\
 \nn&=& \varrho f -\dot{\pr}_i\left(\varrho(g^{ij}-\de^{ij})\dot{\pr}_jv+\varrho\left(1-\frac{1}{\sqrt{|g|}}\right)\de^{ij}\dot{\pr}_jv\right) -2\dot{\nab}\varrho\c\dot{\nab}v-\dot{\De}\varrho v\\
 \nn&=& \varrho f -\dot{\pr}_i\left(\varrho(g^{ij}-\de^{ij})\dot{\pr}_jv+\varrho\left(1-\frac{1}{\sqrt{|g|}}\right)\de^{ij}\dot{\pr}_jv\right) -2\dot{\div}(v\dot{\nab}\varrho)+\dot{\De}\varrho v.
 \eea
Now, we use the following standard elliptic estimates on $\RRR^3$ for any $3<p<+\infty$:
\beaa
(-\dot{\Delta})^{-1}\dot{\partial}_i\in \LL(L^{\frac{3p}{2p+3}}(\RRR^3), L^{\frac{3p}{p+3}}(\RRR^3)),\,\, (-\dot{\Delta})^{-1}\dot{\partial}^2_{ij}\in \LL(L^{\frac{3p}{p+3}}(\RRR^3)).
\eeaa
Together with \eqref{wed5} and our assumptions on the harmonic coordinates \eqref{coorharmth1bis} \eqref{coorharmth2bis}, this yields in a coordinate Patch $U$:
\beaa
&&\norm{\pr(\varrho v)}_{L^\frac{3p}{p+3}(U)}\\
&\les& \norm{f}_{L^\frac{3p}{2p+3}(U)}+\norm{(g^{ij}-\de^{ij})\dot{\pr}_jv}_{L^\frac{3p}{p+3}(U)}+\norm{(\sqrt{|g|}-1)\dot{\pr}_jv}_{L^\frac{3p}{p+3}(U)}+\norm{\dot{\nab}\rho v}_{L^\frac{3p}{p+3}(U)}\\
&&+ \norm{\dot{\De}\varrho v}_{L^\frac{3p}{2p+3}(U)}\\
&\les& \norm{f}_{L^\frac{3p}{2p+3}(U)}+\norm{g^{ij}-\de^{ij}}_{L^\infty(U)}\norm{\dot{\pr}_jv}_{L^\frac{3p}{p+3}(U)}+(\norm{\dot{\nab}\varrho}_{L^3(U)}+\norm{\dot{\De}\varrho}_{L^{\frac{3}{2}}(U)})\norm{v}_{L^p(U)}\\
&\les& \norm{f}_{L^\frac{3p}{2p+3}(U)}+\de\norm{\pr v}_{L^\frac{3p}{p+3}(U)}+C(\de)\norm{v}_{L^p(U)}.
\eeaa
We then sum the contributions of the covering of $\Sigma_t$ by harmonic coordinate patches $U$ satisfying \eqref{coorharmth1bis}  together with a partition of unity $(\varrho_U)$ subordonate to the covering. Eventually increasing $C(\delta)$, we obtain
\beaa
\norm{\pr v}_{L^\frac{3p}{p+3}(\Sigma_t)}&\les& \norm{f}_{L^\frac{3p}{2p+3}(\Sigma_t)}+\de\norm{\pr v}_{L^\frac{3p}{p+3}(\Sigma_t)}+C(\de)\norm{v}_{L^p(\Sigma_t)}.
\eeaa
Recall from Lemma \ref{lemma:estimatesit} that we have the freedom of choice for $\delta>0$. By choosing $\delta>0$ small enough, we finally obtain
\beaa
\norm{\pr v}_{L^\frac{3p}{p+3}(\Sigma_t)}&\les& \norm{f}_{L^\frac{3p}{2p+3}(\Sigma_t)}+\norm{v}_{L^p(\Sigma_t)}
\eeaa
which together with \eqref{aprioriboundLptostart} yields
\beaa
\norm{v}_{L^p(\Sigma_t)}+\norm{\pr v}_{L^\frac{3p}{p+3}(\Sigma_t)}&\les& \norm{f}_{L^\frac{3p}{2p+3}(\Sigma_t)}.
\eeaa
This concludes the proof of Lemma \ref{lemma:maru}.

\section{Proof of Lemma \ref{lemma:marubis}}\lab{sec:prooflemmamarubis}

Recall that $3<p<+\infty$  and $v$ is the solution of
$$\Delta v=\pr f.$$
In view of Lemma \ref{lemma:maru}, we have
$$\pr(-\De)^{-1}\in \LL(L^{\frac{3q}{2q+3}}(\Si_t), L^{\frac{3q}{q+3}}(\Si_t))\textrm{ for any }3<q<+\infty.$$
Taking the dual, we infer
$$(-\De)^{-1}\pr\in \LL(L^{\frac{3q}{2q-3}}(\Si_t), L^{\frac{3q}{q-3}}(\Si_t))\textrm{ for any }3<q<+\infty.$$
In particular, choosing
$$q=\frac{3p}{p-3}\in (3,+\infty)$$
we obtain
$$(-\De^{-1})\pr\in \LL(L^{\frac{3p}{p+3}}(\Si_t), L^p(\Si_t))\textrm{ for any }3<q<+\infty.$$
Since
$$v=-(-\De)^{-1}\pr f,$$
we deduce
\bea\lab{aprioriboundLptostartbis}
\norm{v}_{L^p(\Si_t)}\les \norm{f}_{L^{\frac{3p}{p+3}}(\Si_t)}.
\eea

Next, we need to derive an estimate for $\pr v$. We proceed as in the proof of Lemma \ref{lemma:maru}. We have the harmonic coordinate system on $\Si_t$ of Lemma \ref{lemma:estimatesit}
\beaa
 \dot{\De}v &=& \pr f -\frac{1}{\sqrt{|g|}}\dot{\pr}_i\Big(\sqrt{|g|}(g^{ij}-\de^{ij})\dot{\pr}_jv+(\sqrt{|g|}-1)\de^{ij}\dot{\pr}_jv\Big).
 \eeaa
Let $\varrho$ a smooth cut-off function localized in the coordinate patch $U$. Then, we have
\bea\lab{wed5bis}
 &&\dot{\De}(\varrho v)\\
 \nn&=& \varrho \pr f -\frac{\varrho}{\sqrt{|g|}}\dot{\pr}_i\Big(\sqrt{|g|}(g^{ij}-\de^{ij})\dot{\pr}_jv+(\sqrt{|g|}-1)\de^{ij}\dot{\pr}_jv\Big) -2\dot{\nab}\varrho\c\dot{\nab}v-\dot{\De}\varrho v\\
 \nn&=& \pr(\varrho f)-\pr\varrho f -\dot{\pr}_i\left(\varrho(g^{ij}-\de^{ij})\dot{\pr}_jv+\varrho\left(1-\frac{1}{\sqrt{|g|}}\right)\de^{ij}\dot{\pr}_jv\right) -2\dot{\div}(v\dot{\nab}\varrho)+\dot{\De}\varrho v.
 \eea
Now, we use the following standard elliptic estimates on $\RRR^3$ for any $3<p<+\infty$:
\beaa
(-\dot{\Delta})^{-1}\dot{\partial}_i\in \LL(L^{\frac{3p}{2p+3}}(\RRR^3), L^{\frac{3p}{p+3}}(\RRR^3)),\,\, (-\dot{\Delta})^{-1}\dot{\partial}^2_{ij}\in \LL(L^{\frac{3p}{p+3}}(\RRR^3)).
\eeaa
Together with \eqref{wed5bis} and our assumptions on the harmonic coordinates \eqref{coorharmth1bis} \eqref{coorharmth2bis}, this yields in a coordinate Patch $U$:
\beaa
&&\norm{\pr(\varrho v)}_{L^\frac{3p}{p+3}(U)}\\
&\les& \norm{\varrho f}_{L^\frac{3p}{p+3}(U)}+\norm{\pr\varrho f}_{L^\frac{3p}{2p+3}(U)}+\norm{(g^{ij}-\de^{ij})\dot{\pr}_jv}_{L^\frac{3p}{p+3}(U)}+\norm{(\sqrt{|g|}-1)\dot{\pr}_jv}_{L^\frac{3p}{p+3}(U)}\\
&&+\norm{\dot{\nab}\rho v}_{L^\frac{3p}{p+3}(U)}+ \norm{\dot{\De}\varrho v}_{L^\frac{3p}{2p+3}(U)}\\
&\les& \norm{f}_{L^\frac{3p}{p+3}(U)}+\norm{g^{ij}-\de^{ij}}_{L^\infty(U)}\norm{\dot{\pr}_jv}_{L^\frac{3p}{p+3}(U)}+(\norm{\dot{\nab}\varrho}_{L^3(U)}+\norm{\dot{\De}\varrho}_{L^{\frac{3}{2}}(U)})\norm{v}_{L^p(U)}\\
&\les& \norm{f}_{L^\frac{3p}{p+3}(U)}+\de\norm{\pr v}_{L^\frac{3p}{p+3}(U)}+C(\de)\norm{v}_{L^p(U)}.
\eeaa
We then sum the contributions of the covering of $\Sigma_t$ by harmonic coordinate patches $U$ satisfying \eqref{coorharmth1bis}  together with a partition of unity $(\varrho_U)$ subordonate to the covering. Eventually increasing $C(\delta)$, we obtain
\beaa
\norm{\pr v}_{L^\frac{3p}{p+3}(\Sigma_t)}&\les& \norm{f}_{L^\frac{3p}{p+3}(\Sigma_t)}+\de\norm{\pr v}_{L^\frac{3p}{p+3}(\Sigma_t)}+C(\de)\norm{v}_{L^p(\Sigma_t)}.
\eeaa
Recall from Lemma \ref{lemma:estimatesit} that we have the freedom of choice for $\delta>0$. By choosing $\delta>0$ small enough, we finally obtain
\beaa
\norm{\pr v}_{L^\frac{3p}{p+3}(\Sigma_t)}&\les& \norm{f}_{L^\frac{3p}{p+3}(\Sigma_t)}+\norm{v}_{L^p(\Sigma_t)}
\eeaa
which together with \eqref{aprioriboundLptostartbis} yields
\beaa
\norm{v}_{L^p(\Sigma_t)}+\norm{\pr v}_{L^\frac{3p}{p+3}(\Sigma_t)}&\les& \norm{f}_{L^\frac{p}{p+3}(\Sigma_t)}.
\eeaa
This concludes the proof of Lemma \ref{lemma:marubis}.

\section{Proof of Lemma \ref{lemma:commutation}}

The goal of this appendix is to prove Lemma \ref{lemma:commutation}. The commutation formula \eqref{commsquare2} has already been proved at the beginning of section \ref{sec:bobo5}. Thus, it only remains to prove the commutation formula \eqref{commsquare3}.  Recalling \eqref{eq:YM16},
\beaa
\square\phi=-\pr_0(\pr_0\phi)+\Delta\phi+n^{-1}\nabla n\cdot\nabla\phi, 
\eeaa
 Thus, we have:
\bea\label{commutsquaredelta1}
[\square,\Delta]\phi&=&[-\pr_0\pr_0+n^{-1}\nabla n\c\nabla+\Delta,\Delta]\phi\\
\nn&=& -[\pr_0\pr_0,\Delta]\phi+[n^{-1}\nabla n\c\nabla,\Delta]\phi.
\eea
We thus have to calculate the commutators $[\pr^2_0,\Delta]\phi$ and 
$[n^{-1}\nabla n\c\nabla,\Delta]\phi$. For any tensor $U$ tangent to $\Si_t$, we denote by $\nabla_0U$ the projection of $\D_0U$ to $\Si_t$. We have the following commutator formula 
for any vectorfield $U$ tangent to $\Si_t$:
\bea\lab{commutnablapr0:t}
[\nabla_b, \nabla_0]U_a=k_{bc}\nabla_cU_a-n^{-1}\nabla_bn \nabla_0U_a+(n^{-1}k_{ab}\nabla_cn-n^{-1}k_{bc}\nabla_an+\R_{0abc})U_c,
\eea 
while for a scalar $\phi$, the commutator formula reduces to:
\bea\lab{commutnablapr0:s}
[\nabla_b, \nabla_0]\phi=k_{bc}\nabla_c\phi-n^{-1}\nabla_bn \pr_0\phi.
\eea 
Using the commutator formulas \eqref{commutnablapr0:t} and \eqref{commutnablapr0:s} and the fact that $[\pr_0,\Delta]\phi=[\nabla_0,\nabla^a]\nabla_a\phi+\nabla^a[\nabla_0,\nabla_a]\phi$, we obtain:
\bea\lab{commutdeltapr0}
[\pr_0,\Delta]\phi=-2k^{ab}\nabla_a\nabla_b\phi+2n^{-1}\nabla_bn\nabla_b(\pr_0\phi)+n^{-1}\Delta n\pr_0\phi-2n^{-1}\nabla_an k^{ab}\nabla_b\phi,
\eea
where we used the constraint equation \eqref{constk} and the fact that, in view
of the Einstein equations and the symmetries of $\R$, we have:
$$\g^{ab}\R_{0abc}=0.$$
Differentiating the commutator formula \eqref{commutdeltapr0} with respect to $\pr_0$ and using the commutator formulas \eqref{commutnablapr0:t} and \eqref{commutnablapr0:s}, we obtain:
\beaa
&&\pr_0([\pr_0,\Delta]\phi)\\
&=&-2k^{ab}\nabla_a\nabla_b(\pr_0\phi)+2n^{-1}\nabla_bn\nabla_b(\pr_0(\pr_0\phi))+(-2\nabla_0k^{ab}+4k^{ac}k_c\,^b)\nabla_a\nabla_b\phi\\
&&+(2n^{-1}\nabla_b(\pr_0n)-10k^{ab}n^{-1}\nabla_an)\nabla_b(\pr_0\phi)+(n^{-1}\Delta n+2n^{-2}|\nabla n|^2)\pr_0(\pr_0\phi)\\
&&+(2k^{ac}\R_{0acb}+2k^{ac}\nabla_ck_{ab}-2n^{-1}\nabla_an\nabla_0k^{ab}+2k^{ab}n^{-1}\nabla_a(\pr_0n)+4k^{ac}k_{cb}n^{-1}\nabla_an\\
&&+2|k|^2n^{-1}\nabla_bn-2k^{ab}n^{-2}\nabla_an\pr_0n)\nabla_b\phi\\
&&+(n^{-1}\Delta(\pr_0n)-4k^{ab}n^{-1}\nabla_a\nabla_bn+2n^{-2}\nabla_bn\nabla_b(\pr_0n))\pr_0\phi.
\eeaa
Together with the commutator formula \eqref{commutdeltapr0} applied to $\pr_0\phi$, we obtain:
\bea\label{commutdeltapr00}
&&[\pr_0\pr_0,\Delta]\phi\\
\nn&=&[\pr_0,\Delta]\pr_0\phi+\pr_0([\pr_0,\Delta]\phi)\\
\nn&=&-4k^{ab}\nabla_a\nabla_b(\pr_0\phi)+4n^{-1}\nabla_bn\nabla_b(\pr_0(\pr_0\phi))+(-2\nabla_0k^{ab}+4k^{ac}k_c\,^b)\nabla_a\nabla_b\phi\\
\nn&&+(2n^{-1}\nabla_b(\pr_0n)-12k^{ab}n^{-1}\nabla_an)\nabla_b(\pr_0\phi)+(2n^{-1}\Delta n+2n^{-2}|\nabla n|^2)\pr_0(\pr_0\phi)\\
\nn&&+(2k^{ac}\R_{0acb}+2k^{ac}\nabla_ck_{ab}-2n^{-1}\nabla_an\nabla_0k^{ab}+2k^{ab}n^{-1}\nabla_a(\pr_0n)+4k^{ac}k_{cb}n^{-1}\nabla_an\\
\nn&&+2|k|^2n^{-1}\nabla_bn-2k^{ab}n^{-2}\nabla_an\pr_0n)\nabla_b\phi\\
\nn&&+(n^{-1}\Delta(\pr_0n)-4k^{ab}n^{-1}\nabla_a\nabla_bn+2n^{-2}\nabla_bn\nabla_b(\pr_0n))\pr_0\phi.
\eea

We also compute the commutator $[n^{-1}\nabla n\nabla,\Delta]\phi$:
\beaa
&&[n^{-1}\nabla n\nabla,\Delta]\phi\\
&=& -\Delta(n^{-1}\nabla_bn)\nabla_b\phi-\nabla_a(n^{-1}\nabla_bn)\nabla_a\nabla_b\phi+n^{-1}\nabla_bn[\nabla_b,\Delta]\phi\\
&=& -n^{-1}\nabla_b(\Delta n)\nabla_b\phi-n^{-1}[\Delta,\nabla_b]n\nabla_b\phi+n^{-2}\nabla_an\nabla_a\nabla_bn\nabla_b\phi\\
&&+n^{-2}\nabla_bn\nabla_an\nabla_a\nabla_b\phi-n^{-1}\nabla_a\nabla_bn\nabla_a\nabla_b\phi+n^{-1}\nabla_bn[\nabla_b,\Delta]\phi.
\eeaa
Now, we have the following commutator formula:
\bea\label{commutdeltanabla}
[\nabla_b,\Delta]\phi=R_{b}\,^{c} \nab_c\phi=  (\R_{b00}\,^{c}+k_{bd}k^{dc} )\nabla_c\phi,
\eea
where we used the Gauss equation for $R$, the Einstein equations for $\R$ and the maximal foliation assumption. Thus, we obtain:
\bea\label{commutdeltanablannabla}
&&[n^{-1}\nabla n\nabla,\Delta]\phi\\
\nn &=& (-n^{-1}\nabla_a\nabla_bn+n^{-2}\nabla_bn\nabla_an)\nabla_a\nabla_b\phi+(-n^{-1}\nabla_b(\Delta n)+n^{-2}\nabla_an\nabla_a\nabla_bn\\
\nn&&+2(\R_{b00a}+k_{ba}k_a\,^c)n^{-1}\nabla_an)\nabla_b\phi.
\eea
Finally, \eqref{commutsquaredelta1}, \eqref{commutdeltapr00} and \eqref{commutdeltanablannabla} yield:
\beaa
&&[\pr_0\pr_0,\Delta]\phi\\
\nn&=&[\pr_0,\Delta]\pr_0\phi+\pr_0([\pr_0,\Delta]\phi)\\
\nn&=&-4k^{ab}\nabla_a\nabla_b(\pr_0\phi)+4n^{-1}\nabla_bn\nabla_b(\pr_0(\pr_0\phi))\\
\nn&&+(-2\nabla_0k^{ab}+4k^{ac}k_c\,^b-n^{-1}\nabla_a\nabla_bn+n^{-2}\nabla_bn\nabla_an)\nabla_a\nabla_b\phi\\
\nn&&+(2n^{-1}\nabla_b(\pr_0n)-12k^{ab}n^{-1}\nabla_an)\nabla_b(\pr_0\phi)+(2n^{-1}\Delta n+2n^{-2}|\nabla n|^2)\pr_0(\pr_0\phi)\\
\nn&&+(2k^{ac}\R_{0acb}+2k^{ac}\nabla_ck_{ab}-2n^{-1}\nabla_an\nabla_0k^{ab}+2k^{ab}n^{-1}\nabla_a(\pr_0n)+4k^{ac}k_{cb}n^{-1}\nabla_an\\
\nn&&+2|k|^2n^{-1}\nabla_bn-2k^{ab}n^{-2}\nabla_an\pr_0n-n^{-1}\nabla_b(\Delta n)+n^{-2}\nabla_an\nabla_a\nabla_bn\\
&&+2(\R_{b00a}+k_{ba}k_a\,^c)n^{-1}\nabla_an)\nabla_b\phi\\
&&+(n^{-1}\Delta(\pr_0n)-4k^{ab}n^{-1}\nabla_a\nabla_bn+2n^{-2}\nabla_bn\nabla_b(\pr_0n))\pr_0\phi,
\eeaa
from which \eqref{commsquare3} easily follows. This concludes the proof of Lemma \ref{lemma:commutation}.

\end{document}